\pdfoutput=1
\documentclass[a4paper,11pt]{amsart}
\usepackage[toc,page]{appendix} 
\usepackage[french,english]{babel} 
\usepackage[utf8]{inputenc} 
\usepackage[T1]{fontenc}     
\usepackage{graphicx}
\usepackage{amsmath}
\usepackage{enumerate}
\usepackage{mathrsfs}
\usepackage{multirow}
\usepackage{amsthm}
\usepackage{amsfonts}
\usepackage[all]{xy}
\usepackage{yhmath}
\usepackage[hidelinks]{hyperref}
\usepackage{amssymb}
\usepackage{color}
\definecolor{Green}{RGB}{0,200,0}
\usepackage{geometry}
\usepackage{import}

\newcommand{\rr}{\mathbb{R}}
\newcommand{\nn}{\mathbb{N}}
\newcommand{\cc}{\mathcal{C}}
\newcommand{\hh}{\mathcal{H}}
\newcommand{\qq}{\mathcal{Q}}
\newcommand{\zz}{\mathcal{Z}}

\newcommand{\SSS}{\mathbb{S}}
\newcommand{\WW}{\mathcal{W}}
\newcommand{\LL}{\mathcal{L}}
\newcommand{\aaa}{\mathcal{A}}
\newcommand{\CC}{\mathscr{C}}

\newcommand*{\comm}{\mathbin{\raisebox{1ex}{\rotatebox{270}{$\circlearrowright$}}}}

\newcommand{\dd}{\partial}
\newcommand{\ep}{\varepsilon}

\newtheorem{theo}{Theorem}[section]
\newtheorem{prop}[theo]{Proposition}
\newtheorem{add}[theo]{Addendum}
\newtheorem{lem}[theo]{Lemma}
\theoremstyle{definition} 
\newtheorem{defi}[theo]{Definition}
\theoremstyle{remark} 
\newtheorem{hyp1}[theo]{Hypothesis}
\newtheorem{hyp}[theo]{Hypotheses}
\newtheorem{rem}[theo]{Remark}

\newtheorem{csq}[theo]{Consequence}
\newtheorem{csqs}[theo]{Consequences}
\theoremstyle{remark}

\theoremstyle{definition} 
\newtheorem{nota}[theo]{Notation}
\theoremstyle{remark}

  

\title[Variational and viscosity operators for the evolutive HJ equation]{Variational and viscosity operators for the evolutive Hamilton-Jacobi equation}
\author{Valentine Roos}\thanks{The research leading to these results has received funding from the European Research Council under
		the European Union’s Seventh Framework Programme (FP/2007-2013) / ERC Grant Agreement 307062 and from the French National Research Agency via ANR-12-BLAN-WKBHJ}
\begin{document}
	\begin{abstract}
		We study the Cauchy problem for the first order evolutive Hamilton-Jacobi equation with a Lipschitz initial condition. The Hamiltonian is not necessarily convex in the momentum variable and not a priori compactly supported. We build and study an operator giving a variational solution of this problem, and get local Lipschitz estimates on this operator. Iterating this variational operator we obtain the viscosity operator and extend the estimates to the viscosity framework. We also check that the construction of the variational operator gives the Lax-Oleinik semigroup if the Hamiltonian is convex or concave in the momentum variable.\end{abstract}
	\maketitle	
	\tableofcontents
\section{Introduction}
We study the Cauchy problem associated with the evolutive Hamilton-Jacobi equation
\begin{equation}\tag{HJ}\label{HJ}
\dd_t u (t,q) + H(t,q,\dd_q u (t,q))=0,
\end{equation}
where $H : \rr \times T^\star \rr^d \to \rr $ is a $\cc^2$ Hamiltonian,
$u : \rr\times \rr^d \to \rr$ is the unknown function, and the initial condition is given by $u(0,\cdot)=u_0$ Lipschitz.

The Cauchy problem does not admit classical solutions in large time even for smooth $u_0$ and $H$. Two different types of generalized solutions, namely viscosity and variational solutions, were then defined, respectively by P.-L. Lions and M.G. Crandall (see \cite{cr&lions83}) and by J.-C. Sikorav and M. Chaperon (see \cite{sikoexp}, \cite{chap2}). T. Joukovskaia showed that the two solutions match for compactly supported fiberwise convex Hamiltonians (see \cite{jou}), which is not necessarily true in the non convex case. Examples where the solutions differ were proposed in \cite{chenciner}, \cite{viterboX}, \cite{BC} and \cite{wei}. 

In order to compare these solutions in the framework of Lipschitz initial conditions, we will define two notions of operators, denoted respectively by $V^t_s$ and $R^t_s$, defined on the space of Lipschitz functions on $\rr^d$ and giving respectively viscosity and variational solutions of the Cauchy problem. 
In \cite{wei}, Q. Wei obtains for compactly supported Hamiltonians the viscosity operator via a limiting process on the iterated variational operator, which has a simple expression when the Hamiltonian does not depend on $t$:
\[\left(R^{\frac{t}{n}}\right)^n \underset{n\to \infty}\to V^t.\]
This result is extended to the contact framework in \cite{cast}. 

The assumption of compactness on the Hamiltonian's support is due to the symplectic origin of the variational solution. This paper is aimed at removing this constraint, by replacing it with the following set of assumptions on the Hamiltonian that is more standard in the framework of viscosity solution theory and thus more natural when one is interested in making the link between variational and viscosity solutions.
\begin{hyp1}\label{esti} There is a $C>0$ such that for each $(t,q,p)$ in $\rr\times \rr^d \times \rr^d$,\[
	\|\dd^2_{(q,p)} H(t,q,p)\| < C,\; \|\dd_{(q,p)} H(t,q,p)\| < C(1+\|p\|),\; | H(t,q,p)| < C(1+\|p\|)^2,
	\]where $\dd_{(q,p)} H$ and $\dd^2_{(q,p)} H$ denote the first and second order spatial derivatives of $H$. \end{hyp1}
This hypothesis implies a finite propagation speed principle in both viscosity and variational contexts, which allows to work with non compactly supported Hamiltonians. We refer for example to \cite{barles} for the viscosity side, where in particular the uniqueness of the viscosity operator (see also Proposition \ref{viscuniq}) is studied, and to Appendix B of \cite{carvit} for the existence of variational solutions for Hamiltonians satisfying this finite propagation speed principle. 

In this paper we propose under Hypothesis \ref{esti} a complete and elementary construction of both variational and viscosity operators, and extend Joukovskaia's and Wei's results to this framework. The proof of the convergence of the iterated variational operator relies on the computation of local Lipschitz estimates for the variational operator $R^t_s$ that obey a semigroup-type property with respect to $s$ and $t$. The same Lipschitz estimates then also apply to the viscosity operator via the limiting process.

 A special care was provided in an attempt to produce a self-contained text, accessible to a reader with no specific background on symplectic geometry.
\subsection{Classical solutions: the method of characteristics}
Under Hypothesis \ref{esti}, the \emph{Hamiltonian system}
	\begin{equation}\label{HS}\tag{HS}\left\{\begin{array}{l}
	\dot{q}(t)= \dd_p H(t,q(t),p(t)),\\
	\dot{p}(t)= - \dd_q H(t,q(t),p(t))
	\end{array}\right.
	\end{equation}
	admits a complete \emph{Hamiltonian flow} $\phi^t_s$, meaning that $t\mapsto \phi^t_s(q,p)$ is the unique solution of \eqref{HS} with initial conditions $\left(q(s),p(s)\right)=(q,p)$. We denote by $\left(Q^t_s,P^t_s\right)$ the coordinates of $\phi^t_s$. 
	We call a function $t \mapsto (q(t),p(t))$ solving the Hamiltonian system \eqref{HS} a \emph{Hamiltonian trajectory}.
The \emph{Hamiltonian action} of a $\cc^1$ path $\gamma(t)=\left(q(t),p(t)\right)\in T^\star \rr^d $ is denoted by \[
	\aaa^t_s (\gamma)=\int^t_s p(\tau)\cdot \dot{q}(\tau)-H(\tau,q(\tau),p(\tau))d\tau.
	\]

The next lemma states the existence of \emph{characteristics} for $\cc^2$ solutions of the Hamilton-Jacobi equation \eqref{HJ}. 
\begin{lem}\label{class}
	If $u$ is a $\cc^2$ solution of \eqref{HJ} on $[T_-,T_+]\times\rr^d$ and $\gamma:\tau\mapsto(q(\tau),p(\tau))$ is a Hamiltonian trajectory satisfying $p(s)=\dd_q u(s,q(s))$ for some $s\in[T_-,T_+]$, then $p(t)=\dd_q u(t,q(t))$ for each $t\in[T_-,T_+]$ and \[u(t,q(t))=u(s,q(s))+\aaa^t_s(\gamma) \, \; \forall t \in [T_-,T_+].\]
\end{lem}

\begin{proof}
	If $f(t)$ denotes the quantity $\dd_q u(t,q(t))$, one can show that both $f$ and $p$ solve the ODE $\dot{y}(t)=-\dd_q H(t,q(t),y(t))$ and $p(s)=f(s)$ implies that $p(t)=f(t)$ for each time $t\in[T_-,T_+]$.
	Then, differentiating the function $t \mapsto u(t,q(t))$ gives the result.
\end{proof}

Here is another formulation of the first statement of Lemma \ref{class}: if $\Gamma_s$ denotes the graph of $\dd_q u_s$, then the set $\phi^t_s \Gamma_s$ is included in the graph of $\dd_q u_t$ for each $T_-\leq s \leq t\leq T_+$. In particular, if $\phi^T_s\Gamma_s$ is not a graph for some time $T>s$, then the existence of classical solution on $[s,T]\times \rr^d$ is not possible, hence the introduction of generalized solutions.
	
\subsection{Viscosity operator} A family of operators $\left(V^t_s\right)_{s \leq t}$ mapping $\cc^{0,1}(\rr^d)$ (the space of Lipschitz functions) into itself is called a \textit{viscosity operator} if it satisfies the following conditions:
\begin{hyp}[Viscosity operator]\label{L-O}\begin{enumerate}[(1)]$ $
		\item \label{monot} Monotonicity: if $u\leq v$ are Lipschitz on $\rr^d$, then $V^t_s u \leq V^t_s v$ on $\rr^d$ for each $s \leq t$,
		\item \label{addi} Additivity: if $u$ is Lipschitz on $\rr^d$ and $c\in\rr$, then $V^t_s (c+u) =c + V^t_s u$, 
		\item \label{reg} Regularity: if $u$ is Lipschitz, then for each $\tau \leq T$, $\left\lbrace V^t_\tau u, t\in [\tau,T] \right\rbrace$ is Lipschitz in $q$  uniformly  w.r.t. $t$  and $(t,q) \mapsto V^t_\tau u(q)$ is locally Lipschitz on $(\tau,\infty)\times \rr^d$,
		\item \label{comp} Compatibility with Hamilton-Jacobi equation: if $u$ is a Lipschitz $\cc^2$ solution of the Hamilton-Jacobi equation, then $V^t_s u_s = u_t$ for each $s\leq t$,
		\item \label{mark} Markov property:  $V^t_s=V^t_\tau \circ V^\tau_s$ for all $s\leq \tau \leq t$.\end{enumerate}\end{hyp}
In Appendix \ref{C}, we state that such an operator gives the viscosity solutions as introduced by P.-L. Lions and M. G. Crandall in 1981 (see \cite{cr&lions83}), and we prove the uniqueness of such an operator when $H$ satisfies Hypothesis \ref{esti} (Consequence \ref{viscuniq}).

The existence of the viscosity operator for our framework was already granted by the work of Crandall, Lions and Ishii (see \cite{guide}) and it is proved again in this paper, where we obtain a viscosity operator by a limiting process.

\subsection{Variational operator}\label{varchapter}
A family of operators $\left(R^t_s\right)_{s \leq t}$ mapping $\cc^{0,1}(\rr^d)$  into itself is called a \textit{variational operator} if it satisfies the monotonicity, additivity and regularity properties \eqref{monot}, \eqref{addi}, \eqref{reg} of Hypotheses \ref{L-O} and the following one, requiring the existence of characteristics (see Lemma \ref{class}):
\begin{enumerate}[(4')]
		\item \label{var} Variational property:  for each Lipschitz $\cc^1$ function $u$, $Q$ in $\rr^d$ and $s\leq t$, there exists $(q,p)$ such that $p =d_q u$, $Q^t_s(q,p)=Q$ and if $\gamma$ denotes the Hamiltonian trajectory issued from $(q(s),p(s))=(q,p)$, \[R^t_s u(Q)=u(q)+ \aaa^t_s(\gamma).\]\end{enumerate}
We call \emph{variational solution} to the Cauchy problem associated with $u_0$ a function given by a variational operator as follows: $u(t,q)=R^t_0 u_0(q)$.

\begin{rem}\label{comprop} Variational property (\ref{comp}') implies Compatibility property \eqref{comp}. This implies in particular that if a variational operator satisfies the Markov property \eqref{mark} of Hypotheses \ref{L-O}, it is a viscosity operator. 
\end{rem}

\begin{proof} 
	We fix $s$ and $Q$ and a Lipschitz $\cc^2$ solution $u$ of \eqref{HJ}. Variational property (4') implies that there exists $q$ and $p = \dd_q u_s$ with $Q^t_s(q,p)=Q$ such that, if $\gamma(\tau)=\phi^\tau_{s}(q,p)$ denotes the Hamiltonian trajectory issued from $(q,p)$ at time $s$, $R^t_s u_s(Q)=u_s(q)+\aaa^t_s(\gamma)$.
	Since $p=\dd_q u_s$ and $Q=Q^t_s(q,p)$, Lemma \ref{class} can be applied and states that $u(t,Q)=u_s(q)+\aaa^t_s(\gamma)$, hence the conclusion: $R^t_s u_s(Q) =u_t(Q)$.
\end{proof}

The uniqueness of such an operator is not granted a priori. See Appendix \ref{graphselector} for a presentation of the associated notion of \emph{graph selector} and the different ways to define one. In the same Appendix is also proved the following result.
\begin{prop}\label{solpp}
	If $u_0$ is $\cc^2$ and $R^t_s$ is a variational operator, $(t,q)\mapsto R^t_0 u_0 (q)$ solves \eqref{HJ} in the classical sense for almost every $(t,q)$ in $(0,\infty)\times \rr^d$.
\end{prop}
This is weaker than what happens in the viscosity case, since the viscosity solution $u$ solves the equation on its domain of differentiability (which is of full measure since the solution is Lipschitz) even for $u_0$ only Lipschitz. We do not know either if $(t,q) \mapsto R^t_0 u_0(q)$ solves the equation almost everywhere when $u_0$ is only Lipschitz.

In this paper, we present a complete construction of the variational operator under Hypothesis \ref{esti}, which comes down to build a graph selector directly for the suspended geometric solution and its wavefront introduced in Appendix \ref{graphselector}. We follow the idea of J.-C. Sikorav (see \cite{siko} or \cite{viterboX}) consisting in selecting suitable critical values of a \emph{generating family} describing this geometric solution. In order to get Lipschitz estimates for this operator, we work with the explicit generating family constructed by M. Chaperon via the \emph{broken geodesics} method (see \cite{chaperon} and \cite{chap2}), whose critical points and values are related to the Hamiltonian objects of the problem. We use a  general \emph{critical value selector} $\sigma$ defined from an axiomatic point of view (see Proposition \ref{mmax}), for functions which differ by a Lipschitz function from a nondegenerate quadratic form. 

An obstacle is that the generating family of Chaperon is of this form only for Hamiltonians that are quadratic for large $\|p\|$, so we need to modify the Hamiltonian for large $\|p\|$ into a quadratic form $Z$ to be able to use the critical value selector, and check that the choice of $Z$ does not matter in the definition of the operator. 

We denote by $R^t_s$ the obtained operator, keeping in mind that it depends a priori on the choice of a critical value selector $\sigma$. The explicit derivatives of the generating family allow to prove the estimates of the following statement.
\begin{theo}\label{lip} If $H$ satisfies Hypothesis \ref{esti} with constant $C$, there exists a variational operator, denoted by $(R^t_s)_{s\leq t}$, such that for $0\leq s \leq s' \leq t'\leq t$ and $u$ and $v$ two $L$-Lipschitz functions,	
	\begin{enumerate}\renewcommand{\labelitemi}{$\bullet$}
		\item $R^t_s u$ is Lipschitz with ${\rm Lip}(R^t_s u) \leq e^{C(t-s)}(1+L)-1$,
		\item $\|R^{t'}_s u-R^t_s u\|_\infty \leq Ce^{2C(t-s)}(1+L)^2|t'-t|$,
		\item $\|R^t_{s'} u-R^t_s u\|_\infty \leq C(1+L)^2|s'-s|$,
		\item $\forall Q \in \rr^d, \left|R^t_s u(Q) - R^t_s v(Q)\right| \leq \|u-v\|_{\bar{B}\!\left(Q,(e^{C(t-s)}-1)(1+L)\right)}$,
	\end{enumerate}
	where $\bar{B}\!\left(Q,r\right)$ denotes the closed ball of radius $r$ centered in $Q$ and $\|u\|_K:= sup_K |u|$.
\end{theo}
The interest of these estimates is that they behave well with the iteration of the operator $R^t_s$, and Theorem \ref{lip} allows then to prove Theorem \ref{iteratedvisc} with no compactness assumption on $H$.

With the same method we are also able to quantify the dependence of the constructed operator $R^t_s$ with respect to the Hamiltonian:
\begin{prop}\label{Hlip}
	Let $H_0$ and $H_1$ be two $\cc^2$ Hamiltonians satisfying Hypothesis \ref{esti} with constant $C$, $u$ be a $L$-Lipschitz function, $Q$ be in $\rr^d$ and $s\leq t$. Then
	\[|R^t_{s,H_1} u(Q) -R^t_{s,H_0}u(Q)| \leq (t-s)\|H_1-H_0\|_{\bar{V}},\]
	where $\bar{V}=[s,t]\times \bar{B}\!\left(Q,(e^{C(t-s)}-1)(1+L)\right)\times \bar{B}\!\left(0,e^{C(t-s)}(1+L)-1\right)$.
\end{prop}

An other formulation of the two last estimates is a localized version of the monotonicity of this variational operator with respect to the initial condition or to the Hamiltonian:

\begin{prop}\label{locmon} If $H_0$ and $H_1$ are two $\cc^2$ Hamiltonians satisfying Hypothesis \ref{esti} with constant $C$, then for $s  \leq t$, $Q$ in $\rr^d$ and $u$ and $v$ two $L$-Lipschitz functions,
\begin{itemize}
\item $R^t_s u(Q) \leq R^t_s v(Q)$ if $u\leq v$ on the set $\bar{B}\!\left(Q,(e^{C(t-s)}-1)(1+L)\right)$,
\item $R^t_{s,H_1} u(Q) \leq R^t_{s,H_0}u(Q)$ if $H_1 \geq H_0$ \\ \hspace*{1cm} on the set $[s,t]\times \bar{B}\!\left(Q,(e^{C(t-s)}-1)(1+L)\right)\times \bar{B}\!\left(0,e^{C(t-s)}(1+L)-1\right)$.
\end{itemize}
\end{prop}

\subsection{A link between variational and viscosity operator} If the variational and viscosity operators do not coincide in general, Q. Wei showed in \cite{wei} that, for compactly supported Hamiltonians, it is possible to obtain the viscosity operator by iterating the variational operator along a subdivision of the time space and letting then the maximal step of this subsequence tend to $0$. This result fits in the approximation scheme proposed by Souganidis in \cite{souga} for a  slightly different set of assumptions, where the variational operator acts like a \emph{generator}. We also refer to \cite{barsouga} for a presentation of this approximation scheme method in a wider framework that includes second order Hamilton-Jacobi equations.

Let us fix a sequence of subdivisions of $[0,\infty)$ $\left((\tau^N_i)_{i\in\nn}\right)_{N \in \nn}$ such that for all $N$, $0=\tau^N_0$, $\tau^N_i \underset{i\to \infty}{\to} \infty$ and $i \mapsto \tau^N_i$ is increasing. Let us also assume that for all $N$, $i \mapsto \tau^N_{i+1}-\tau^N_i$ is bounded by a constant $\delta_N$ such that $\delta_N\to 0$ when $N\to \infty$. For $t$ in $\rr_+$, we denote by $i_N(t)$ the unique integer such that $t$ belongs to $[\tau_{i_N(t)}^N,\tau_{i_N(t)+1}^N)$. If $u$ is Lipschitz on $\rr^d$, and $0\leq s \leq t$, let us define the iterated operator at rank $N$ by\begin{displaymath}
	R^t_{s,N} u = R^{t}_{\tau^N_{i_N(t)}} R^{\tau^N_{i_N(t)}}_{\tau^N_{i_N(t)-1}}\cdots R^{\tau^N_{i_N(s)+1}}_{s} u,
	\end{displaymath}
where $R^t_s$ is any variational operator satisfying the Lipschitz estimate of Theorem \ref{lip}.
\begin{theo}[Wei's theorem]\label{iteratedvisc}\label{cc}
	For each Hamiltonian $H$ satisfying Hypothesis \ref{esti}, the sequence of iterated operators $(R^t_{s,N})$ converges simply when $N\to \infty$ to the viscosity operator $V^t_s$. Furthermore, for each Lipschitz function $u$, $\left\{(s,t,Q) \mapsto R^t_{s,N}u(Q)\right\}_N$ converges uniformly towards $(s,t,Q)\mapsto V^t_s u(Q)$ on every compact subset of $\left\lbrace 0\leq s \leq t\right\rbrace \times \rr^d$.
\end{theo}
Theorem \ref{iteratedvisc} implies amongst other things the existence of the viscosity operator, and the local uniform convergence allows to transfer Lipschitz estimates to the viscosity framework:
\begin{prop}\label{visclip} If $H$ satisfies Hypothesis \ref{esti} with constant $C$, the viscosity operator $(V^t_{s})_{s\leq t}$ satisfies the following estimates: 
	for $0\leq s \leq s' \leq t'\leq t$ and $u$ and $v$ two Lipschitz functions with Lipschitz constant $L$,
	\begin{enumerate}\renewcommand{\labelitemi}{$\bullet$}
		\item $V^t_s u$ is Lipschitz with ${\rm Lip}(V^t_s u) \leq e^{C(t-s)}(1+L)-1$,
		\item $\|V^{t'}_s u-V^t_s u\|_\infty \leq Ce^{2C(t-s)}(1+L)^2|t'-t|$,
		\item $\|V^t_{s'} u-V^t_s u\|_\infty \leq C(1+L)^2|s'-s|$,
		\item $\forall Q \in \rr^d, \left|V^t_s u(Q) - V^t_s v(Q)\right| \leq \|u-v\|_{\bar{B}\!\left(Q,(e^{C(t-s)}-1)(1+L)\right)}$.\end{enumerate}
	
	Moreover, if $H_0$ and $H_1$ are two Hamiltonians satisfying Hypothesis \ref{esti} with constant $C$, $u$ is a $L$-Lipschitz function, $Q$ is in $\rr^d$ and $s\leq t$, the associated operators satisfy
	\[|V^t_{s,H_1} u(Q) -V^t_{s,H_0}u(Q)| \leq (t-s)\|H_1-H_0\|_{\bar{V}},\]
	where $\bar{V}=[s,t]\times \bar{B}\!\left(Q,(e^{C(t-s)}-1)(1+L)\right)\times \bar{B}\!\left(0,e^{C(t-s)}(1+L)-1\right)$.
	
	Furthermore, 
	\begin{itemize}
		\item $V^t_s u(Q) \leq V^t_s v(Q)$ if $u\leq v$ on the set $\bar{B}\!\left(Q,(e^{C(t-s)}-1)(1+L)\right)$,
		\item $V^t_{s,H_1} u(Q) \leq V^t_{s,H_0}u(Q)$ if $H_1 \geq H_0$\\ \hspace*{1cm} on the set $[s,t]\times \bar{B}\!\left(Q,(e^{C(t-s)}-1)(1+L)\right)\times \bar{B}\!\left(0,e^{C(t-s)}(1+L)-1\right)$.
	\end{itemize}
\end{prop}
These estimates are not a priori very surprising since they are satisfied for classical solutions, but due to their dynamical origin they are likely to be sharper than the ones obtained using viscosity arguments. For example, the Lipschitz estimate with respect to $u$ gives a better speed of propagation than the one obtained in Proposition \ref{viscspeed} with $e^{CT}(1+L)-1$ as uniform Lipschitz constant.

\subsection{The convex case: Joukovskaia's theorem and the Lax-Oleinik semigroup}
If $H$ is strictly convex w.r.t. $p$, the Lagrangian function, defined on the tangent bundle, is the Legendre transform of $H$:
\[L(t,q,v)= \sup_{p \in \left(\rr^d\right)^\star} p\cdot v - H(t,q,p).\]
The Legendre inequality writes
\[L(t,q,v)+ H(t,q,p) \geq p\cdot v\]
for all $t$, $q$, $p$ and $v$, and is an equality if and only if $p= \dd_v L(t,q,v)$ or equivalently $v=\dd_p H(t,q,p)$. In particular, if $(q(\tau),p(\tau))$ is a Hamiltonian trajectory, $\dot{q}(\tau)=\dd_pH(\tau,q(\tau),p(\tau)$ and \[
\int^t_s L(\tau,q(\tau),\dot{q}(\tau))d\tau = \int^t_s p(\tau)\cdot \dot{q}(\tau) -H(\tau,q(\tau),p(\tau) d\tau. 
\]
In other words, the Hamiltonian action of a Hamiltonian trajectory coincides with the so-called Lagrangian action of its projection on the position space.

The Lax-Oleinik semigroup $\left(T^t_s\right)_{s\leq t}$ is usually expressed with this Lagrangian action:
 if $u$ is a Lipschitz function on $\rr^d$, then $T^t_s u$ is defined by
	\begin{equation}
	T^t_s u (q)= \inf_{c} u(c(s)) + \int^t_s L\left(\tau,c(\tau),\dot{c}(\tau)\right) d\tau,
	\label{defLO}
\end{equation} 
	where the infimum is taken over all the Lipschitz curves $c:[s,t]\to\rr^d$ such that $c(t)=q$.
Under this form, it is straightforward that the Markov property \eqref{mark} is satisfied by the operator. The Lax-Oleinik semigroup is known to be the viscosity operator when the Hamiltonian is Tonelli, \emph{i.e.} strictly convex and superlinear w.r.t. $p$, and also to satisfy the Variational property (\ref{comp}'),  see for example \cite{fathi}, \cite{bernard}. It is hence both a variational and a viscosity operator for Tonelli Hamiltonians:\[T^t_s=V^t_s.\]

The following theorem states that the variational operator construction of this paper gives effectively the Lax-Oleinik semigroup for uniformly strictly convex Hamiltonian, and the viscosity operator in the convex case. We assume for this result that the critical value selector $\sigma$ satisfies two additional assumptions, presented in Proposition \ref{mmaxbis}.
\begin{theo}[Joukovskaia's theorem]\label{jouk} If $p \mapsto H(t,q,p)$ is convex for each $(t,q)$ or concave for each $(t,q)$, the variational operator ${R}^t_s$ associated with the critical value selector $\sigma$ is the viscosity operator. In particular, it coincides with the Lax-Oleinik semigroup if $H$ is uniformly strictly convex w.r.t. $p$.\end{theo}
The last part of this statement was proved by T. Joukovskaia in the case of a compact manifold, see \cite{jou}.

This theorem was generalized to convex-concave type Hamiltonians, see \cite{weithz} and \cite{BC}, but only when both the Hamiltonian and the initial condition are in the form of splitting variables: 
\[H(t,q,p)=H_1(t,q_1,p_1)+H_2(t,q_2,p_2) \textrm{ and } u_0(q)=u_1(q_1)+u_2(q_2)
\]
where $d=d_1+d_2$, $(q_i,p_i)$ denotes the variables in $T^\star \rr^{d_i}$, $H_1$ (resp. $H_2$) is a  Hamiltonian on $\rr\times \rr^{d_1}$ (resp. on $\rr\times \rr^{d_2}$) convex in $p_1$ (resp. concave in $p_2$), and
$u_1$ and $u_2$ are Lipschitz functions on $\rr^{d_1}$ and $\rr^{d_2}$.
\vspace*{1cm}

The paper is organized as follows: in Section \ref{buildvar} we build the variational operator and prove Theorem \ref{lip}. We first describe the construction of Chaperon's generating family  and its properties (\S\ref{chapgen}), and introduce the notion of critical value selector and its properties (\S \ref{sectcritval}). Then, we address carefully the difficulty related to the bahaviour of the Hamiltonian for large $p$ in order to define the variational operator without compactness assumption  (\S \ref{defR}). We finally collect some properties of the variational operator and its Lipschitz estimates, proving Theorem \ref{lip} and Propositions \ref{Hlip} and \ref{locmon} (\S \ref{varop}). 

In Section \ref{iteration} we prove Theorem \ref{iteratedvisc}. We study the uniform Lipschitz estimates of the iterated operator  $R^t_{s,N}$ (\S\ref{iterlip}), and then show that the limit of any subsequence is the viscosity operator (\S\ref{AA}). Section \ref{iteration} can be read independently from Section \ref{buildvar}, once the Lipschitz constants of Theorem \ref{lip} are granted.

In Section \ref{convex} we give a direct proof of Joukovskaia's theorem, while describing the Lax-Oleinik semigroup with the broken geodesics method (\S \ref{chapgencv}). 

Appendix \ref{chap} details the construction and properties of Chaperon's generating families for the Hamiltonian flow, both in the  general (\S \ref{gfgc}) and in the convex case (\S \ref{gfcc}). Appendix \ref{minmax} proposes a functorial construction of a critical value selector as needed in the construction of the variational operator.  It requires two deformation lemmas proved in Appendix \ref{deformation}. Appendix \ref{C} is about viscosity solutions, and gives an elementary proof of the uniqueness for Lipschitz initial data and under Hypothesis \ref{esti}, via a standard doubling variables argument. Appendix \ref{graphselector} presents the notion of graph selector and contains a proof of Proposition \ref{solpp}.\\
$ $\\
\emph{Acknowledgement.} I am grateful to my supervisor Patrick Bernard for his advices and careful reading. I also thank Qiaoling Wei, Marc Chaperon and Alain Chenciner for fruitful discussions in the cheerful Observatoire de Paris. This paper was improved by many suggestions of the referees of my PhD thesis, Jean-Claude Sikorav and Guy Barles, and by the anonymous referee of the paper.

\section{Building a variational operator}\label{buildvar}

In this section we present the complete construction of the variational operator, following the idea proposed by  J.-C. Sikorav in \cite{sikoexp} and M. Chaperon in \cite{chap2}. We work with an explicit generating family of the geometric solution defined by Chaperon via the \emph{broken geodesics} method (see \cite{chapgeo}). We gather its properties in the next paragraph, referring to Appendix \ref{chap} for some of the proofs. Then we apply on this generating family a \emph{critical value selector}, which we handle only via a few axioms, see Proposition \ref{mmax}. The existence of a selector satisfying these axioms is proved in Appendix \ref{minmax}. This selector can only be directly applied to generating families associated with Hamiltonians coinciding with a quadratic form at infinity, so we need to handle this difficulty by modifying the Hamiltonian for large $p$, see Proposition \ref{propladefi} and Definition \ref{ladefinition}. The rest of the chapter consists in verifying that the obtained operator is a variational operator, and that it satisfies the Lipschitz estimates of Theorem \ref{lip}.

\subsection{Chaperon's generating families}\label{chapgen}
We first build a \emph{generating family} of the Hamiltonian flow, following Chaperon's \emph{broken geodesics} method introduced in \cite{chapgeo} and detailed in \cite{chaperon}, and then adapt it to the Cauchy problem. The results of this section are detailed and proved in Appendix \ref{chap}.
	
Under Hypothesis \ref{esti}, it is possible to find a $\delta_1 >0$ depending only on $C$ (for example $\delta_1=\frac{\ln(3/2)}{C}$) such that $\phi^t_s-{\rm id}$ is $\frac{1}{2}$-Lipschitz (see Proposition \ref{gron}), and as a consequence $(q,p) \mapsto (Q^t_s(q,p),p)$ is a $\cc^1$-diffeomorphism for each $|t-s| \leq \delta_1$, where $(Q^t_s,P^t_s)$ denotes the components of the Hamiltonian flow $\phi^t_s$. 

For a Hamiltonian $H$ satisfying Hypothesis \ref{esti} and $0\leq t-s\leq \delta_1$, let $F^t_s:\rr^{2d} \to \rr$ be the $\cc^1$ function defined by\begin{equation}
	F^t_s(Q,p)=\int^t_s \left(P^\tau_s(q,p)-p\right)\cdot\dd_\tau Q^\tau_s(q,p)-H(\tau,\phi^\tau_s(q,p))\;d\tau,
	\end{equation}
	where $q$ is the only point satisfying $Q^t_s(q,p)=Q$. The function $F^t_s$ is called a \emph{generating function} for the flow $\phi^t_s$, meaning that\begin{displaymath}
	(Q,P) = \phi^t_s (q,p) \iff  \left\{\begin{array}{rcl}
	\dd_p F^t_s (Q,p)&=&q-Q,\\
	\dd_Q F^t_s(Q,p)&=&P-p,
	\end{array}\right.
	\end{displaymath}
	which is proved in Proposition \ref{F}.

	Note that if $H(t,q,p)=H(p)$ is integrable, Hamiltonian trajectories have constant impulsion $p$ and $F^t_s(Q,p)=-(t-s)H(p)$ does not depend on $Q$.  
	
When $t-s$ is large, we choose a subdivision of the time interval with steps smaller than $\delta_1$ and add intermediate coordinates along this trajectory.	For each $s\leq t$ and $(t_i)$ such that $t_0=s\leq t_1 \leq \cdots \leq t_{N+1}=t$ and $t_{i+1}-t_i\leq \delta_1$ for each $i$, let $G^t_s:\rr^{2d(1+N)}\to \rr$ be the function defined by\begin{equation}\label{Gfam}
	G^t_s(p_0,Q_0,p_1,Q_1,\cdots ,Q_{N-1}, p_N,Q_N)=\sum_{i=0}^N F^{t_{i+1}}_{t_i}(Q_i,p_i) +p_{i+1}\cdot(Q_{i+1}-Q_i)
	\end{equation}
	where indices are taken modulo $N+1$.

In Proposition \ref{Ggen} we prove that $G^t_s$ is a \emph{generating function} for the flow $\phi^t_s$, meaning that if $(Q,p)=(Q_N,p_0)$ and $\nu=(Q_0,p_1,\cdots, Q_{N-1}, p_N)$, \begin{displaymath}
(Q,P) = \phi^t_s (q,p) \iff  \exists \nu \in \rr^{2dN},
\left\{\begin{array}{rcl}
\dd_p G^t_s (p,\nu,Q)&=&q-Q,\\
\dd_Q G^t_s(p,\nu,Q)&=&P-p,\\
\dd_\nu G^t_s(p,\nu,Q)&=&0,
\end{array}\right.
\end{displaymath}
and in this case $(Q_i,p_{i+1})=\phi^{t_{i+1}}_s(q,p)$ for all $0\leq i\leq N-1$. Furthermore, if $Q=Q^t_s (q,p)$ and $\gamma$ denotes the Hamiltonian trajectory issued from $(q,p)$,\begin{equation}\label{critG} G^t_s(p,\nu,Q)=\aaa^t_s(\gamma)-p\cdot(Q-q)\end{equation}
for critical points $\nu$ of $\nu \mapsto G^t_s(p,\nu,Q)$.

This is called the \emph{broken geodesics method}: $G^t_s$ is actually the sum of the actions of the unique Hamiltonian trajectories $\gamma_i$ such that $\gamma_i(t_i)=(\star,p_i)$ and $\gamma_i(t_{i+1})=(Q_i,\star)$ and of boundary terms (of the form $p_{i+1}\cdot(q_{i+1}-Q_i)$) smartly arranged in order that taking critical values for $G^t_s$ is equivalent to sew the pieces of trajectories $\gamma_i$ at the intermediate points into a nonbroken geodesic on the whole time interval.

Note that if $H(t,q,p)=H(p)$, this function is quite simple: \begin{equation}\label{Gquad} G^t_s(p_0,Q_0,p_1,Q_1,\cdots, Q_{N-1}, p_N,Q_N)=\sum_{i=0}^N -(t_{i+1}-t_i)H(p_i) +p_{i+1}\cdot(Q_{i+1}-Q_i).\end{equation}

Now let us use the generating family $G^t_s$ of the flow to build what is called a \emph{generating family} for the Cauchy problem associated with the Hamilton-Jacobi equation \eqref{HJ} and an initial condition $u$, using a composition formula proposed by Chekanov.
 If $u : \rr^d \to \rr$ is Lipschitz and $s\leq t$, let us define $S^t_s u$ by\begin{equation}\label{S}
	\begin{array}{r c c l}  S^t_s u : &\rr^d \times \rr^d \times \rr^d\times \rr^{2dN}  &\to &\rr\\
	& (Q,\underbrace{q,p,\nu}_{\xi}) & \mapsto & u(q)+G^t_s(p,\nu,Q) + p\cdot(Q-q).\end{array}
	\end{equation}

\begin{prop}\label{crit} Let $u:\rr^d \to \rr$ be a Lipschitz $\cc^1$  initial condition and $0\leq t-s \leq T$.
	If $Q$ is fixed in $\rr^d$, $(q,p,Q_0,p_1,\cdots,p_N)$ is a critical point of $S^t_s u(Q,\cdot)$ if and only if  \[
	\left\{ \begin{array}{l}
	p=du(q),\\
	Q^t_s(q,p)=Q,\\
	(Q_{i-1},p_i)=\phi^{t_i}_s(q,p) \;\, \forall\, 1\leq i\leq N,\end{array}\right.
	\]
	and in that case, $\dd_Q S^t_s u(Q,q,p,Q_0,\cdots,p_N)=P^t_s(q,p)$.
	
	Furthermore, the critical value of $S^t_s u(Q,\cdot)$ associated with a critical point $(q,p,\nu)$ is equal to $u(q)+\aaa^t_s(\gamma)$, where $\gamma$ denotes the Hamiltonian trajectory $\tau \mapsto \phi^\tau_s(q,p)$.
	\end{prop}

\begin{proof} 	The point $(q,p,\nu)$ is a critical point of $S^t_s u(Q,\cdot)$, if and only if
	\[\left\{\begin{array}{rl}
	0=\dd_q S^t_s u(Q,q,p,\nu)=& du(q)-p,\\
	0=\dd_p S^t_s u(Q,q,p,\nu)=& \dd_p G^t_s(p,\nu,Q)+Q-q,\\
	0=\dd_\nu S^t_s u(Q,q,p,\nu)=& \dd_\nu G^t_s(p,\nu,Q).
	\end{array}\right.
	\]
	Since $G$ is a generating family of the flow, the two last lines implies that $Q^t_s(q,p)=Q$ and $\phi^{t_i}_s(q,p)=(Q_{i-1},p_i)$, hence $P^t_s(q,p)=\dd_Q G^t_s u(p,\nu,Q)+p=\dd_Q S^t_s u(Q,\xi)$. The form of the critical values directly follows from the form of the critical values of $G$, see \eqref{critG}.
\end{proof}

In other words, if $\Gamma$ denotes the graph of $du$ $\{(q,du(q)),q\in \rr^d\}$, the generating family that we built describes the so-called \emph{geometric solution} $\phi^t_s \!\left(\Gamma\right)$ as follows:
\begin{displaymath}
\phi^t_s\! \left(\Gamma\right)  = \left\lbrace (Q, \dd_q S^t_s u(Q,\xi))|Q\in \rr^d,\dd_\xi S^t_s u(Q,\xi)=0\right\rbrace,
\end{displaymath}
meaning that above each point $Q$, a point $(Q,P)$ is in $\phi^t_s \!\left(\Gamma\right) $ if and only if there is a critical point $\xi$ of $\xi \mapsto S^t_s u(Q,\xi)$ such that $P=\dd_Q S^t_s u(Q,\xi)$.

Let us state the values of the other derivatives of $S^t_s u$ at the points of interest:
\begin{prop}\label{Sder} Let $u$ a $\cc^1$ $L$-Lipschitz function and $Q$ in $\rr^d$ be fixed. \begin{enumerate}
		\item \label{temps} If $\xi=(q,p,\nu)$ is a critical point of $\xi \mapsto S^t_s u (Q,\xi)$, then		
			\[
			\left\{ \begin{array}{rcl}
			\dd_t S^t_s u(Q,\xi)&=&-H(t,Q,P^t_s(q,p)),\\
			\dd_s S^t_s u(Q,\xi)&=& H(s,q,p).\end{array}\right.
			\]	
		\item \label{SHdep} If $H_\mu$ is a $\cc^2$ family of Hamiltonians satisfying Hypothesis \ref{esti} with constant $C$, the same subdivision can be chosen to build the associated generating families $S^t_{s,\mu} u$, and then $\mu\mapsto S^t_{s,\mu} u(Q,\xi)$ is $\cc^1$ and	
		if $\xi = (q,p,\nu)$ is a critical point of $\xi \mapsto S^t_{s,\mu} u (Q,\xi)$, \[
		\dd_\mu S^t_{s,\mu} u (Q,\xi) = - \int^t_s \dd_\mu H_\mu(\tau,\phi^\tau_s(q,p)) \;d\tau.
		\]		
	\end{enumerate}
\end{prop}

\begin{proof} We obtain these derivatives using Proposition \ref{F} and \ref{FHdep}, and the fact that a critical point $\xi=(q,p,\nu)$ of the generating family $\xi \mapsto S^t_s u(Q,\xi)$ describes steps of a nonbroken Hamiltonian trajectory  from $(q,p)$ to $(Q,P^t_s(q,p))$ (Proposition \ref{crit}).
\end{proof}

Propositions \ref{crit} and \ref{Sder} imply that if $\xi$ is a critical point of $S^t_s u(Q,\cdot)$, 
the Hamilton-Jacobi equation is satisfied at this one point: $\dd_t S^t_s u(Q,\xi)+H(t,Q,\dd_Q S^t_s u(Q,\xi))=0$. In particular if $(t,Q)\mapsto \xi(t,Q)$ is a differentiable function giving for each $(t,Q)$ a critical point of $S^t_su(Q,\cdot)$, then $(t,Q)\mapsto S^t_s u(Q,\xi(t,Q))$ is a differentiable solution of the Cauchy problem. An idea to build a generalized solution is then to select adequatly critical values of $S^t_s u(Q,\cdot)$, which we are going to do in the next paragraphs.

Until now, we only used the part of Hypothesis \ref{esti} stating that $\|\dd^2_{(q,p)} H\|$ is uniformly bounded. The two next propositions requires the fact that $\|\dd_{(q,p)}H(t,q,p)\|\leq C(1+\|p\|)$. The first one states that if $H$ is nearly quadratic at infinity, so is $\xi \mapsto S^t_s u (Q,\xi)$, and the second one allows to localize the critical points of $S^t_s u$.

\begin{prop}\label{Squad} Let $Z$ be a (possibly degenerate) quadratic form on $\rr^d$. If both
	 $H$ and $(t,q,p)\mapsto Z(p)$ satisfy Hypothesis \ref{esti} with the same constant $C$, and $H(t,q,p)=Z(p)$ for all $\|p\|\geq R$, then $S^t_s u(Q,\xi)=\zz(\xi)+\ell(Q,\xi)$, where $\xi\mapsto \ell(Q,\xi)$ is a Lipschitz function with constant $\|Q\|+{\rm Lip}(u)+4(1+R)$ and $\zz$ is the nondegenerate quadratic form with associated matrix\begin{displaymath}
	 \frac{1}{2}\left(\begin{array}{c c c c c | c c c c c }
	 2\tau_0 Z & 0 & 0  & \cdots & 0                  & -{\rm Id}  &  {\rm Id}   & 0      & \cdots   & 0 \\
	 0 & 2\tau_1 Z & 0  & \cdots & 0                & 0      & - {\rm Id}  &  {\rm Id}   &  \ddots  & \vdots \\
	 0 & 0 & 2\tau_2 Z &  \ddots & 0                & 0      & 0      & \ddots & \ddots   & 0\\
	 \vdots & \vdots & \ddots & \ddots & \vdots           & \vdots & \vdots & \ddots & - {\rm Id}    & {\rm Id}\\
	 0 & 0 & \cdots & 0 & 2\tau_N Z            & 0      &0       &\cdots  &0         & - {\rm Id}\\ \hline
	 -{\rm Id} & 0 & 0 & \cdots & 0 & 0 & &\cdots & &0 \\
	 {\rm Id} & - {\rm Id}& 0 & \cdots & 0 &  & & & &  \\
	 0 &  {\rm Id}& \ddots & \ddots & \vdots & \vdots  & & 0 & & \vdots\\
	 \vdots & \ddots & \ddots & - {\rm Id}& 0 & & & & &  \\
	 0 & \cdots & 0 &  {\rm Id}& -{\rm Id} & 0 & &\cdots & &0 
	 \end{array}\right)
	 \end{displaymath}
	 when written in the basis $(p,p_1,\cdots,p_N,q,Q_0,\cdots, Q_{N-1})$, where $\tau_i=t_{i+1}-t_i$.
\end{prop}
\begin{proof} Let us denote $\tilde{H}(t,q,p)=Z(p)$, and apply Proposition \ref{lipfam}, noticing that since $H=\tilde{H}$ for $\|p\|\geq R$, $\|d_{q,p}(H-\tilde{H})(t,q,p)\|\leq 2C(1+\|p\|)\leq 2C(1+R)$. It gives that a subdivision can be chosen for both $H$ and $\tilde{H}$ and that $\tilde{G}^t_s- G^t_s $ is then $4(1+R)$-Lipschitz.
	
For $\tilde{H}$, it directly follows from \eqref{Gquad} that $\tilde{S}^t_s u(Q,q,p,\nu)= u(q)+\zz(\xi)+p_N\cdot Q$. The quadratic form $\zz$ is nondegenerate as the associated matrix is invertible.
	
Since $\xi\mapsto \tilde{S}^t_s u(Q,\xi)-{S}^t_s u(Q,\xi)= \tilde{G}^t_s (Q,p,\nu)- G^t_s u(Q,p,\nu)$, it is $4(1+R)$-Lipschitz, which proves the point.
\end{proof}

\begin{prop}\label{critloc} Let $H$ be a Hamiltonian satisfying Hypothesis \ref{esti} with constant $C$, $u$ be a $\cc^1$ $L$-Lipschitz function, $s< t$ and $Q$ be in $\rr^d$. If $\xi=(q,p,\nu)$ is a critical point of $\xi \mapsto S^t_s u (Q,\xi)$, then for all $\tau$ in $[s,t]$, \[\phi^\tau_s(q,p) \in B\!\left(Q,(e^{C(t-s)}-1)(1+L)\right)\times  B\!\left(0,e^{C(t-s)}(1+L)-1\right),\]
where $B(x,r)$ denotes the open ball of radius $r$ centered on $x$.	

As a consequence, if $H$ and $\tilde{H}$ are two Hamiltonians satisfying Hypothesis \ref{esti} with constant $C$ and coinciding on $[s,t]\times B\!\left(Q,(e^{C(t-s)}-1)(1+L)\right)\times  B\!\left(0,e^{C(t-s)}(1+L)-1\right)$, the functions $\xi \mapsto S^t_{s,H} u (Q,\xi)$ and $\xi \mapsto S^t_{s,\tilde{H}} u (Q,\xi)$ have the same critical points and the same associated critical values.
\end{prop}
\begin{proof}
We need to quantify the maximal distance covered by Hamiltonian trajectories. Hypothesis \ref{esti} gives an estimate which is uniform with respect to the initial position $q$:
\begin{lem}\label{trajesti} 
	If $H$ satisfies Hypothesis \ref{esti} with constant $C$, then for each $(q,p)$, $s\leq t$, \[
	\|P^t_s(q,p)-p\| < (1+\|p\|)(e^{C(t-s)} -1), \; \; \|Q^t_s(q,p)-q\|< (1+\|p\|)(e^{C(t-s)} -1).
	\]
	In other words, $\phi^t_s(q,p)$ belongs to $B(q,(1+\|p\|)(e^{C(t-s)} -1))\times B(p,(1+\|p\|)(e^{C(t-s)} -1))$.	
\end{lem}
\begin{proof}
	The Hamiltonian system gives that $\|P^t_s(q,p)-p\|\leq  \int^t_s \|\dd_q H(\tau,\phi^\tau_s(q,p))\|d\tau$ and using the hypothesis, we get \begin{equation}\label{eq1}
	\|P^t_s(q,p)-p\|<  C\int^t_s\left(1+ \|P^\tau_s(q,p)\|\right)d\tau\leq C\int^t_s(\|P^\tau_s(q,p)-p\|+1+\|p\|) \;d\tau.
	\end{equation}
	Lemma \ref{gronlem} applied to $f(t)=\|P^t_s(q,p)-p\|$ with $K=C(1+\|p\|)$ gives the first estimate. Since $\|Q^t_s(q,p)-q\|$ is bounded by the same inequality \eqref{eq1}, it is easy to check the second one.
\end{proof}
Now, if $\xi=(q,p,\nu)$ is a critical point, Proposition \ref{crit} states that $p=du(q)$, whence $\|p\|\leq L$. Lemma \ref{trajesti} hence implies that for all $s\leq \tau\leq t$,
\[\|P^\tau_s(q,p)\| \leq \|p\|+ (1+\|p\|)(e^{C(\tau-s)} -1) \leq e^{C(\tau-s)}(1+L)-1.\]
Now using Lemma \ref{trajesti} between $\tau$ and $t$ gives, since $Q=Q^t_\tau(Q^\tau_s(q,p),P^\tau_s(q,p))$:
\[\|Q-Q^\tau_s(q,p) \| \leq (1+\|P^\tau_s(q,p)\|)(e^{C(t-\tau)}-1),\]
and since $1+\|P^\tau_s(q,p)\| \leq e^{C(\tau-s)}(1+L)$, we get
\[\|Q-Q^\tau_s(q,p) \| \leq (1+L)(e^{C(t-s)}-e^{C(\tau-s)})\leq (1+L)(e^{C(t-s)}-1).\]

To prove the second statement, let us recall that if $\tilde{\phi}^t_s=(\tilde{Q}^t_s,\tilde{P}^t_s)$ denotes the Hamiltonian flow for $\tilde{H}$, Proposition \ref{crit} states that $\xi=(q,p,Q_0,p_1,\cdots,p_N)$ is a critical point of $\xi \mapsto S^t_{s,H} u(Q,\xi)$ (resp. of $\xi \mapsto S^t_{s,\tilde{H}} u(Q,\xi)$) if and only if
\[\left\{ \begin{array}{l}
	p=du(q),\\
	Q^t_s(q,p)=Q, \,(\textrm{resp. } \tilde{Q}^t_s(q,p)=Q,)\\
	(Q_{i-1},p_i)=\phi^{t_i}_s(q,p) \,(\textrm{resp. } (Q_{i-1},p_i)=\tilde{\phi}^{t_i}_s(q,p)) \;\, \forall\, 1\leq i\leq N.\end{array}\right.
\]
But if $\xi$ is a critical point of $\xi \mapsto S^t_{s,H} u(Q,\xi)$, the previous work shows that the trajectory $\gamma(\tau)=\phi^\tau_s(q,p)$ stays in $B\!\left(Q,(e^{C(t-s)}-1)(1+L)\right)\times  B\!\left(0,e^{C(t-s)}(1+L)-1\right)$. It is hence a Hamiltonian trajectory both for $H$ and $\tilde{H}$ and $\tilde{\phi}^\tau_s(q,p)=\phi^\tau_s(q,p)$ for all $s \leq \tau \leq t$, which hence shows that $\xi$ is a critical point of  $\xi \mapsto S^t_{s,\tilde{H}} u(Q,\xi)$. The associated critical value $u(q)+\aaa^t_s(\gamma)$ is also the same for $H$ and $\tilde{H}$ since $\gamma$ stays in the set where $H$ and $\tilde{H}$ coincide.
\end{proof}
\begin{rem}\label{estint}
	If $H(p)$ is an integrable Hamiltonian satisfying Hypothesis \ref{esti} with constant $C$, then for each $(q,p)$, $s\leq t$, $P^t_s(q,p)=p$ and Lemma \ref{trajesti} may be improved: \[\|Q^t_s(q,p)-q\|< C(t-s)(1+\|p\|).\] As a consequence, if $u$ is a $\cc^1$ $L$-Lipschitz function, $s< t$ and $Q$ is in $\rr^d$, and $\xi=(q,p,\nu)$ is a critical point of $\xi \mapsto S^t_s u (Q,\xi)$, then for all $\tau$ in $[s,t]$, \[\phi^\tau_s(q,p) \in B\!\left(Q,C(t-s)(1+L)\right)\times  B\!\left(0,L\right).\]
\end{rem}

\subsection{Critical value selector}\label{sectcritval}
Let us denote by $\mathcal{Q}_m$ the set of functions on $\rr^m$ that can be written as the sum of a nondegenerate quadratic form and of a Lipschitz function.

\begin{prop} \label{mmax} There exists a function $\sigma:\bigcup_{m\in \nn}\mathcal{Q}_m\to \rr$ that satisfies:
	\begin{enumerate}
		\item \label{critval} if $f$ is $\cc^1$, then $\sigma(f)$ is a critical value of $f$,
		\item \label{add} if $c$ is a real constant, then $\sigma(c+f)=c+\sigma(f)$,
		\item \label{transl} if $\phi$ is a Lipschitz $\cc^\infty$-diffeomorphism of $\rr^m$ such that $f\circ \phi$ is in $\mathcal{Q}_m$, then \[\sigma(f \circ \phi)=\sigma(f),\]	
		\item \label{monmin} if $f_0-f_1$ is Lipschitz and $f_0 \leq f_1$ on $\rr^d$, then $\sigma(f_0)\leq \sigma(f_1)$,
		\item \label{loc} if $(f_\mu)_{\mu\in[s,t]}$ is a $\cc^1$ family of $\mathcal{Q}_m$ with $(\zz-f_\mu)_\mu$ equi-Lipschitz for some nondegenerate quadratic form $\zz$, then for all $\mu \neq \tilde{\mu}\in [s,t]$,
\[ \min_{\mu \in [s,t]} \min_{x \in Crit(f_\mu)} \dd_\mu f_\mu(x) \leq  \frac{\sigma(f_{\tilde{\mu}})-\sigma(f_\mu)}{\tilde{\mu}-\mu}\leq \max_{\mu \in [s,t]} \max_{x \in Crit(f_\mu)}  \dd_\mu f_\mu(x). \] 
		\item \label{stab} if $g(x,\eta)=f(x)+\zz(\eta)$ where $f$ is in $\mathcal{Q}_m$ and $\zz$ is a nondegenerate quadratic form, then $\sigma(g)=\sigma(f)$.
	\end{enumerate}
	We call such an object a \emph{critical value selector}.
\end{prop}
Such a critical value selector, named \emph{minmax}, was introduced by Chaperon in 1991, see \cite{chap2}. Its construction and properties are detailed in Appendix \ref{minmax}, which proves Proposition \ref{mmax}. The uniqueness of such a selector is not guaranteed, see \cite{wei}.

\begin{rem} Additional assumptions, which are satisfied by the minmax, will be made on the critical value selector (see Proposition \ref{mmaxbis}) in order to prove Joukovskaia's theorem. They are not needed to prove Theorems \ref{lip} and \ref{iteratedvisc}, so we choose not to require them until then. 
\end{rem}	
	
\begin{rem}
Properties \ref{mmax}-\eqref{add}, \ref{mmax}-\eqref{transl} and \ref{mmax}-\eqref{stab} coupled with Viterbo's uniqueness theorem on generating functions (see \cite{viterbouniq} and \cite{theret}) imply that the variational operator we are going to obtain does not depend on the choice of generating family. See Remark \ref{discvit} for more details. Property \ref{mmax}-\eqref{transl} implies in particular that $\sigma(f\circ \tau)=\sigma(f)$ for each affine transformation $\tau$ of $\rr^d$, which would be sufficient to prove Theorems \ref{lip} and \ref{iteratedvisc}.
\end{rem}

Let us fix a critical value selector $\sigma$ for the rest of the discussion. We gather here three consequences of the properties of the critical value selector.
\begin{csq}\label{conti} If $f$ and $g$ are two functions of $\mathcal{Q}_m$ with difference bounded and Lipschitz on $\rr^m$, then\begin{displaymath}
	|\sigma(f)-\sigma(g)| \leq \|f-g\|_\infty. \end{displaymath}
	This is obtained by combining \ref{mmax}-\eqref{monmin} and \ref{mmax}-\eqref{add}.
\end{csq}
\begin{csq}\label{coercmin}
	If $f$ is a coercive function of $\mathcal{Q}_m$, then $\sigma(f)=\min(f)$.
\end{csq}
\begin{proof}
	Since $f$ is in $\mathcal{Q}_m$, there exist a nondegenerate quadratic form $\zz$ and an $L$-Lipschitz function $\ell$ on $\rr^m$ such that $f=\zz+\ell$. 	Since $f$ is coercive, it attains a global minimum at some point $x_0$, and necessarily $\zz$ is coercive, hence convex. Without loss of generality, we assume that $x_0=0$.
	
	We are going to use the following regularization of the norm: for each $\ep>0$, the function $x \mapsto \|x\|+ \ep e^{-\|x\|/\ep}$ is $\cc^1$, strictly convex, $1$-Lipschitz and attains its global minimum $\ep$ at $0$ which is its only critical point.
	
	We have necessarily $\sigma(f)\geq \min(f)=f(0)$ (if $f$ is $\cc^1$, this is true because $\sigma(f)$ is a critical value of $f$ - see Proposition \ref{mmax}-\eqref{critval} - and we get the result for a general $f$ by continuity - see Consequence \ref{conti}). Let us prove the other inequality.
	For each $x$, \[f(x)=\zz(x)+\ell(x) \leq \zz(x) + \ell(0)+ L\|x\| \leq \zz(x)+\ell(0)+ L\left(\|x\|+ \ep e^{-\|x\|/\ep}\right).\]
	The function $x \mapsto \zz(x)+ \ell(0)+L\left(\|x\|+ \ep e^{-\|x\|/\ep}\right)$ is convex as a sum of convex functions and admits $0$ as a critical point, hence its only critical value is $ \ell(0)+\ep$. Since the difference with $f$ is $2L$-Lipschitz, we may apply the Monotonicity property (Proposition \ref{mmax}-\eqref{monmin}) which gives $\sigma(f)\leq \ell(0)+\ep=f(0)+\ep$. Letting $\ep$ tend to $0$ gives the wanted inequality.
\end{proof}

\begin{csq} \label{mmaxquad}
	If $f_\mu=\zz_\mu + \ell_\mu$ is a $\cc^1$ family of $\qq_m$ with $\ell_\mu$ equi-Lipschitz, such that the set of critical points $f_\mu$ does not depend on $\mu$ and such that $\mu\mapsto f_\mu$ is constant on this set, then $\mu \mapsto \sigma(f_\mu)$ is constant.
\end{csq}
\begin{proof} Let us take $\mu$ in some bounded set $[s,t]$.
	Since $\mu \mapsto \zz_\mu$ is $\cc^1$ and $\zz_\mu$ is non degenerate for all $\mu$, the index of $\zz_\mu$ does not depend on $\mu$ and for all $\mu$ there exists a linear isomorphism $\phi_\mu:\rr^m \to \rr^m$ such that $\zz_\mu\circ \phi_\mu=\zz_s$, and $\mu \mapsto \phi_\mu$ is $\cc^1$. Let us define $\tilde{f}_\mu=f_\mu \circ \phi_\mu=\zz_s + \ell_\mu\circ \phi_\mu$ and observe that $\tilde{f}_\mu$ satisfies the hypotheses of Proposition \ref{mmax}-\eqref{loc}: to do so, we only need to check that $\ell_\mu \circ \phi_\mu$ is equi-Lipschitz, which follows from the fact that $\phi_\mu$ is equi-Lipschitz for $\mu$ in the compact set $[s,t]$.
	
	Now, let us check that $\dd_\mu \tilde{f}_\mu(x)=0$ for each critical point $x$ of $\tilde{f}_\mu$, so that both bounds of Proposition \ref{mmax}-\eqref{loc} are zero. Since $\phi_\mu$ is a $\cc^1$-diffeomorphism, $x$ is a critical point of $\tilde{f}_\mu$ if and only if $\phi_\mu(x)$ is a critical point of $f_\mu$, \emph{i.e.} $df_\mu(\phi_\mu(x))=0$. Then since $\mu\mapsto f_\mu$ is constant on its critical points, $\dd_\mu f_\mu(\phi_\mu(x)) = 0$. As a consequence, $\dd_\mu \tilde{f}_\mu(x)= \dd_\mu f_\mu(\phi_\mu(x)) + \dd_\mu \phi_\mu(x) df_\mu(\phi_\mu(x))=0$ and $\mu \mapsto \sigma(\tilde{f}_\mu)$ is constant by  Proposition \ref{mmax}-\eqref{loc}. Proposition \ref{mmax}-\eqref{transl} ends the proof, stating that for all $\mu$, $\sigma(\tilde{f}_\mu)=\sigma(f_\mu\circ \phi_\mu)=\sigma(f_\mu)$.
\end{proof}

\subsection[Definition]{Definition of $R^t_s$}\label{defR}
In this section, we will say that a Hamiltonian is \emph{fiberwise compactly supported} if there exists a $R>0$ such that $H(t,q,p)=0$ for $\|p\|\geq R$. If $Z(p)$ is a quadratic form, we denote by $\hh^C_Z$ the set of $\cc^2$ Hamiltonians $H$ satisfying Hypothesis \ref{esti} with constant $C$ and such that $H(t,q,p)-Z(p)$ is fiberwise compactly supported. 

If $Z$ is a (possibly degenerate) quadratic form, Proposition \ref{Squad} proves that the generating family associated with a Hamiltonian in $\hh^C_Z$ differs by a Lipschitz function from a nondegenerate quadratic form. For Hamiltonians in $\hh^C_Z$, we are then able to define the operator $R^t_s$ directly by applying the critical value selector $\sigma$ on the generating family. The localization of the critical points of the generating family (Proposition \ref{critloc}) allows then to show that the value of the operator does only depend on the behaviour of $H$ on a large enough strip $\rr \times \rr^d \times B(0,R)$.

For general Hamiltonians satisfying Hypothesis \ref{esti}, the generating family is a priori not in any $\mathcal{Q}_m$, so we cannot select a critical value with the selector $\sigma$. To get aound this difficulty, we modify the Hamiltonian outside a large enough strip into some $Z(p)$. It is remarkable that the choice of $Z$ has no incidence on the value of the operator: we hence obtain exactly the same operator by making the Hamiltonian compactly supported with respect to $p$ or by setting it on $\|p\|^2$, for example. To prove Theorems \ref{lip} and \ref{iteratedvisc}, we will simply use $Z=0$, but when dealing with fiberwise convex Hamiltonians, for example to prove Theorem \ref{jouk}, the choice of a convex nondegenerate quadratic form will be more adequate.

\begin{defi} If $H$ is in $\hh^C_Z$ and $s\leq t$, let the operator $(R^t_s)$ be defined for Lipschitz functions $u$ on $\rr^d$ by\begin{displaymath}
R^t_s u (Q)=\sigma(S^t_s u(Q,\cdot)) \,\; \forall Q \in \rr^d,
	\end{displaymath}
	where $S^t_s u(Q,\cdot)$ is the function $\xi \mapsto S^t_s u(Q, \xi)$ and $S$ is the generating family defined at \eqref{S}. In particular, if $u$ is $\cc^1$, $R^t_s u(Q)$ is a critical value of $\xi \mapsto S^t_s u(Q,\xi)$.
\end{defi}

\begin{proof} Proposition \ref{Squad} states that $\xi\mapsto S^t_s u(Q,\xi)$ is in some $\mathcal{Q}_{m}$.
\end{proof}
\begin{prop}\label{indsub}
The operator $R^t_s$ does not depend on the choice of subdivision of $[s,t]$ in the definition of $G$, see \eqref{Gfam}.
\end{prop}

\begin{proof}	
It is enough to consider two cases: either the subdivisions are identical with only one intermediate step $t_i$ changing, or one subdivision is obtained from the other by adding artificially an intermediate step of length zero.
	
In the first case, we observe that if the subdivision is fixed except for one intermediate step $t_i$, the function $t_i \mapsto S^t_s u(Q,\xi)$ is $\cc^1$, hence uniformly continuous, and by Consequence \ref{conti} this implies that $t_i \mapsto R^t_s u(Q)$ is continuous. But the set of critical values of $\xi \mapsto S^t_s u(Q,\xi)$ does not depend on $t_i$ (see Proposition \ref{crit}) and is discrete, hence $t_i \mapsto R^t_s u(Q)$ must be a constant function.
	
In the second case, let us artificially add an intermediate step $t_\iota$ equal to $t_i$: the subdivision is now $s=t_0 \leq t_1 \leq \cdots \leq t_{i-1}\leq t_\iota=t_i  \leq \cdots \leq t_{N+1}=t$ and the variables $(Q,p,Q_0,p_1,Q_1,\cdots,Q_{i-1},p_\iota,Q_\iota,p_i,\cdots p_N)$. We denote by $G$ (resp. $\tilde{G}$) the family associated with the subdivision without (resp. with) $t_\iota$, that takes variables $(Q,p,Q_0,\cdots,Q_{i-1},p_i,\cdots, p_N)$ (resp. $(Q,p,Q_0,\cdots,Q_{i-1},p_\iota,Q_\iota,p_i,\cdots, p_N)$).

Since $F^{t_i}_{t_\iota}=0$ and $F^{t_i}_{t_{i-1}}=F^{t_\iota}_{t_{i-1}}$, we may observe that:\[
\tilde{G}(Q,\cdots,Q_i,p_\iota,Q_\iota,p_{i+1},\cdots, p_N) = G(Q,\cdots,Q_i,p_{i+1},\cdots, p_N)
 - (p_i-p_\iota)\cdot(Q_\iota-Q_{i-1}),
\]
and the same holds for the associated families $S$ and $\tilde{S}$:
\[
\tilde{S}(Q,q,\cdot\cdot,Q_i,p_\iota,Q_\iota,p_{i+1},\cdot\cdot, p_N) = S(Q,q,\cdot\cdot,Q_i,p_{i+1},\cdot\cdot, p_N)
- (p_i-p_\iota)\cdot(Q_\iota-Q_{i-1}).
\]
The affine transformation mapping $p_\iota$ to $\tilde{p}_\iota=p_i-p_\iota$, $Q_\iota$ to $\tilde{Q}_\iota=Q_\iota-Q_{i-1}$ and keeping the other variables fixed preserves the value of the selector by property \ref{mmax}-\eqref{transl} of $\sigma$. In these new coordinates, the family writes:
\[
\tilde{S}(Q,q,\cdots,Q_i,\tilde{p}_\iota,\tilde{Q}_\iota,p_{i+1},\cdots, p_N) = S(Q,q,\cdots,Q_i,p_{i+1},\cdots, p_N)
- \tilde{p}_\iota\cdot \tilde{Q}_\iota
\]
and since $(\tilde{p}_\iota, \tilde{Q}_\iota )\mapsto -\tilde{p}_\iota\cdot \tilde{Q}_\iota$ is a nondegenerate quadratic function of $(\tilde{p}_\iota, \tilde{Q}_\iota )$, the invariance by stabilization \ref{mmax}-\eqref{stab} for $\sigma$ of the critical value selector concludes the proof.
\end{proof}
The following basic continuity result for $R^t_s$, which is improved in Theorem \ref{lip}, is only there to allow to work with $u$ os class  $\cc^1$ and extend the results by density:
\begin{prop}[Weak contraction]\label{contsc}
	If $H$ is in $\hh^C_Z$ and $u$ and $v$ are two Lipschitz functions such that $u-v$ is bounded, then $R^t_s u - R^t_s v$ is bounded by $\|u-v\|_\infty$.
\end{prop}

\begin{proof}
	Let us fix $s$, $t$ and $Q$, and note that the quantity $S^t_s u(Q,\xi) - S^t_s v(Q,\xi)= u(q)-v(q)$ is a Lipschitz and bounded function of $\xi$. The continuity of $\sigma$ established in Consequence \ref{conti} gives that
	\[\|R^t_s u(Q)-R^t_s v(Q)\|\leq \|S^t_s u(Q,\cdot) - S^t_s v(Q,\cdot)\|_\infty \leq \|u-v\|_\infty.\]
\end{proof}

The following proposition implies that the value of the operator depends only on the value of $H$ on a large enough compact set:
\begin{prop}\label{Hdep} Let $Z$ and $\tilde{Z}$ be two quadratic forms, and $H$ (resp. $\tilde{H}$) be a Hamiltonian in $\hh^C_Z$ (resp. $\hh^C_{\tilde{Z}}$). For each $L$-Lipschitz function $u$ and $s\leq t$, if $H=\tilde{H}$ on $\rr\times \rr^d \times  B\!\left(0,e^{C(t-s)}(1+L)-1\right)$, then $R^t_{s,H} u = R^t_{s,\tilde{H}} u$.
\end{prop} 

\begin{proof} 
	Let us first assume that  $u$ is a $\cc^1$ $L$-Lipschitz function and $s\leq t$. Let us define $H_\mu=\mu H +(1-\mu)\tilde{H}$. Observe that $H_\mu$ is in $\hh^C_{\tilde{Z_\mu}}$ where $Z_\mu=\mu Z +(1-\mu)\tilde{Z}$ is a quadratic form, and that there exists $R>0$ such that for all $\mu$ in $[0,1]$, $H_\mu(t,q,p)=Z_\mu(p)$ if $\|p\|\geq R$. 
	
	Proposition \ref{Squad} hence guarantees that for all $\mu$, $S^t_{s,H_\mu} u(Q,\xi)=\zz_\mu(\xi) + \ell_\mu(Q,\xi)$ where $\zz_\mu$ is a nondegenerate quadratic form and $\xi \mapsto \ell_\mu(Q,\xi)$ is Lipschitz with constant ${\rm Lip}(u)+\|Q\|+4(1+R)$. Note that if $Q$ is fixed, the family $\xi \mapsto \ell_\mu(Q,\xi)$ is hence equi-Lipschitz when $\mu$ is in $[0,1]$.
	
	As $H_\mu$ is constant on $\rr\times \rr^d \times  B\!\left(0,e^{C(t-s)}(1+L)-1\right)$, the second part of Proposition \ref{critloc} states that the set of critical points of $\xi \mapsto S^t_{s,H_\mu}u(Q,\xi)$ does not depend on $\mu$, and neither do the associated critical values. 
	
	So if $Q$ is fixed, the family of functions $f_\mu=S^t_{s,H_\mu} u (Q,\cdot)$ satisfies the conditions of Consequence \ref{mmaxquad}, and hence $R^t_{s,H_\mu}u (Q)=\sigma(f_\mu)$ does not depend on $\mu$. As a consequence, $R^t_{s,H} u = R^t_{s,\tilde{H}} u$.
	
	The result extends to every $L$-Lipschitz $u$ thanks to Proposition \ref{contsc} and the fact that $u$ can be $L^\infty$-approximated by a $\cc^1$ $L$-Lipschitz function.
\end{proof}

We now want to extend the definition to a Hamiltonian that is not quadratic at infinity, by modifying it outside some large enough strip $\rr \times \rr^d \times B(0,R)$ into some $Z(p)$. We cannot make sure that the modified Hamiltonian still satisfies Hypothesis \ref{esti} with the same constant $C$ than $H$, so we have to be cautious since the width of the strip depends on $C$. Lemma \ref{lemext} shows that the constant of the modified Hamiltonian can be arbitrarily close to $C$, and this independently from the width of the strip, which avoids any trouble.

\begin{prop}\label{propladefi} Let $H$ be a $\cc^2$ Hamiltonian satisfying Hypothesis \ref{esti} with constant $C$, $u$ be a $L$-Lipschitz function and $s\leq t$. For all $\delta>0$, and for each quadratic form $Z$ such that $\|d^2Z\|\leq C$, there exists a Hamiltonian $H_{\delta,Z}$ in $\hh^{C(1+\delta)}_Z$ that coincides with $H$ on $\rr\times \rr^d\times B\!\left(0,e^{C(1+\delta)(t-s)}(1+L)-1\right)$. Then, $R^t_{s,H_{\delta,Z}}u$ does neither depend on the choice of $H_{\delta,Z}$, nor on the choice of $Z$, nor on $\delta>0$.
\end{prop}
This proposition allows to define the variational operator for general Hamiltonians:
\begin{defi}\label{ladefinition}
	Let $H$ be a $\cc^2$ Hamiltonian satisfying Hypothesis \ref{esti} with constant $C$. For each $L$-Lipschitz function  $u$ and $s\leq t$, we define $R^{t}_{s,H}u =R^t_{s,H_{\delta,Z}}u$, where $\delta>0$ and $H_{\delta,Z}$ is a Hamiltonian of  $\hh^{C(1+\delta)}_Z$ for some quadratic form $Z$ such that $\|d^2Z\|\leq C$, which coincides with $H$ on $\rr\times \rr^d\times B\!\left(0,e^{C(1+\delta)(t-s)}(1+L)-1\right)$ .
\end{defi}

\begin{proof}[Proof of Proposition \ref{propladefi}] Let us show that for all $\delta>0$, there exists $H_\delta$ in $\hh^{C(1+\delta)}_Z$ coinciding with $H$ on $\rr\times \rr^d\times B\!\left(0,R_\delta\right)$, where $R_\delta=e^{C(1+\delta)(t-s)}(1+L)-1$. To do so, we use the following lemma:
	\begin{lem}\label{lemext} 
		If $R>0$ and $\ep>0$, there exists a compactly supported $\cc^2$ function $\varphi:\rr_+ \to [0,1]$, equal to $1$ on $[0,R]$, such that for all $r\geq 0$,
		\[
		|\varphi'(r)|\leq \frac{\ep}{6(1+r)},\;|\varphi''(r)|\leq\frac{\ep}{6(1+r)^2} \textrm{ and } \frac{|\varphi'(r)|}{r}\leq\frac{\ep}{6(1+r)^2}. \] 
		For such a function $\varphi$, if $H$ and $\tilde{H}$ are two Hamiltonians satisfying Hypothesis \ref{esti} with constant $C$, the Hamiltonian $H_\varphi:(t,q,p)\mapsto \varphi(\|p\|) H(t,q,p)+(1-\varphi(\|p\|))\tilde{H}(t,q,p)$ satisfies Hypothesis \ref{esti} with constant $C(1+\ep)$, is equal to $H$ on $\rr \times \rr^d \times B(0,R)$ and $H_\varphi-\tilde{H}$ is fiberwise compactly supported.
	\end{lem}
	
	\begin{proof} Take some $R'>\max(1,R)$ and let us define \[\varphi(r)=\max\!\left(0,1-\frac{\ep}{12}\max\!\left(0,\ln(1+r)-\ln(1+R')\right)\right).\]
		
		If $r\leq R'$, $\varphi(r)=1$. 
		If $r \geq (1+R')e^{12/\ep}-1$, $\varphi(r)=0$. 
		For all $r\geq 0$, $0\leq \varphi(r)\leq 1$.
		
		The function $\varphi$ is $\cc^\infty$ except at $r=R'$ or $r=(1+R')e^{12/\ep}-1$. Let us evaluate its derivatives on $(R',(1+R')e^{12/\ep}-1)$, where $f(r)=1-\frac{\ep}{12}\left(\ln(1+r)-\ln(1+R')\right)$ : \[
		\varphi'(r)=\frac{-\ep}{12(1+r)}, \varphi''(r)=\frac{\ep}{12(1+r)^2}.
		\]
		Furthermore, as long as $r\geq R' > 1$, this implies that
		\[|\varphi'(r)|=\frac{\ep}{12(1+r)}\leq \frac{\ep r}{6(1+r)^2}.\]
			
		Hence the three wanted estimates are satisfied on $(R',(1+R')e^{12/\ep}-1)$. Since $\varphi'$ and $\varphi''$ are zero if $r < R'$ or $r> (1+R')e^{12/\ep}-1$, it is possible to smooth $\varphi$ by below at $R'$ and by above at $(1+R')e^{12/\ep}-1$ without increasing the derivative bounds, keeping $\varphi=1$ for $r \leq R$ and $\varphi$ compactly supported.

Now if $H$ and $\tilde{H}$ are two Hamiltonians satisfying Hypothesis \ref{esti} with constant $C$, let us define $H_\varphi$ by $H_\varphi(t,q,p)=\varphi(\|p\|) H(t,q,p)+(1-\varphi(\|p\|))\tilde{H}(t,q,p)$. It is $\cc^2$, coincides with $H$ on $\rr \times \rr^d \times B(0,R_\delta)$, and $H_\varphi(t,q,p)-\tilde{H}(t,q,p)= \varphi(\|p\|)(H(t,q,p)-\tilde{H}(t,q,p))$ is fiberwise compactly supported since $\varphi(r)=0$ for $r$ large enough.

In order to verify that $H_\varphi$ satisfies Hypothesis \ref{esti} with constant $C(1+\ep)$, let us bound the derivatives of $\phi(p)=\varphi(\|p\|)$:
\[ \begin{split}
\|d\phi(p)\|&=|\varphi'(\|p\|)|\leq \frac{\ep}{6(1+\|p\|)},\\ \|d^2\phi(p)\|&\leq \max\left(|\varphi''(\|p\|)|,\frac{|\varphi'(\|p\|)|}{\|p\|}\right) \leq \frac{\ep}{6(1+\|p\|)^2}.
\end{split}\]
Now, since both $H$ and $\tilde{H}$ satisfy $|H(t,q,p)|\leq C(1+\|p\|)^2$ and $\phi(p) \in [0,1]$ for all $p$,
\[|H_\varphi(t,q,p)| \leq \phi(p)|H(t,q,p)|+(1-\phi(p))|\tilde{H}(t,q,p)| \leq C(1+\|p\|)^2, \]
	Since $H$ and $\tilde{H}$ satisfies Hypothesis \ref{esti} with constant $C$, $H-\tilde{H}$ satisfies Hypothesis \ref{esti} with constant $2C$, and the following holds:	
	\[\begin{split}
	\|dH_\varphi\| & \leq \phi(p) \underbrace{\|dH\|}_{\leq C(1+\|p\|)}+(1-\phi(p))\underbrace{\|d\tilde{H}\|}_{\leq C(1+\|p\|)} + \underbrace{|d\phi(p)|}_{\leq \frac{\ep}{6(1+\|p\|)}} \underbrace{|H-\tilde{H}|}_{\leq 2C(1+\|p\|)^2}\\
	& \leq C(1+\|p\|)+ \frac{\ep}{3}C(1+\|p\|)\leq C(1+\ep)(1+\|p\|),\\
	\|d^2H_\varphi\|&\leq  \phi \|d^2H\|+(1-\phi)\|d^2 \tilde{H}\| + 2 \|d\phi\| \|dH-d\tilde{H}\|+ \|d^2\phi\| |H-\tilde{H}| 	\\
	& \leq  \phi C+(1- \phi ) C + 2 \frac{\ep}{6(1+\|p\|)} \cdot 2C(1+\|p\|) +\frac{\ep}{6(1+\|p\|)^2} \cdot2C(1+\|p\|)^2\\	
	& \leq C+	2\frac{\ep}{3}C + \frac{\ep}{3} C \leq C(1+\ep).
	\end{split}
	\]\end{proof}	
	To build $H_{\delta,Z}$ in $\hh^{C(1+\delta)}_Z$ coinciding with $H$ on $\rr \times \rr^d \times B(0,R_\delta)$, it is enough to apply Lemma \ref{lemext} with $\tilde{H}(t,q,p)=Z(p)$, $\ep=\delta$ and $R=R_\delta=e^{C(1+\delta)(t-s)}(1+L)-1$.
	
	Let us now check that $R^t_{s,H_{\delta,Z}}u$ is independent from the choice of $H_{\delta,Z}$ and $Z$: if $H_{\delta,Z}$ in $\hh^{C(1+\delta)}_Z$ and $\tilde{H}_{\delta,\tilde{Z}}$ in $\hh^{C(1+\delta)}_{\tilde{Z}}$ coincide on $\rr\times \rr^d\times B\!\left(0,e^{C(1+\delta)(t-s)}(1+L)-1\right)$, Proposition \ref{Hdep} applies and  $R^t_{s,H_{\delta,Z}} u =R^t_{s,\tilde{H}_{\delta,\tilde{Z}}} u$.
	
From now on, we may take $Z=0$, hence the set $\hh^C_0$ is exactly the set of $\cc^2$ fiberwise compactly supported Hamiltonians satisfying Hypothesis \ref{esti} with constant $C$.	Let us prove the independence with respect to $\delta$.
	
	Let $s\leq t$ and $u$ a $L$-Lipschitz function be fixed, and still denote by $R_\delta$ the radius  given by $e^{C(1+\delta)(t-s)}(1+L)-1$, which is increasing with respect to $\delta$. 
	Take $\delta>\tilde{\delta}>0$, and $H_\delta$ (resp. $H_{\tilde{\delta}}$) a Hamiltonian in $\hh^{C(1+\delta)}_0$ (resp. $\hh^{C(1+\tilde{\delta})}_0$) coinciding with $H$ on $\rr\times\rr^d\times B\!\left(0,R_\delta\right)$ (resp. $\times B\!\left(0,R_{\tilde{\delta}}\right)$), so that $R^{t,\delta}_{s,H}u(Q)=R^t_{s,H_\delta}u(Q)$ and $R^{t,\tilde{\delta}}_{s,H}u(Q)=R^t_{s,H_{\tilde{\delta}}}u(Q)$. 
	
	Lemma \ref{lemext} applied with $R=R_\delta$, $\ep = \tilde{\delta}$ and $\tilde{H}=0$ gives a Hamiltonian $H_\varphi$ in $\hh^{C(1+\tilde{\delta})}_0$ coinciding with $H$ (hence $H_\delta$) on $\rr\times\rr^d\times B\!\left(0,R_\delta\right)$, and therefore since 
	$B\!\left(0,R_{\tilde{\delta}}\right)\subset B\!\left(0,R_{\delta}\right)$, with $H_{\tilde{\delta}}$ on $\rr\times\rr^d\times B\!\left(0,R_{\tilde{\delta}}\right)$. Proposition \ref{Hdep} gives on the one hand that $R^t_{s,H_\delta}u=R^t_{s,H_\varphi} u$, and on the other hand that $R^t_{s,H_\varphi} u=R^t_{s,H_{\tilde{\delta}}}u$, hence the result.
	
\end{proof}
\begin{add}\label{defcvx}
	If $H$ is uniformly strictly convex with respect to $p$ (\textit{i.e.} there exists $m>0$ such that $\dd^2_p H(t,q,p)\geq m{\rm id}$ for all $(t,q,p)$) and $Z$ is a strictly positive quadratic form such that $\frac{m}{2}{\rm id} \leq Z \leq \frac{C}{2}{\rm id}$, then the function $H_{\delta,Z}$ of Proposition \ref{propladefi} can be chosen uniformly strictly convex w.r.t. $p$.
\end{add}
\begin{proof}
	In the proof of Lemma \ref{lemext}, we assume that $H$ and $\tilde{H}$ are uniformly strictly convex with respect to $p$ with a constant $m>0$. Then following the construction of $H_\varphi$, we may estimate its second derivative with respect to $p$:
	\[\begin{split}
	\dd^2_p H_\varphi &\geq \phi\dd^2_p H+(1-\phi)\dd^2_p \tilde{H} - \left(2 \|d\phi\| \|\dd_p H-\dd_p\tilde{H}\|+ \|d^2\phi\| |H-\tilde{H}|\right){\rm id}\\
	& \geq (m-C\ep) {\rm id} 
	\end{split}	
	\]
	using the estimates on the derivatives of $\varphi$, $H$ and $\tilde{H}$. So, if $\ep<m/C$, the obtained function is uniformly strictly convex.
\end{proof}

\subsection[Properties and Lipschitz estimates]{Properties and Lipschitz estimates of $R^t_s$.}\label{varop}
Let us prove that $(R_s^t)_{s\leq t}$ is a variational operator. Monotonicity and additivity properties are straightforward:

\begin{prop}[Monotonicity]\label{vermonot} If $u\leq v$ are Lipschitz functions on $\rr^d$, then  for each $s \leq t$, $R^t_s u \leq R^t_s v$ on $\rr^d$. \end{prop}

\begin{proof} Let $L$ be a Lipschitz constant for both $u$ and $v$, and fix $s\leq t$, $\delta>0$. Let $H_\delta$ be a Hamiltonian in $\hh^{C(1+\delta)}_0$ coinciding with $H$ on $\rr \times \rr^d \times B\!\left(0,e^{C(1+\delta)(t-s)}(1+L)-1 \right)$ as in Definition \ref{ladefinition}, so that $R^t_{s,H} u(Q)=R^t_{s,H_\delta} u(Q)$ and $R^t_{s,H} v(Q)=R^t_{s,H_\delta}v(Q)$.

Since $S^t_{s,H_\delta} v(Q,\xi) -S^t_{s,H_\delta} u(Q,\xi) = v(q)-u(q)$ is a non negative and Lipschitz function of $\xi$, the monotonicity \ref{mmax}-(\ref{monmin}) of $\sigma$ applies and $R^t_{s,H_\delta} u(Q) \leq R^t_{s,H_\delta} v(Q)$, thus \[R^t_{s,H} u(Q)\leq R^t_{s,H} v(Q).\]
\end{proof}
\begin{prop}[Additivity]\label{veradd}
	If $c$ is a real constant, then $R^t_s (c+u) = c+ R^t_s u$ for each Lipschitz function $u$.
\end{prop}
\begin{proof} The additivity property \ref{mmax}-\eqref{add} of $\sigma$ and the form of $S^t_s u$ conclude, as in the previous proof. 
\end{proof}
\begin{prop}[Variational property]\label{vervar}
	For each $\cc^1$ Lipschitz function $u$, $Q$ in $\rr^d$ and $s\leq t$, there exists $(q,p)$ such that $p =d_q u$, $Q^t_s(q,p)=Q$ and if $\gamma$ denotes the Hamiltonian trajectory issued from $(q(s),p(s))=(q,p)$, \[R^t_s u(Q)=u(q)+ \aaa^t_s(\gamma),\]
\end{prop}

\begin{proof}
	Let us fix $u$, $s\leq t$ and $\delta>0$ and take as in Definition \ref{ladefinition} a Hamiltonian $H_\delta$ in $\hh^{C(1+\delta)}_0$ equal to $H$ on $\rr \times \rr^d \times B\!\left(0,e^{C(1+\delta)(t-s)}(1+L)-1\right)$, such that $R^t_{s,H} u(Q)=R^t_{s,H_\delta}u(Q)$.
	
	Since $u$ is $\cc^1$, $R^t_{s,H_\delta}u(Q)$ is a critical value of $\chi \mapsto S^t_{s,H_\delta} u(Q,\chi)$. Proposition \ref{crit}, which describes the critical points and values of $S$, gives the existence of $(q,p)$ such that $Q^t_{s,H_\delta}(q,p)=Q$ and $p=du(q)$, and states that if $\gamma_\delta(\tau)=\phi^\tau_{s,H_\delta}(q,p)$ denotes the Hamiltonian trajectory issued from $(q,p)$ for the Hamiltonian $H_\delta$, \[R^t_{s,H_\delta}u(Q)=u(q)+\aaa^t_{s,H_\delta}(\gamma_\delta).\] Proposition \ref{critloc}, which localizes the critical points of $S$ under Hypothesis \ref{esti}, gives that $\gamma_\delta(\tau)$ belongs to the set $\rr\times\rr^d \times B\!\left(0,e^{C(1+\delta)(t-s)}(1+L)-1\right)$ for all $\tau$ in $[s,t]$.
	
	Since $H$ and $H_\delta$ coincide on that set for each time in $[s,t]$, $\gamma_\delta$ is also a Hamiltonian trajectory for $H$ on $[s,t]$, the Hamiltonian action of $\gamma_\delta$ has the same expression for $H$ and $H_\delta$, and the conclusion holds: $Q=Q^t_{s,H_\delta}(q,p)=Q^t_{s,H}(q,p)$ and
	\[R^t_{s,H} u(Q)=R^t_{s,H_\delta}u(Q)=u(q)+\aaa^t_{s,H}(\gamma_\delta).\]
\end{proof}
We now prove the Lipschitz estimates of Theorem \ref{lip}, which imply that ${R}^t_s$ satisfies the regularity property \eqref{reg} of Hypotheses \ref{L-O}. 

\begin{proof}[Proof of Theorem \ref{lip}.] Suppose to begin with that $u$ is $\cc^1$ and that $H$ is fiberwise compactly supported, meaning that there exists $R>0$ such that $H(t,q,p)=0$ for $\|p\|\geq R$. Under that assumption, in Proposition \ref{Squad}, the nondegenerate quadratic form $\zz$ does not depend on $s$ or $t$.

For each item of this proof, we are going to use Property \ref{mmax}-(\ref{loc}) on a suitable homotopy $f_\mu$, the form of the derivatives of $S^t_s u$ given in Propositions \ref{crit} and \ref{Sder} and the localization of the critical points of $S^t_s u$ described in Proposition \ref{critloc}. 
	\begin{enumerate}\renewcommand{\labelitemi}{$\bullet$}
		\item Let us show that $R^t_s u$ is Lipschitz with ${\rm Lip}(R^t_s u) \leq e^{C(t-s)}(1+L)-1$.
		Let us fix $Q$ and $h$ in $\rr^d$ and define $f_\mu(\xi)=S^t_s u (Q+\mu h,\xi)$ for $\mu$ in $[0,1]$. The aim is to estimate $|R^t_s u(Q+h)-R^t_s u(Q)| = |\sigma(f_1)-\sigma(f_0)|$.
		
		 Proposition \ref{Squad} states that the family $f_\mu$ is of the form required in Property \ref{mmax}-(\ref{loc}), \emph{i.e.} $f_\mu(\xi)=\zz(\xi)+\ell_\mu(\xi)$, where the family $\ell_\mu$ is equi-Lipschitz with constant ${\rm Lip}(u)+\|Q\|+\|h\|+4(1+R)$.
		 
		  Let us then estimate $\dd_\mu f_\mu$:
		\[\dd_\mu f_\mu(q,p,\nu)=h\cdot\dd_Q S^t_s(Q+\mu h,\xi).\]
		
		If $\xi_\mu=(q_\mu,p_\mu,\nu_\mu)$ is a critical point of $f_\mu$, Proposition \ref{crit} gives on one hand that  $\dd_Q S^t_s(Q+\mu h,\xi_\mu)=P^t_s(q_\mu,p_\mu)$ and Proposition \ref{critloc}, on the other hand, that $\|P^t_s(q_\mu,p_\mu)\|\leq e^{C(t-s)}(1+L)-1$.
		
		To sum it up, we have just proved that $\|\dd_\mu f_\mu\|\leq \|h\|(e^{C(t-s)}(1+L)-1)$ for each critical point of $f_\mu$. This implies by Property \ref{mmax}-(\ref{loc}) of the selector that $|\sigma(f_1)-\sigma(f_0)| \leq \|h\|(e^{C(t-s)}(1+L)-1)$, hence the result.
		
		\item Let us show that $\|R^{t'}_s u-R^t_s u\|_\infty \leq Ce^{2C(t-s)}(1+L)^2|t'-t|$. 
		It is enough to prove the result for $|t-t'|<\delta_1/2$. We may therefore assume that $(t_1,\cdots,t_N)$ is a subdivision suitable both between $s$ and $t$ and between $s$ and $t'$, since the choice of the subdivision does not change the value of the variational operator $R$ (see Proposition \ref{indsub}).
			
			Let us fix $Q$, $t'<t$ and $s$ and define $f_\mu(\xi)=S^{\mu}_s u (Q,\xi)$ for $\mu$ in $[t',t]$. The aim is to estimate $|R^{t}_s u(Q)-R^{t'}_s u(Q)| = |\sigma(f_t)-\sigma(f_{t'})|$.
			
			By Proposition \ref{Squad}, the family $f_\mu$ is as required in Property \ref{mmax}-(\ref{loc}), thanks to the fact that the nondegenerate quadratic form $\zz$ does not depend on $t$ ($=\mu$).
			
			If $\xi_\mu=(q_\mu,p_\mu,\nu_\mu)$ is a critical point of $f_\mu$, Proposition \ref{Sder}-\eqref{temps} gives on one hand that  $\dd_\mu S^\mu_s(Q,\xi_\mu)=-H(\mu,Q,P^{\mu}_s(q_\mu,p_\mu))$ and Proposition \ref{critloc} gives on the other hand that $\|P^\mu_s(q_\mu,p_\mu)\|\leq e^{C(\mu-s)}(1+L)-1$. 
			
	By Hypothesis \ref{esti}, we hence get that\[|\dd_\mu S^{\mu}_s(Q,\xi_\mu)|\leq C(1+\|P^{\mu}_s(q_\mu,p_\mu)\|)^2 \leq Ce^{2C(\mu-s)}(1+L)^2.\]
			
			To sum it up, we have just proved that $\|\dd_\mu f_\mu\|\leq Ce^{2C(t-s)}(1+L)^2$ for each $\mu$ in $[t',t]$ and each critical point of $f_\mu$. Property \ref{mmax}-\eqref{loc} hence states that $\mu \mapsto \sigma(f_\mu)$ is Lipschitz with constant $Ce^{2C(t-s)}(1+L)^2$ on $[t',t]$, hence the result.

		\item Let us show that $\|R^t_{s'} u-R^t_s u\|_\infty \leq C(1+L)^2|s'-s|$.
		Again we may assume that $|s-s'|$ is small enough to choose a subdivision suitable both between $s$ and $t$ and between $s'$ and $t$.
		
		 Let us fix $Q$, $t$ and $s\leq s'$ and define $f_\mu(\xi)=S^{t}_\mu u (Q,\xi)$ for $\mu$ in $[s,s']$. The aim is to estimate $|R^{t}_{s'} u(Q)-R^{t}_s u(Q)| = |\sigma(f_{s'})-\sigma(f_s)|$.
		
		By Proposition \ref{Squad}, the family $f_\mu$ is, again, as required in Property \ref{mmax}-(\ref{loc}).
		
		If $\xi_\mu=(q_\mu,p_\mu,\nu_\mu)$ is a critical point of $f_\mu$, Proposition \ref{Sder}-\eqref{temps} gives on one hand that  $\dd_\mu S^t_\mu(Q,\xi_\mu)=H(\mu,q_\mu,p_{\mu})$ and Proposition \ref{crit} on the other hand that $\|p_\mu\|=\|du(q_\mu)\|\leq L$. 
		
	By Hypothesis \ref{esti}, we hence get that\[|\dd_\mu S^t_{\mu}(Q,\xi)| \leq C(1+L)^2.\]
		
		To sum it up, we have just proved that $\|\dd_\mu f_\mu\|\leq C(1+L)^2$ for each $\mu$ in $[s,s']$ and each critical point of $f_\mu$, hence $\mu \mapsto \sigma(f_\mu)$ is Lipschitz with constant $C(1+L)^2$ on $[s,s']$ and the result holds.

		\item Let us show that $\forall Q \in \rr^d, \left|R^t_s u(Q) - R^t_s v(Q)\right| \leq \|u-v\|_{\bar{B}\!\left(Q,(e^{C(t-s)}-1)(1+L)\right)}$.

	For $Q$ fixed, let us again define $f_\mu=S^t_s\!\left((1-\mu) u+\mu v\right)(Q,\cdot)$ for $\mu$ in $[0,1]$. The aim is to estimate $|R^{t}_s v(Q)-R^{t}_s u(Q)| = |\sigma(f_1)-\sigma(f_0)|$.
		
	By Proposition \ref{Squad}, since $(1-\mu) u+\mu v$ is $L$-Lipschitz, the family $f_\mu$ is, again, as required in Property \ref{mmax}-(\ref{loc}). Let us then estimate $\dd_\mu f_\mu$:
	\[\dd_\mu f_\mu(q,p,\nu)=v(q)-u(q).\]
	
	If $\xi_\mu=(q_\mu,p_\mu,\nu_\mu)$ is a critical point of $f_\mu$, Proposition \ref{critloc} gives that $q_\mu$ belongs to $B\!\left(Q,(e^{C(t-s)}-1)(1+L)\right)$,  
	so that $\|\dd_\mu f_\mu\|\leq \|u-v\|_{\bar{B}\!\left(Q,(e^{C(t-s)}-1)(1+L)\right)}$ for each critical point of $f_\mu$, hence the result.

	\begin{rem}\label{locmonproof1}
		The proof of the alternative Proposition \ref{locmon} is contained here: if $u\leq v$ on $B\!\left(Q,(e^{C(t-s)}-1)(1+L)\right)$, then $\dd_\mu f_\mu(q,p,\nu)=v(q)-u(q)\geq 0$ for each critical point of $f_\mu$, hence $R^t_s v(Q)-R^t_s u(Q)=\sigma(f_1)-\sigma(f_0) \geq0$.
	\end{rem}
	\end{enumerate}
		If $u$ is only Lipschitz with constant $L$, for all $\ep>0$ we may find a $\cc^1$ and $L$-Lipschitz function $u_\ep$ such that $\|u-u_\ep\|_\infty \leq \ep$, and then by weak contraction (Proposition \ref{contsc})$R^t_s u- R^t_s u_\ep$ is also bounded by $\ep$ for each $s\leq t$ . Writing the previous results for $u_\ep$ and then letting $\ep$ tend to zero gives us the wanted estimates.
		
		If $H$ is not fiberwise compactly supported, let us fix $L$, $T$, and $\delta>0$ and take a Hamiltonian $H_\delta$ in $\hh^{C(1+\delta)}_0$ that coincides with $H$ on $\rr\times\rr^d\times B\!\left(0,e^{C(1+\delta)T}(1+L)-1\right)$ as in Definition \ref{ladefinition}, so that if $u$ is $L$-Lipschitz and $0\leq s \leq t \leq T$,  $R^t_s u = R^t_{s,H_\delta} u$.
		
		The previous Lipschitz estimates, applied to $R^t_{s,H_\delta}$, give that:
		\begin{enumerate}\renewcommand{\labelitemi}{$\bullet$}
			\item $R^t_s u$ is Lipschitz with constant ${\rm Lip}(R^t_s u) \leq e^{C(1+\delta)(t-s)}(1+L)-1$,
			\item $\|R^{t'}_s u-R^t_s u\|_\infty \leq C(1+\delta)e^{2C(1+\delta)(t-s)}(1+L)^2|t'-t|$,
			\item $\left\|R^t_{s'} u(Q)-R^t_s u(Q)\right\|_\infty \leq C(1+\delta)(1+L)^2|s'-s|$,
			\item $\left|R^t_s u(Q) - R^t_s v(Q)\right| \leq \|u-v\|_{\bar{B}\!\left(Q,(e^{C(1+\delta)(t-s)}-1)(1+L)\right)}$,
		\end{enumerate}
		and we conclude the proof by letting $\delta$ tend to $0$.	
\end{proof}

Let us end this section with the analogous proof of Proposition \ref{Hlip}, which describes the dependence of the constructed operator with respect to the Hamiltonian.
\begin{proof}[Proof of Proposition \ref{Hlip}]  Let $H_0$ and $H_1$ be two $\cc^2$ Hamiltonians satisfying Hypothesis \ref{esti} with constant $C$, $u$ be a $L$-Lipschitz function, $Q$ be in $\rr^d$ and $s\leq t$. We are going to show that \[|R^t_{s,H_1} u(Q) -R^t_{s,H_0}u(Q)| \leq (t-s)\|H_1-H_0\|_{\bar{V}},\]
	where $\bar{V}=[s,t]\times \bar{B}\!\left(Q,(e^{C(t-s)}-1)(1+L)\right)\times \bar{B}\!\left(0,e^{C(t-s)}(1+L)-1\right)$.
	
	Let us first assume that $u$ is a $\cc^1$ function, and that $H_0$ and $H_1$ are fiberwise compactly supported. Let us define $H_\mu=(1-\mu)H_0 + \mu H_1$ for $\mu$ in $[0,1]$ and observe that $H_\mu$ is in $\hh^C_0$, and that there exists a $R>0$ such that $H_\mu(t,q,p)=0$ for all $\|p\|\geq R$ and all $\mu$ in $[0,1]$. Let us denote by $\phi^t_{s,\mu}=(Q^t_{s,\mu},P^t_{s,\mu})$ the Hamiltonian flow for $H_\mu$.
	
	Let us fix $Q$ and $h$ in $\rr^d$ and define $f_\mu(\xi)=S^t_{s,H_\mu} u (Q,\xi)$ for $\mu$ in $[0,1]$. The aim is to estimate $|R^t_{s,H_1} u(Q)-R^t_{s,H_0} u(Q)| = |\sigma(f_1)-\sigma(f_0)|$.
	
	Proposition \ref{Squad} states that the homotopy $f_\mu$ is of the form required in the condition \ref{mmax}-(\ref{loc}): $f_\mu(\xi)=\zz(\xi)+\ell_\mu(\xi)$, where the family $(\ell_\mu)$ is equi-Lipschitz with constant ${\rm Lip}(u)+\|Q\|+4(1+R)$. 	
	
	Let $\xi=(q,p,\nu)$ be a critical point of $f_\mu$. On the one hand, Proposition \ref{critloc}  gives that $\phi^\tau_{s,\mu}(q,p)$ is in $B\!\left(Q,(e^{C(t-s)}-1)(1+L)\right)\times  B\!\left(0,e^{C(t-s)}(1+L)-1\right)$ for all $s\leq \tau\leq t$, since $H_\mu$ satisfies Hypothesis \ref{esti} with constant $C$. On the other hand, Proposition \ref{Sder}-\eqref{SHdep} gives that
	\[\dd_\mu f_\mu(\xi)=\dd_\mu S^t_{s,H_\mu} u (Q,q,p,\nu)=-\int^t_s \dd_\mu H_{\mu}(\tau,\phi^\tau_{s,\mu}(q,p)) \; \, d\tau.\]
	
	Since $\dd_\mu H_\mu = H_1-H_0$, we have just proved that $\|\dd_\mu f_\mu\|\leq (t-s)\|H_0-H_1\|_V$ for each critical point of $f_\mu$. This implies  that $|\sigma(f_1)-\sigma(f_0)| \leq (t-s)\|H_0-H_1\|_V$ by Property \ref{mmax}-(\ref{loc}) of the selector, hence the result.

		\begin{rem}\label{locmonproof2}
			The proof of the alternative Proposition \ref{locmon} is contained here: if $H_0\leq H_1$ on $V$, then $\dd_\mu f_\mu(\xi)=-\int^t_s (H_1-H_0)(\tau,\phi^\tau_{s,\mu}(q,p))\leq 0$ for each critical point of $f_\mu$, hence $R^t_{s,H_1} u(Q) -R^t_{s,H_0}u(Q)=\sigma(f_1)-\sigma(f_0) \leq0$.
		\end{rem}

	If $u$ is only Lipschitz with constant $L$, for all $\ep>0$ we may find a $\cc^1$ and $L$-Lipschitz function $u_\ep$ such that $\|u-u_\ep\|_\infty \leq \ep$, and then by continuity (Proposition \ref{contsc})$R^t_s u- R^t_s u_\ep$ is also bounded by $\ep$ for each $s\leq t$ . Writing the previous results for $u_\ep$ and then letting $\ep$ tend to zero gives us the wanted estimates.

	If $H_0$ and $H_1$ are not fiberwise compactly supported, take $\delta >0$ and $H_{0,\delta}$ (resp.
		$H_{1,\delta}$) in $\hh^{C(1+\delta)}_0$ coinciding with $H_0$ (resp. with $H_1$) on $\rr \times \rr^d\times B\!\left(0,e^{C(1+\delta)(t-s)}(1+L)-1\right)$ as in Definition \ref{ladefinition}, so that $R^t_{s,H_{0}} u= R^t_{s,H_{0,\delta}} u$ and $R^t_{s,H_{1}} u= R^t_{s,H_{1,\delta}} u$.
		The previous work applied to $H_{0,\delta}$ and $H_{1,\delta}$ gives that
		\[\left|R^t_{s,H_1} u(Q) -R^t_{s,H_0}u(Q)\right| = \left|R^t_{s,H_{1,\delta}} u(Q) -R^t_{s,H_{0,\delta}}u(Q)\right|\leq  (t-s) \underbrace{\|H_{1,\delta}-H_{0,\delta}\|_{V_\delta}}_{=\|H_1-H_0\|_{V_\delta}},\]
		where $V_\delta=[s,t]\times B\!\left(Q,(e^{C(1+\delta)(t-s)}-1)(1+L)\right)\times B\!\left(0,e^{C(1+\delta)(t-s)}(1+L)-1\right)$. The result is then obtained by letting $\delta$ tend to $0$.
\end{proof}
	
Let us add here the considerably simpler Lipschitz estimates obtained for integrable Hamiltonians, using Remark \ref{estint} instead of Proposition \ref{critloc} in the previous proofs.
	
\begin{add}\label{varlipint}
 If $H(p)$ (resp. $\tilde{H}(p)$) satisfies Hypothesis \ref{esti} with constant $C$, then for $0\leq s \leq s' \leq t'\leq t$ and $u$ and $v$ two $L$-Lipschitz functions,		\begin{enumerate}\renewcommand{\labelitemi}{$\bullet$}
		\item $R^t_s u$ is $L$-Lipschitz,
		\item $\|R^{t'}_s u-R^t_s u\|_\infty \leq C(1+L)^2|t'-t|$,
		\item $\|R^t_{s'} u-R^t_s u\|_\infty \leq C(1+L)^2|s'-s|$,
		\item $\forall Q \in \rr^d, \left|R^t_s u(Q) - R^t_s v(Q)\right| \leq \|u-v\|_{\bar{B}\!\left(Q,C(t-s)(1+L)\right)}$,
		\item $\|R^t_{s,\tilde{H}} u -R^t_{s,H}u\|_\infty \leq (t-s)\|\tilde{H}-H\|_{\bar{B}\!\left(0,L\right)}.$
	\end{enumerate}
	where $\bar{B}\!\left(Q,r\right)$ denotes the closed ball of radius $r$ centered in $Q$ and $\|u\|_K:= sup_K |u|$.
\end{add}

\section{Iterating the variational operator}\label{iteration}

A variational operator does a priori not satisfy the Markov property \eqref{mark} of Hypotheses \ref{L-O}, and in that case it cannot coincide with the viscosity operator. Yet we may obtain the viscosity operator from the variational operator we have just constructed by iterating it along a subdivision of the time space and letting then the maximal step of the subdivision tend to zero. Doing so preserves the monotonicity, additivity, regularity and compatibility properties of the operator and the limit operator satisfies the Markov property, hence is the viscosity operator.

\subsection{Iterated operator and uniform Lipschitz estimates}\label{iterlip}

Let us recall the definition of the iterated operator. We fix a sequence of subdivisions of $[0,\infty)$ $\left((\tau^N_i)_{i\in\nn}\right)_{N \in \nn}$ such that for all $N$, $0=\tau^N_0$, $\tau^N_i \underset{i\to \infty}{\to} \infty$ and $i \mapsto \tau^N_i$ is increasing. Assume also that for all $N$, $i \mapsto \tau^N_{i+1}-\tau^N_i$ is bounded a constant $\delta_N$ such that $\delta_N$ tends to zero when $N$ tends to the infinite.

\begin{defi} Let $N$ be fixed and omitted in the notations. For $t$ in $\rr_+$, denote by $i(t)$ the unique integer such that $t$ belongs to $[\tau_{i(t)},\tau_{i(t)+1})$. Now, if $u$ is a Lipschitz function on $\rr^d$, and $0\leq s \leq t$, let us define the iterated operator at rank $N$ by\begin{displaymath}
	R^t_{s,N} u = R^{t}_{\tau_{i(t)}} R^{\tau_{i(t)}}_{\tau_{i(t)-1}}\cdots R^{\tau_{i(s)+1}}_{s} u,
	\end{displaymath}
where $R^t_s$ is any variational operator satisfying the Lipschitz estimate of Theorem \ref{lip}.
\end{defi}
Let us now sum up the Lipschitz estimates of the iterated operator: note that thanks to the semigroup form of Lipschitz constants for the non iterated operator in Theorem \ref{lip}, the new estimates do not depend on $N$.
\begin{prop}\label{ulip} Let $0\leq s \leq s' \leq t'\leq t \leq T$ and $u$ and $v$ two $L$-Lipschitz functions. The Lipschitz constants for the iterated operator are:	\begin{enumerate}\renewcommand{\labelitemi}{$\bullet$}
		\item ${\rm Lip}(R^t_{s,N} u) \leq  e^{CT}(1+L)-1$,
		\item $\|R^{t'}_{s,N} u-R^t_{s,N} u\|_\infty \leq Ce^{2CT}(1+L)^2|t'-t|$,
		\item $\|R^t_{s',N} u-R^t_{s,N} u\|_\infty \leq C(1+L)^2|s'-s|$,
		\item \label{uliploc}$\forall Q \in \rr^d, \left|R^t_{s,N} u(Q) - R^t_{s,N} v(Q)\right| \leq \|u-v\|_{\bar{B}\!\left(Q,(e^{CT}-1)(1+L)\right)}$.
	\end{enumerate}
\end{prop}

\begin{proof} This whole proof consists in exploiting the results of Theorem \ref{lip} while keeping the Lipschitz estimates independent of $N$.
	\begin{enumerate}
		\item Since ${\rm Lip}(R^t_s u)\leq e^{C(t-s)}(1+{\rm Lip}(u))-1$ and $R^t_{s,N} u= R^t_{\tau_{i(t)}}(R^{\tau_{i(t)}}_{\tau_{i(t)-1}}\cdots R^{\tau_{i(s)+1}}_{s} u)$:
		\begin{displaymath}
		\begin{split}
		{\rm Lip}(R^t_{s,N} u) &\leq e^{C(t-\tau_{i(t)})}(1+{\rm Lip}(R^{\tau_{i(t)}}_{\tau_{i(t)-1}}\cdots R^{\tau_{i(s)+1}}_{s} u))-1\\
		&\leq e^{C(t-\tau_{i(t)})}e^{C(\tau_{i(t)} - \tau_{i(t)-1})}(1+{\rm Lip}(R^{\tau_{i(t)-1}}_{\tau_{i(t)-2}}\cdots R^{\tau_{i(s)+1}}_{s} u))-1\\
		&\leq e^{C(t-\tau_{i(t)}+\tau_{i(t)}-\cdots -s)} (1+{\rm Lip} (u))-1 \\
		&\leq e^{CT}(1+L)-1.
		\end{split}
		\end{displaymath}
		
		\item Assume that $0\leq s\leq t'\leq t \leq T$. It is enough to prove the result for $|t-t'|\leq \delta_N$, and in that case either $i(t)=i(t')$, or $i(t)=i(t')+1$.
		 If $i(t)=i(t')$, then \begin{displaymath}\begin{split}
		\|R^t_{s,N} u-R^{t'}_{s,N}u\|_{\infty} & = \|R^t_{\tau_{i(t)}}\!\left(R^{\tau_{i(t)}}_{\tau_{i(t)-1}}\cdots R^{\tau_{i(s)+1}}_{s} u\right)-R^{t'}_{\tau_{i(t)}}\!\left(R^{\tau_{i(t)}}_{\tau_{i(t)-1}}\cdots R^{\tau_{i(s)+1}}_{s} u\right)\|_{\infty}\\
		&\leq Ce^{2C(t-\tau_{i(t)})}\left(1+{\rm Lip}\!\left(R^{\tau_{i(t)}}_{\tau_{i(t)-1}}\cdots R^{\tau_{i(s)+1}}_{s} u\right)\right)^2|t'-t|.
		\end{split}\end{displaymath}
		Now since $1+{\rm Lip}\!\left(R^{\tau_{i(t)}}_{\tau_{i(t)-1}}\cdots R^{\tau_{i(s)+1}}_{s} u\right) \leq e^{C(\tau_{i(t)}-s)}(1+L)$,\[
		\|R^t_{s,N} u-R^{t'}_{s,N}u\|_{\infty}\leq Ce^{2C(t-s)}(1+L)^2|t-t'|\leq Ce^{2CT}(1+L)^2|t-t'|.\]
		Else, assume that $i(t)=i(t')+1$. Then \[
		\|R^t_{s,N} u-R^{t'}_{s,N}u\|_{\infty}= \|R^t_{s,N} u-R^{\tau_{i(t)}}_{s,N}u + R^{\tau_{i(t)}}_{\tau_{i(t)-1}}\cdots R^{\tau_{i(s)+1}}_{s} u-R^{t'}_{\tau_{i(t)-1}}\cdots R^{\tau_{i(s)+1}}_{s} u\|_{\infty}
		\]
		and we may use the previous case to estimate both quantities:\begin{displaymath}\begin{split}
		\|R^t_{s,N} u-R^{t'}_{s,N}u\|_{\infty}&\leq Ce^{2C(t-s)}(1+L)^2|t-\tau_{i(t)}|+Ce^{2C(t-s)}(1+L)^2|\tau_{i(t)}-t'| \\
		& \leq Ce^{2C(t-s)}(1+L)^2 |t-t'|\leq Ce^{2CT}(1+L)^2|t-t'|
		\end{split}\end{displaymath}
		since in that case $t' \leq \tau_{i(t)}\leq t$.
		
		\item Again, it is enough to prove the result for $|s-s'|\leq \delta_N$. We freely use a consequence of the estimate proved in the next point:
		\[\|R^t_{s,N} u -R^t_{s,N}v\|_\infty \leq \|u-v\|_\infty \]
		If $i(s')=i(s)$, 
		 \begin{displaymath}\begin{split}
		 \|R^t_{s,N} u-R^t_{s',N}u\|_{\infty} & = \|R^t_{\tau_{i(s)+1},N} R^{\tau_{i(s)+1}}_{s} u-R^{t}_{\tau_{i(s)+1},N} R^{\tau_{i(s)+1}}_{s'} u\|_{\infty}\\
		 &\leq \|R^{\tau_{i(s)+1}}_{s'} u-R^{\tau_{i(s)+1}}_{s} u\|_\infty \leq C(1+L)^2|s-s'|.
		 \end{split}\end{displaymath}
		 		
		If $i(s')=i(s)+1$,
		\begin{displaymath}\begin{split}
		\|R^t_{s',N} u-R^t_{s,N}u\|_{\infty} & \leq  \|R^t_{s',N} u-R^t_{\tau_{i(s')},N}u\|_{\infty}+\|R^t_{\tau_{i(s')},N}u-R^t_{s,N}u\|_{\infty}\\
		&\leq C(1+L)^2\left((s'-i(s'))+(i(s')-s)\right) \leq C(1+L)^2|s-s'|.
		\end{split}\end{displaymath}
		
		\item \label{ppp} Let $Q$ be fixed. Note that $R^{\tau_{i(t)}}_{\tau_{i(t)-1}}\cdots R^{\tau_{i(s)+1}}_{s} u$ and $R^{\tau_{i(t)}}_{\tau_{i(t)-1}}\cdots R^{\tau_{i(s)+1}}_{s} v$ are both Lipschitz with constant $(e^{C(\tau_{i(t)}-s)}(1+L)-1)$.
		Then\begin{displaymath}\begin{split}
		|R^t_{s,N} & u(Q)-R^{t}_{s,N}v(Q)| \\& = |R^t_{\tau_{i(t)}}\!\left(R^{\tau_{i(t)}}_{\tau_{i(t)-1}}\cdots R^{\tau_{i(s)+1}}_{s} u\right)(Q)-R^{t}_{\tau_{i(t)}}\!\left(R^{\tau_{i(t)}}_{\tau_{i(t)-1}}\cdots R^{\tau_{i(s)+1}}_{s} v\right)(Q)|\\
		& \leq \|R^{\tau_{i(t)}}_{\tau_{i(t)-1}}\cdots R^{\tau_{i(s)+1}}_{s} u-R^{\tau_{i(t)}}_{\tau_{i(t)-1}}\cdots R^{\tau_{i(s)+1}}_{s} v\|_{\bar{B}\!\left(Q,(e^{C(t-\tau_{i(t)})}-1)e^{C(\tau_{i(t)}-s)}(1+L))\right)}.\end{split}\end{displaymath}
		Estimating the Lipschitz constant of $R^{\tau_{i(t)-1}}_{\tau_{i(t)-2}}\cdots R^{\tau_{i(s)+1}}_{s} u$ and $R^{\tau_{i(t)-1}}_{\tau_{i(t)-2}}\cdots R^{\tau_{i(s)+1}}_{s} v$ gives the next step:
		\begin{displaymath}\begin{split}
		|R^t_{s,N} u(Q)-& R^{t}_{s,N}v(Q)|\\
		& \leq \|R^{\tau_{i(t)-1}}_{\tau_{i(t)-2}}\cdots R^\cdot_{s} u-R^{\tau_{i(t)-1}}_{\tau_{i(t)-2}}\cdots R^\cdot_{s} v\|_{\bar{B}\!\left(Q,(e^{C(t-s)}-e^{C(\tau_{i(t)-1}-s)})(1+L))\right)}\\
		&\;\;\;\;\;\;\,\,\,\,\,\,\leq \cdots \leq \|u-v\|_{\bar{B}\!\left(Q,(e^{C(t-s)}-1)(1+L))\right)}.\end{split}\end{displaymath}
	\end{enumerate}
\end{proof}

Let us gather the Lipschitz dependence in $s$ and $t$ to obtain an estimation of how non-Markov the iterated operator is:

\begin{prop}\label{nonmarkov}
	Take $0\leq s \leq r \leq t\leq T$ and $u$ $L$-Lipschitz. Then for all integer $N$,
	\[\|R^t_{s,N}u - R^t_{r,N}R^r_{s,N} u\|_\infty \leq 2C e^{2CT}(1+L)^2 \delta_N\]  
	where $\delta_N$ is the upper bound of $i \mapsto \tau^N_{i+1}-\tau^N_i$.
\end{prop}

\begin{proof}
	Let us first show that if $s\leq r \leq t$, then
	\[\|R^t_s u - R^t_r R^r_s u\|_\infty \leq 2C e^{2C(t-s)}\left(1+{\rm Lip}(u)\right)^2|r-s|\]
	for each Lipschitz function $u$. Since $R^s_s u = u$, we might write
	\[\begin{split}
	\|R^t_s u - R^t_r R^r_s u\|_\infty &\leq \|R^t_s u -R^t_r u\|_\infty 
	+ \|R^t_r R^s_s u-R^t_r R^r_s u\|_\infty\\	
	& \leq C (1+{\rm Lip}(u))^2|r-s| + \| R^s_s u- R^r_s u\|_\infty\\
	& \leq C \left(1+{\rm Lip}(u)\right)^2|r-s| + C e^{2C(r-s)}\left(1+{\rm Lip}(u)\right)^2|r-s|\\
	& \leq C(1+e^{2C(t-s)})\left(1+{\rm Lip}(u)\right)^2|r-s|\\
	& \leq 2Ce^{2C(t-s)}\left(1+{\rm Lip}(u)\right)^2|r-s|.
	\end{split}\]
	The second line is obtained by applying the Lipschitz estimates w.r.t. $s$ and $u$ of Theorem \ref{lip}, the third line by applying the Lipschitz estimate w.r.t. $t$ (same Theorem).
	
	Now, let us fix $N$ and estimate $\|R^t_{s,N}u - R^t_{r,N}R^r_{s,N} u\|_\infty$. The fourth point of Proposition \ref{ulip} implies that
	\[\begin{split}
		\|R^t_{s,N}u -& R^t_{r,N}R^r_{s,N} u\|_\infty \leq \|R^{\tau_{i(r)+1}}_{s,N}u  -R^{\tau_{i(r)+1}}_r R^r_{s,N} u\|_\infty \\
		& \leq \|R^{\tau_{i(r)+1}}_{\tau_{i(r)}}R^{\tau_{i(r)}}_{s,N} u -R^{\tau_{i(r)+1}}_{r} R^r_{\tau_{i(r)}} R^{\tau_{i(r)}}_{s,N}  u\|_\infty.\end{split}\]
		Using the previous result gives that 
			\[
			\|R^t_{s,N}u - R^t_{r,N}R^r_{s,N} u\|_\infty
		 \leq 	2C e^{2C(\tau_{i(r)+1}-\tau_{i(r)})}\left(1+{\rm Lip}(R^{\tau_{i(r)}}_{s,N}u)\right)^2|r-\tau_{i(r)}| \]
		 and since $\left(1+{\rm Lip}(R^{\tau_{i(r)}}_{s,N}u)\right)^2 \leq e^{2C(\tau_{i(r)}-s)}\left(1+{\rm Lip}(u)\right)^2$, we get
			\[
			\|R^t_{s,N}u - R^t_{r,N}R^r_{s,N} u\|_\infty \leq 2C e^{2C(\tau_{i(r)+1}-s)}\left(1+{\rm Lip}(u)\right)^2|r-\tau_{i(r)}|.\]
	Then the result comes by using the definition of $\delta_N$.
\end{proof}

Let us add a word on the dependence with respect to $H$, extending Proposition \ref{Hlip}:
\begin{prop}
	Let $H_0$ and $H_1$ be two $\cc^2$ Hamiltonians satisfying Hypothesis \ref{esti} with constant $C$, $u$ be a $L$-Lipschitz function, $Q$ be in $\rr^d$ and $s\leq t$. Then
	\[|{R}^t_{s,H_1,N} u(Q) -{R}^t_{s,H_0,N}u(Q)| \leq (t-s)\|H_1-H_0\|_{\bar{V}},\]
	where $\bar{V}=[s,t]\times \bar{B}\!\left(Q,(e^{C(t-s)}-1)(1+L)\right)\times \bar{B}\!\left(0,e^{C(t-s)}(1+L)-1\right)$.
\end{prop}
\begin{proof}  To lighten the notation, let us prove that for the non iterated operator,	\[\begin{split}
	|{R}^t_{\tau,H_1}{R}^{\tau}_{s,H_1} & u(Q) -{R}^t_{\tau,H_0}{R}^{\tau}_{s,H_0}u(Q)| \\&\leq (t-s)\|H_1-H_0\|_{[s,t]\times\bar{B}\!\left(Q,(e^{C(t-s)}-1)(1+L)\right)\times \bar{B}\!\left(0,e^{C(t-s)}(1+L)-1\right)}.
	\end{split}\]
	The result is then obtained for the iterated operator by induction on the number of steps between $s$ and $t$.
	
	For both $H_0$ and $H_1$, $1+{\rm Lip}(R^{\tau}_{s}u)\leq e^{C(\tau-s)}(1+L)$ by Theorem \ref{lip}. Hence, on the one hand, Proposition \ref{Hlip} gives that
	\[\begin{split}
	|{R}^t_{\tau,H_1}&{R}^{\tau}_{s,H_1}  u(Q) -{R}^t_{\tau,H_0}{R}^{\tau}_{s,H_1}u(Q)| \\&\leq (t-\tau)\|H_1-H_0\|_{[\tau,t]\times\bar{B}\!\left(Q,(e^{C(t-\tau)}-1)e^{C(\tau-s)}(1+L)\right)\times \bar{B}\!\left(0,e^{C(t-\tau)}e^{C(\tau-s)}(1+L)-1\right)}\\
	& \leq (t-\tau)\|H_1-H_0\|_{\bar{V}}.
	\end{split}\]	
	On the other hand, using the Lipschitz estimate with respect to $u$ of Theorem \ref{lip},
	\[|{R}^t_{\tau,H_0}{R}^{\tau}_{s,H_1}  u(Q) -{R}^t_{\tau,H_0}{R}^{\tau}_{s,H_0}u(Q)| \leq \|{R}^{\tau}_{s,H_1}  u-{R}^{\tau}_{s,H_0}u\|_{\bar{B}\!\left(Q,(e^{C(t-s)}-1)e^{C(\tau-s)}(1+L)\right)}\]		
	Proposition \ref{Hlip} gives that for each $q$ of $\bar{B}\!\left(Q,(e^{C(t-s)}-1)e^{C(\tau-s)}(1+L)\right)$,
	\[|{R}^{\tau}_{s,H_1}  u(q)-{R}^{\tau}_{s,H_0}u(q)| \leq (\tau-s)\|H_1-H_0\|_{[s,\tau]\times \bar{B}\!\left(q,(e^{C(\tau-s)}-1)(1+L)\right)\times \bar{B}\!\left(0,e^{C(\tau-s)}(1+L)-1\right)},\]
	and then summing up the radius of the balls gives
		\[\begin{split}
		|{R}^t_{\tau,H_0}&{R}^{\tau}_{s,H_1}  u(Q) -{R}^t_{\tau,H_0}{R}^{\tau}_{s,H_0}u(Q)| \\&\leq (\tau-s)\|H_1-H_0\|_{[s,\tau]\times\bar{B}\!\left(Q,(e^{C(t-s)}-1)(1+L)\right)\times \bar{B}\!\left(0,e^{C(\tau-s)}(1+L)-1\right)}\\
		& \leq (\tau-s)\|H_1-H_0\|_{\bar{V}}.
		\end{split}\]
		Summing up the two estimates concludes the proof.
\end{proof}

\subsection{Convergence towards the viscosity operator}\label{AA}
In this section we prove that the iterated operator sequence $(R^t_{s,N})_N$ converges to a limit operator when the maximal step of the subdivision tends to $0$. To do so, we first use a compactness argument to get a converging subsequence (Theorem \ref{sscv}), then show that the limit of such a subsequence is the viscosity operator (Proposition \ref{work}) and finally prove Theorem \ref{cc} with the uniqueness of this operator.

\begin{defi} Let $\|\cdot\|_{Lip}$ be the norm on the sets of real-valued Lipschitz functions on $\rr^d$ given by \begin{displaymath}
	\|u\|_{Lip}=|u(0)|+ {\rm Lip}(u).
	\end{displaymath} 
\end{defi}

\begin{defi} We denote by $\mathcal{L}^L(K)$ the set of Lipschitz functions on $\rr^d$ supported by the compact set $K$ and with Lipschitz norm $\|\cdot\|_{Lip}$ bounded by the constant $L$: \begin{displaymath}
	\mathcal{L}^L(K) = \left\{ u \in C^{0,1}(\rr^d,\rr) \Big| \begin{array}{l}
	supp(u) \subset K\\
	\|u\|_{Lip}\leq L \end{array}\right\}
	\end{displaymath}
\end{defi}
\begin{prop} The set $\mathcal{L}^L(K)$ is a compact set for the uniform norm.
\end{prop}
\begin{proof} The Arzelà-Ascoli theorem immediately gives that the closure of $\mathcal{L}^L(K)$ is compact. Then, it is easy to check that $\mathcal{L}^L(K)$ is closed. Hence, it is compact.
\end{proof}

\begin{prop}\label{arg}
	For each $T>0$, $R>0$, $L>0$, the family $\left\{(s,t,Q,u) \mapsto R^t_{s,N} u (Q)\right\}_{N}$ is equi-Lipschitz on the set $\left\{0\leq s \leq t \leq T\right\}\times \bar{B}(0,R) \times \mathcal{L}^L(\bar{B}(0,R))$.
\end{prop}

\begin{proof}
	It is enough to observe that the Lipschitz constants obtained in Proposition \ref{ulip} depend only on $T$, $R$, $L$, and that if $u$ and $v$ are compactly supported Lipschitz functions,\[\|R^t_{s,N} u-R^t_{s,N} v\|\leq \|u-v\|_\infty.\]
\end{proof}

\begin{theo}\label{sscv} There exists a subsequence $N_k$ such that for all $0\leq s \leq t$, $Q \in \rr^d$, $u$ Lipschitz function on $\rr^d$, $R^t_{s,N_k}u(Q)$ has a limit when $k$ tends to $\infty$, denoted $\bar{R}^t_su (Q)$. Furthermore, the sequence of functions $\left\{(s,t,Q) \mapsto R^t_{s,N_k}u(Q)\right\}_k$ converges uniformly towards $(s,t,Q)\mapsto \bar{R}^t_s u(Q)$ on every compact subset of $\left\lbrace 0\leq s \leq t\right\rbrace \times \rr^d$. 
\end{theo}

\begin{proof} The first step consists in applying Arzelà-Ascoli theorem with $(s,t,Q,u)$ living in the compact set $\{0\leq s \leq t \leq T\} \times \bar{B}(0,R) \times \mathcal{L}^L\!\left(\bar{B}(0,R)\right)$, where $T$, $R$ and $L$ are fixed. The second step is to get a subsequence working for all $T$, $R$ and $L$. The third step consists in extending the result to Lipschitz functions which are not compactly supported.\\	
	$ $\\
	\emph{First step.} Since Proposition \ref{arg} gives that $\left\{(s,t,Q,u) \mapsto R^t_{s,N}v(Q)\right\}_N$ is equi-Lipschitz on $\{0\leq s \leq t \leq T\} \times \bar{B}(0,R) \times \mathcal{L}^L\!\left(\bar{B}(0,R+CT)\right)$, it is enough to prove that it is uniformly bounded at one point - for example $(s,s,Q,0)$ - to gather all the conditions required to apply Arzelà-Ascoli theorem. \begin{displaymath}
	|R^s_{s,N} 0(Q)|=|0(Q)|= 0,
	\end{displaymath}
	hence, there exists a subsequence $N_k$ (a priori depending on $T$, $R$ and $L$) such that the sequence $\left\{(s,t,Q,u) \mapsto R^t_{s,N_k}u(Q)\right\}_k$ converges uniformly to a limit $(s,t,Q,u) \mapsto \bar{R}^t_s u(Q)$ on the compact set $\{0\leq s \leq t \leq T\} \times \bar{B}(0,R) \times \mathcal{L}^L\!\left(\bar{B}(0,R)\right)$.
	
	$ $\\
	\emph{Second step.} In this paragraph we will describe a subsequence by the diagonal process. Note that the first step also applies on every subsequence of $(R^t_{s,N})_N$.

	Let $T_i=R_i=L_i=i$ for each integer $i$.  
	
	For $i=1$, let $\psi^1$ be the subsequence given by the Arzelà-Ascoli theorem for the sequence $(R^t_{s,N})_{N\in \nn}$ and the constants $T_1=L_1=R_1=1$.
	
	For $i>1$, let $\psi^i$ be the subsequence given by the Arzelà-Ascoli theorem for the sequence $(R^t_{s,\psi^{i-1}(N)})_{N\in\nn}$ and the constants $T_i=L_i=R_i=i$.
	
	Now define the diagonal subsequence $N_k=\psi^k(k)$: for all $k$, $(N_i)_{i\geq k}$ is extracted from $\psi^k$.
	
	For each $T$, $R$, $L$, there exists $i$ such that $T\leq i$, $R\leq i$ and $L\leq i$. Since $R^t_{s,\psi^i(k)}$ converges on $\{0\leq s \leq t \leq i\} \times \bar{B}(0,i) \times \mathcal{L}^i\!\left(\bar{B}(0,i)\right)$, it converges on $\{0\leq s \leq t \leq T\} \times \bar{B}(0,R) \times \mathcal{L}^L\!\left(\bar{B}(0,R)\right)$, and so does $R^t_{s,N_k}$ since $N_k$ is a subsequence of $\psi^i(k)$. Hence we have constructed a subsequence that works for all $L$, $R$, $T$ positive constants. If $\mathcal{L}_c$ denotes the set of compactly supported Lipschitz functions,\[
	\bigcup_{T,L,R} \{0\leq s \leq t \leq T\} \times \bar{B}(0,R) \times \mathcal{L}^L\!\left(\bar{B}(0,R)\right) = \{0\leq s \leq t\}\times \rr^d \times \mathcal{L}_c,
	\]
	and the subsequence we have constructed converges for all $s\leq t$, $Q\in \rr^d$ and $u$ compactly supported Lipschitz function.
	
	$ $\\
	\emph{Third step.} Now take $T$ and $R$ two constants and $u$ a Lipschitz function on $\rr^d$, with Lipschitz constant $L$. For all $\bar{L}> L$, we build a compactly supported $\bar{L}$-Lipschitz function $\bar{u}$ such that $\bar{u}=u$ on $\bar{B}\!\left(0,R +(e^{CT}-1)(1+\bar{L})\right)$: to do so, let us take a compactly supported $\cc^1$ function $\phi:\rr^+\to [0,1]$ such that \[\left\{\begin{array}{l}
	\phi = 1 \textrm{ on } [0,R+(e^{CT}-1)(1+\bar{L})],\\	|\phi'(x)|\leq \frac{L'-L}{|u(0)|+Lx} \;\; \forall x \geq 0,
	\end{array}\right.\] and $\bar{u}(q)=\phi(\|q\|)\cdot u(q)$. 
	
	If $u$ is $\cc^1$, so is $\bar{u}$, and since $\|d_q(\phi(\|q\|))\| =|\phi'(\|q\|)|\leq \frac{L'-L}{|u(0)|+L \|q\|}$, the differential of $\bar{u}$ is bounded by $\bar{L}$:
	\[
	\|d\bar{u}(q)\|\leq \underbrace{\|d_q(\phi(\|q\|))\| \cdot |u(q)|}_{\leq \bar{L}-L }+\underbrace{|\phi(q)|}_{\leq 1}\cdot \underbrace{\|du(q)\|}_{\leq L}\leq \bar{L}.
	\] 
	If $u$ is not $\cc^1$, one can show that $\bar{u}$ is $\bar{L}$-Lipschitz by applying the mean value theorem to $\phi$. 
	
	For all $Q$ in the ball $\bar{B}(0,R)$, since $u$ and $\bar{u}$ are $\bar{L}$-Lipschitz and coincide on the ball centered in $Q$ of radius $(e^{CT}-1)(1+\bar{L})$ , the Lipschitz property \ref{ulip}-\eqref{uliploc} gives\begin{displaymath}
	R^t_{s,N} \bar{u}(Q) = R^t_{s,N} u(Q) \; \; \forall N \in \nn, \forall \; 0\leq s \leq t \leq T.
	\end{displaymath}
	Since $\bar{u}$ is a compactly supported function, $\left\{(s,t,Q) \mapsto R^t_{s,N_k} \bar{u} (Q)\right\}_k$ uniformly converges on $\{0\leq s \leq t \leq T\}\times \bar{B}(0,R)$, and thus the same holds for $\left\{(s,t,Q) \mapsto R^t_{s,N_k} u (Q)\right\}_k$.
\end{proof}

\begin{prop}\label{work} The limit operator $\bar{R}^t_s$ is the viscosity operator: $\bar{R}^t_s=V^t_s$.
\end{prop}

\begin{proof}\begin{enumerate}
		\item Monotonicity property follows from the monotonicity of $R^t_s$, for $s\leq t$.
		\item Same thing for the additivity property.
		\item Regularity: since the convergence of $\left\{(s,t,Q) \mapsto R^t_{s,N_k} v (Q)\right\}_k$ is uniform on every compact subset of $\left\lbrace 0\leq s \leq t\right\rbrace \times \rr^d$, and the family is equi-Lipschitz in time and space, the limit satisfies that  $\left\lbrace \bar{R}^t_\tau u, t\in [\tau,T] \right\rbrace$ is uniformly Lipschitz for each $\tau \leq T$ and $(t,q) \mapsto \bar{R}^t_\tau u(q)$ is locally Lipschitz on $(\tau,\infty)\times \rr^d$.
		\item Compatibility with Hamilton-Jacobi equation: Remark \ref{comprop} and Proposition \ref{vervar} give the compatibility property for the operator $R^t_s$. Hence if $u$ is a Lipschitz $\cc^2$ solution of the Hamilton-Jacobi equation, for all $N$:\begin{displaymath}
		R^t_{s,N} u_s = R^{t}_{\tau_{i(t)}} \cdots  \underbrace{R^{\tau_{i(s)+1}}_{s} u_s}_{=u_{\tau_{i(s)+1}}}
		=R^{t}_{\tau_{i(t)}} u_{\tau_{i(t)}} = u_t,
		\end{displaymath}
		and the limit satisfies $\bar{R}^t_s u_s = u_t$.
		\item Markov property: take $u$ Lipschitz, and $0\leq s \leq \tau \leq t\leq T$. Let us show the equality $\bar{R}^t_\tau  \circ \bar{R}^\tau_s u = \bar{R}^t_s u$. Let $Q$ be fixed in $\rr^d$.
		
		Since $Q \mapsto \bar{R}^\tau_s u(Q)$ is Lipschitz, $\left( R^t_{\tau,N_k} \bar{R}^\tau_s u(Q) \right)_k$ converges to $\bar{R}^t_\tau \bar{R}^\tau_s u (Q)$.\\

		Let us first show that $R^t_{\tau,N_k} R^\tau_{s,N_k} u(Q)$ tends to $\bar{R}^t_\tau \bar{R}^\tau_s u (Q)$.
		
		\begin{displaymath}\begin{split}
		\left|R^t_{\tau,N_k} R^\tau_{s,N_k} u(Q) - \bar{R}^t_\tau \bar{R}^\tau_s u (Q) \right|  \leq	 \left|R^t_{\tau,N_k}  R^\tau_{s,N_k} u(Q)  - R^t_{\tau,N_k} \bar{R}^\tau_s u (Q) \right|&\\
		+ \underbrace{\left|R^t_{\tau,N_k} \bar{R}^\tau_s u (Q) - \bar{R}^t_\tau \bar{R}^\tau_s u (Q) \right|}_{\to 0 \; }&.
		\end{split}
		\end{displaymath}
		
		Now, the uniform Lipschitz estimates of property \ref{ulip}-(\ref{uliploc}) give \begin{displaymath}
		\left|R^t_{\tau,N_k} R^\tau_{s,N_k} u(Q) - R^t_{\tau,N_k} \bar{R}^\tau_s u (Q) \right| \leq \|R^\tau_{s,N_k} u -  \bar{R}^\tau_s u\|_{\bar{B}(Q,r)}
		\end{displaymath}
		for some radius $r$ depending only on $C$, $T$, $L$; as the convergence is uniform on every compact subset of $\rr^d$, the right hand side tends to $0$ when $k$ tends to $\infty$.
		
		Now, since $\delta_{N_k}\underset{k\to \infty}{\to}0$,  Proposition \ref{nonmarkov} implies that  $R^t_{\tau,{N_k}}R^\tau_{s,{N_k}} u(Q)$ and $R^t_{s,{N_k}}u(Q)$ have the same limit, hence the conclusion:
		\begin{displaymath}
		\bar{R}^t_s u(Q)  =  \bar{R}^t_\tau \bar{R}^\tau_s u (Q).
		\end{displaymath}
	\end{enumerate}
\end{proof}

\begin{csq} We have proved, for every Hamiltonian satisfying Hypothesis \ref{esti}, that the viscosity operator exists. In particular, for such a Hamiltonian and for a Lipschitz initial condition, there exists a viscosity solution of \eqref{HJ} on $(0,\infty)\times \rr^d$ that coincides with the initial condition at time $0$, see Proposition \ref{key}.
\end{csq}

\begin{proof}[Proof of Theorem \ref{cc}] Since every subsequence of $R^t_{s,N}u$ admits a subsequence uniformly converging to the viscosity solution $V^t_s u$ on every compact set, the whole family $(R^t_{s,N}u)_N$ converge to $V^t_s u$	by uniqueness of the viscosity solution.
\end{proof}

The local Lipschitz estimates on the viscosity operator $V$ and the local monotonicity properties stated in Proposition \ref{visclip} are directly deduced from this uniform convergence and the estimates on the variational operator $R$. In the integrable case, the iterated operator $R^t_{s,N}$ satisfies the same Lipschitz estimate than the variational operator $R^t_s$ (see Addendum \ref{varlipint}), whence the following result.

\begin{add}\label{visclipint}
	If $H(p)$ (resp. $\tilde{H}(p)$) satisfies Hypothesis \ref{esti} with constant $C$, then for $0\leq s \leq s' \leq t'\leq t$ and $u$ and $v$ two $L$-Lipschitz functions,		\begin{enumerate}\renewcommand{\labelitemi}{$\bullet$}
		\item $V^t_s u$ is $L$-Lipschitz,
		\item $\|V^{t'}_s u-V^t_s u\|_\infty \leq C(1+L)^2|t'-t|$,
		\item $\|V^t_{s'} u-V^t_s u\|_\infty \leq C(1+L)^2|s'-s|$,
		\item $\forall Q \in \rr^d, \left|V^t_s u(Q) - V^t_s v(Q)\right| \leq \|u-v\|_{\bar{B}\!\left(Q,C(t-s)(1+L)\right)}$,
		\item $\|V^t_{s,\tilde{H}} u -V^t_{s,H}u\|_\infty \leq (t-s)\|\tilde{H}-H\|_{\bar{B}\!\left(0,L\right)}.$
	\end{enumerate}
	where $\bar{B}\!\left(Q,r\right)$ denotes the closed ball of radius $r$ centered in $Q$ and $\|u\|_K:= sup_K |u|$.
\end{add}

\section{The convex case}\label{convex}
The purpose of this chapter is to prove Theorem \ref{jouk}, that states in particular that for strictly convex Hamiltonians, the variational operator constructed in this paper coincides with the Lax-Oleinik semigroup.
To do so, we give a description of the Lax-Oleinik semigroup in terms of broken geodesics, and discuss the link between the so-called \emph{Lagrangian generating family} involved in this description and the generating family used for general Hamiltonians.
\subsection{The Lax-Oleinik semigroup with broken geodesics}\label{chapgencv}
The Lax-Oleinik semigroup defined by the equation \eqref{defLO} in the introduction may also be written as a finite dimensional optimization problem.  If $H$ is strictly uniformly convex w.r.t. $p$ and satisfies Hypothesis \ref{esti}, we fix $\delta_2>0$ such that $(q,p)\mapsto(q,Q^t_s(q,p))$ is a $\cc^1$-diffeomorphism for each $|t-s|\leq \delta_2$ (see Proposition \ref{twist}).
\begin{prop} If $s=t_0\leq t_1\leq \cdots \leq t_N=t$ is a subdivision such that $t_{i+1}-t_i<\delta_2$ for all $i$, then 
	\[T^t_s u(Q)= \min_{q,Q_0,\cdots,Q_{N-1}} \aaa^t_s u(Q,q,Q_0,\cdots,Q_{N-1}),\]
	with the Lagrangian generating family $\aaa$ defined by
	\[\aaa^t_s u(Q,q,Q_0,\cdots,Q_{N-1})= u(q)+ \sum_{i=0}^{N} \int_{t_i}^{t_{i+1}} L\left(\tau,Q_{t_i}^\tau(Q_{i-1},p_i),\dd_\tau Q_{t_i}^\tau(Q_{i-1},p_i)\right) d\tau\]
where $p_i$ is uniquely defined by $Q^{t_{i+1}}_{t_i}(Q_{i-1},p_i)=Q_i$ and while denoting $q=Q_{-1}$ and $Q=Q_N$.
\end{prop}
A proof of this statement can be found in \cite{bernard}, Lemma 48 and Proposition 49.

The two next propositions gather properties of the Lagrangian generating family $\aaa$.
\begin{prop}\label{A=maxS}
	If $H$ is uniformly strictly convex w.r.t. $p$, for $\delta_2$ small enough, \[\aaa^t_s u(Q,q,Q_0,\cdots,Q_{N-1})= \max_{p,p_1,\cdots, p_N} S^t_s u(Q,q,p,Q_0,\cdots,p_N).\]
\end{prop}
\begin{proof}
	This is a direct consequence of Proposition \ref{famcvx}, since by definition \[
	\aaa^t_s u(Q,q,Q_0,\cdots,Q_{N-1}) = u(q) + A^t_s(q,Q_0,\cdots, Q)\]
	and
	\[ S^t_s u( Q,q,p,Q_0,\cdots,p_N)= u(q) + G^t_s(p,Q_0,\cdots,p_N,Q)+p\cdot(Q-q),
	\]
	with the notations of Appendix \ref{chap}.
\end{proof}
\begin{prop}\label{AQm}
	If $H$ satisfies Hypothesis \ref{esti} with constant $C$, is uniformly strictly convex w.r.t. $p$ and $H(t,q,p)=\frac{\|p\|^2}{2}$ outside of a band $\rr\times \rr^d \times B(0,R)$, then the function $(q,Q_0,\cdots,Q_{N-1}) \mapsto \aaa^t_s u(Q,q,Q_0,\cdots,Q_{N-1})$ is coercive and in some $\mathcal{Q}_{m}$.
\end{prop}
\begin{proof}
	We are first going to prove the result for $H(t,q,p)=\frac{\|p\|^2}{2}$. In that case, $L(t,q,v)=\frac{\|v\|^2}{2}$ and $Q^{t_{i+1}}_{t_i}(Q_{i-1},p_i)=Q_i$ if and only if $Q_i=Q_{i-1}+(t_{i+1}-t_i) p_i$.
	Thus \[\begin{split}
\aaa^t_s u(Q,q,Q_0,\cdots,Q_{N-1})= u(q)+ \sum_{i=0}^{N} \int_{t_i}^{t_{i+1}} L&\left(\tau,Q_{t_i}^\tau(Q_{i-1},p_i),\dd_\tau Q_{t_i}^\tau(Q_{i-1},p_i)\right) d\tau\\
 = u(q)+ \frac{1}{2}\sum_{i=0}^{N} \int_{t_i}^{t_{i+1}} \frac{\|Q_i -Q_{i-1}\|^2}{(t_{i+1}-t_i)^2} d\tau
& = u(q) + \frac{1}{2} \sum_{i=0}^{N} \frac{\|Q_i -Q_{i-1}\|^2}{t_{i+1}-t_i}
	\end{split}\]
	always denoting $q=Q_{-1}$ and $Q=Q_N$. To see that the considered function is coercive and in some $\mathcal{Q}_m$, we may then use the affine diffeomorphism $(q,Q_0,\cdots, Q_{N-1})\mapsto (\frac{Q_0-q}{\sqrt{t_1-s}},\frac{Q_1-Q_0}{\sqrt{t_2-t_1}},\cdots,\frac{Q-Q_{N-1}}{\sqrt{t-t_N}})$. 
	
	Now, if $H(t,q,p)=\frac{\|p\|^2}{2}$ outside of a band $\rr\times \rr^d \times B(0,R)$, and if $\tilde{H}$ denotes the quadratic form $\tilde{H}(p)=\frac{\|p\|^2}{2}$, $H$ and $\tilde{H}$ satisfy the hypotheses of Proposition \ref{lbnl2} with constants $C$ and $K=C(1+R)$, and thus $\tilde{\aaa}^t_s u - \aaa^t_s u= \tilde{A}^t_s - A^t_s$ is a Lipschitz function of $(q,Q_0,\cdots, Q_{N-1})$. The previous part hence proves that the function $(q,Q_0,\cdots, Q_{N-1})\mapsto \aaa^t_s u(Q,q,Q_0,\cdots,Q_{N-1})$ is coercive and in some $\mathcal{Q}_m$.
\end{proof}

\subsection{Proof of Joukovskaia's Theorem}
 To prove that the variational operator $R^t_s$ constructed in Chapter \ref{buildvar} is the viscosity operator, it is enough to prove that it satisfies the Markov property \eqref{mark}, see Remark \ref{comprop}. In that purpose, we need the critical value selector to satisfy the two additional following properties - properties that are actually satisfied by the minmax constructed in Appendix \ref{minmax}. 
	\begin{prop}\label{mmaxbis}  There exists a critical value selector $\sigma:\bigcup_{m\in \nn}\mathcal{Q}_m\to \rr$, as defined in Proposition \ref{mmax}, that satisfies:
		\begin{enumerate}
			\item \label{opp} $\sigma(-f)=-\sigma(f)$,
			\item \label{cvx} if $f(x,y)$ is a $\cc^2$ function of $\mathcal{Q}_m$ such that $\dd^2_y f \geq c\rm{id}$ for some $c>0$, and if $g$ defined by $g(x)=\min_y f(x,y)$ is in some $\mathcal{Q}_{\tilde{m}}$, then $\sigma(g)=\sigma(f)$.
		\end{enumerate}
	\end{prop}
We assume $\sigma$ to be such a critical value selector.\\ 
\begin{proof}[Proof of Theorem \ref{jouk}]
\emph{First step.} We assume that the Hamiltonian $H$ is uniformly strictly convex w.r.t. $p$ ($\dd^2_p H \geq m \textrm{id}$), satisfies Hypothesis \ref{esti} with some constant $C$ and coincides with the quadratic form $Z(p)=\|p\|^2$ outside of a band $\rr \times \rr^d\times B(0,R)$. Then the variational operator constructed in Chapter \ref{buildvar} is the Lax-Oleinik operator: $R^t_s=T^t_s$.  

To see this, we apply the last item to the function $f(x,y)=S^t_s u(Q,q,p,Q_0,\cdots,p_N)$ where $x=(q,Q_0,Q_1,\cdots,Q_{N-1})$ and $y=(p,\cdots,p_N)$.  Proposition \ref{Gccve} gives, since  $S^t_s u( Q,q,p,Q_0,\cdots,p_N)= u(q) + G^t_s(p,Q_0,\cdots,p_N,Q)+p\cdot(Q-q)$, that $y \mapsto f(x,y)$ is uniformly strictly concave, and Proposition \ref{A=maxS} gives that \[g(x)=\max_y f(x,y)=u(q)+ \sum_{i=0}^{N} \int_{t_i}^{t_{i+1}} L\left(\tau,Q_{t_i}^\tau(Q_{i-1},p_i),\dd_\tau Q_{t_i}^\tau(Q_{i-1},p_i)\right) d\tau.\]  Proposition \ref{AQm} states that $g$ is a coercive function of some $\mathcal{Q}_{\tilde{m}}$. Since $g$ is coercive, Consequence \ref{coercmin} states that $\sigma(g)=\min g$, so we have that
\[T^t_s u(Q)=\min g=\sigma(g)=\sigma(f)=R^t_s u(Q).\]
 \\
\emph{Second step.} We only assume that the Hamiltonian $H$ is uniformly strictly convex w.r.t. $p$ ($\dd^2_p H \geq m \textrm{id}$) and satisfies 
Hypothesis \ref{esti} with some constant $C$. It does not a priori coincide with a quadratic form at infinity.

Let us prove the Markov property: we fix $u$, $s\leq \tau \leq t$ and $Q$ and we are going to show that
$R^t_\tau R^\tau_s u(Q)=R^t_s u (Q)$. 
If $Z$ denotes the quadratic form $Z(p)=\|p\|^2$, we may choose $\delta>0$ and build as in Definition \ref{ladefinition} a Hamiltonian $H_\delta$ in $\hh^{C(1+\delta)}_Z$ such that both $R^t_s u(Q)=R^t_{s,H_\delta} u(Q)$ and $R^t_{\tau,H_\delta} R^\tau_{s,H_\delta} u(Q)=R^t_\tau R^\tau_s u(Q)$. Addendum \ref{defcvx} states that $H_\delta$ can moreover be constructed uniformly strictly convex w.r.t. $p$. 

The previous work applies to $H_\delta$, and hence \[R^t_\tau R^\tau_s u(Q)=R^t_{\tau,H_\delta} R^\tau_{s,H_\delta} u(Q) =T^t_{\tau,H_\delta} T^\tau_{s,H_\delta} u(Q) = T^t_{s,H_\delta} u (Q)=R^t_{s,H_\delta} u (Q) =R^t_s u (Q)\] since $T^t_{s,H_\delta}$ is a semigroup. We hence showed that $R^t_s$ satisfies the Markov property \eqref{mark}.

The uniqueness of the viscosity operator concludes: $R^t_s=V^t_s=T^t_s$.\\
\emph{Third step.} If $H$ is convex with respect to $p$ and satisfies Hypothesis \ref{esti} with constant $C$, $H_\ep(t,q,p)=H(t,q,p) +\frac{1}{2} \ep \|p\|^2$ is uniformly strictly convex w.r.t. $p$ ($\dd^2_p H_\ep \geq \ep {\rm id}$) and satisfies Hypothesis \ref{esti} with constant $C+\ep$.

Now for all $\ep \leq 1$, the estimates of Propositions \ref{Hlip} and \ref{visclip} give, for all $s\leq t$ and Lipschitz function $u$:
\[\|R^t_{s,H_\ep} u -R^t_{s,H}u\|_\infty \leq (t-s)\|H_\ep-H\|_{\bar{V}},\]
\[\|V^t_{s,H_\ep} u -V^t_{s,H}u\|_\infty \leq (t-s)\|H_\ep-H\|_{\bar{V}},\]
where $V=\rr \times \rr^d \times \bar{B}\!\left(0,e^{(C+1)(t-s)}(1+Lip(u))-1\right)$. In other words,
\[\|R^t_{s,H_\ep} u -R^t_{s,H}u\|_\infty \leq \frac{1}{2} \ep (t-s)\left(e^{(C+1)(t-s)}(1+Lip(u))-1\right)^2,\]
\[\|V^t_{s,H_\ep} u -V^t_{s,H}u\|_\infty \leq \frac{1}{2}  \ep (t-s)\left(e^{(C+1)(t-s)}(1+Lip(u))-1\right)^2.\]
The second step applied to $H_\ep$ states that $R^t_{s,H_\ep} u=V^t_{s,H_\ep} u$, and hence letting $\ep$ tend to zero gives the conclusion: $R^t_{s,H}u=V^t_{s,H}u$.

The result is obtained analogously in the concave case, where the Lax-Oleinik semigroup is defined as a maximum, see Remark \ref{concave}.
\end{proof}

\begin{appendices}

\section[Generating families]{Generating families of the Hamiltonian flow}\label{chap}
All the results and proofs of this appendix are inspired from \cite{chaperon}. We write them down here only to explicit the time derivatives of the generating families (see Proposition \ref{F}) and the Lipschitz constant in Proposition \ref{lipfam}.

We first state a useful basic technical lemma.
\begin{lem}\label{lipliplip}
	If $u,v : \rr^n \to \rr^n$ are $\cc^1$ and such that ${\rm Lip}(u)<1$ and ${\rm Lip}(v)<1$, then $f={\rm id}-u$ and $g={\rm id}-v$ are $\cc^1$-diffeomorphisms of $\rr^n$. If $f-g$ is bounded, then $f^{-1}-g^{-1}$ is bounded by $\frac{\|f-g\|_\infty}{1-{\rm Lip}(u)}$. 
\end{lem}

\begin{proof}
	Let us first proof that $f$ is a $\cc^1$-diffeomorphism of $\rr^n$. It is clearly $\cc^1$, and is a local diffeomorphism since $\|df\|=\|{\rm id} - du\| \geq 1 - {\rm Lip}(u)>0$. To see that it is invertible, we observe that $f(q)=\theta$ can be rewritten as a fixed point problem $q=u(q) + \theta$, where the map $q \mapsto u(q)+\theta$ is contracting.
	
	Now, if $f-g$ is bounded, so is $u-v$, with $\|u-v\|_\infty=\|f-g\|_\infty$. Let us denote $x=f^{-1}(z)$ and $y=g^{-1}(z)$ for some $z$ in $\rr^n$. Then $x=u(x)+z$ and $y=v(y)+z$ and\[
	|x-y| \leq |u(x)-v(y)| \leq |u(x)-u(y)| +|u(y)-v(y)| \leq {\rm Lip}(u)|x-y|+ \|u-v\|_\infty,
	\] 
	whence $|x-y| \leq \frac{\|f-g\|_\infty}{1-{\rm Lip}(u)}$.
\end{proof}	

Let us now state two Grönwall-type estimates on Hamiltonian flows:
\begin{prop}\label{gron} Let $H$ and $\tilde{H}$ be two $\cc^2$ Hamiltonians on $\rr \times T^\star \rr^d$ such that $\|\dd^2_{q,p}H\|$ and $\|\dd^2_{q,p}\tilde{H}\|$ are uniformly bounded by a constant $C$ and $\|\dd_{q,p} H-\dd_{q,p} \tilde{H}\|$ is uniformly bounded by a constant $K$. Then, if $\phi$ and $\tilde{\phi}$ denote the Hamiltonian flows respectively associated with $H$ and $\tilde{H}$, we have for all $s\leq t$:
	\begin{enumerate}	
		\item $\|\phi^t_s - \tilde{\phi}^t_s\| \leq  \frac{K}{C} (e^{C(t-s)}-1)$,
		\item $\|d\phi^t_s - {\rm id} \|\leq e^{C(t-s)}-1$.
	\end{enumerate}	
	In particular if $t-s \leq \delta_1=\frac{\ln(3/2)}{C}$, $\phi^t_s-{\rm id}$ is $\frac{1}{2}$-Lipschitz.
\end{prop}	

\begin{lem}[Grönwall's lemma, elementary version]	\label{gronlem}
	If $t\mapsto f(t)$ is a continuous non negative function such that $f(s)=0$ and $f(t)\leq \int^t_s \left(Cf(u)+K \right) du$, then $f(t)\leq \frac{K}{C} (e^{C(t-s)}-1)$. \end{lem}

\begin{proof} Observe that the assumed inequality can be written \[\dd_t \!\left(e^{-C(t-s)} \int^t_s f(u) \;\, du\right) \leq e^{-C(t-s)} K(t-s),\] and integrating this between $s$ and $t$ we get
	\[\int^t_s  f(u) du \leq K e^{C(t-s)} \int^t_s e^{-C(u-s)} (u-s)\;\, du = \frac{K}{C^2} (e^{C(t-s)}-C(t-s)-1).
	\]
	Reinjecting this into $f(t)\leq \int^t_s \left(Cf(u)+K \right) du$ gives the result.\end{proof}

\begin{proof}[Proof of Proposition \ref{gron}]	
	Let us define $\Gamma_t (q, p) = (\dd_p H(t,q, p),- \dd_q H(t,q, p))$, so that the Hamiltonian system $\eqref{HS}$ can be rewritten $\dd_t \phi^t_s = \Gamma_t(\phi^t_s)$, and $\tilde{\Gamma}$ associated similarly with $\tilde{H}$. 
	\begin{enumerate}
		\item Since $\|\dd_{q,p} H-\dd_{q,p} \tilde{H}\|\leq K$, $\|\Gamma_u-\tilde{\Gamma}_u\|\leq K$ for all $u$ and since $\Gamma_u$ is $C$-Lipschitz, \[\begin{split}
		\|\phi^t_s& - \tilde{\phi}^t_s \| = \left\|\int^t_s \Gamma_u(\phi^u_s) - \tilde{\Gamma}_u(\tilde{\phi}^u_s) \;\,du \right\|\\
		&\leq \int^t_s \| \Gamma_u(\phi^u_s) - \Gamma_u(\tilde{\phi}^u_s) \| + \|\Gamma_u(\tilde{\phi}^u_s)- \tilde{\Gamma}_u(\tilde{\phi}^u_s)\|\;\, du\\
		& \leq \int^t_s C \|\phi^u_s - \tilde{\phi}^u_s \|+ K \;\,du.
		\end{split}
		\]
		So, $f(t)=\|\phi^t_s - \tilde{\phi}^t_s \|$ satisfies the conditions of Lemma \ref{gronlem} and hence \[
		\|\phi^t_s - \tilde{\phi}^t_s \|\leq \frac{K}{C} (e^{C(t-s)}-1).\]
		
		\item Since $\|\dd^2_{q,p} H\|\leq C$, $d\Gamma_t$ is bounded by $C$ and hence \[\|\dd_t d\phi^t_s \|=\|d(\Gamma_t(\phi^t_s))\|=\|d\Gamma_t(\phi^t_s) \circ d\phi^t_s\| \leq C \|d\phi^t_s\|.\]
		which implies that $\|d\phi^t_s - {\rm id}\|\leq \int^t_s C\!\left(  \|d\phi^t_s-{\rm id}\|+1\right)\,\; du$.
		
		Since $d\phi^s_s={\rm id}$, $f(t)=\|d\phi^t_s - {\rm id}\|$ satisfies the conditions of Lemma \ref{gronlem} with $K=C$, and hence \[
		\|d\phi^t_s - {\rm id}\|\leq e^{C(t-s)}-1.\]	\end{enumerate}
\end{proof}

If $\gamma=(q,p)$ is a path on $T^\star \rr^d$, its Hamiltonian action is given by
\[\aaa^t_s(\gamma)= \int^t_s p(\tau)\cdot \dd_\tau q(\tau) -H(\tau,\gamma(\tau)) d\tau.\]
We give here a simple element of calculus of variations, giving for a parametrized family of Hamiltonian trajectories the link between the dependence of the Hamiltonian action with respect to the parameter and the behaviour of the family at the endpoints. It is useful to understand the construction of the generating family of the flow in the next paragraph.
\begin{lem}\label{varcalc}
	If $\gamma_u=(q_u,p_u):  \rr \to T^\star \rr^d$ is a $\cc^1$ family of Hamiltonian trajectories, 
	\[\dd_u \aaa^t_s(\gamma_u) =p_u(t)\cdot\dd_u q_u(t)-p_u(s)\cdot\dd_u q_u(s). \]
\end{lem}	
\begin{proof}
	We recall the Hamiltonian system satisfied by the Hamiltonian trajectory $\gamma_u$:
	\[\left\{\begin{array}{l}
	\dd_\tau{q}_u(\tau)= \dd_p H(t,q_u(\tau),p_u(\tau)),\\
	\dd_\tau{p}_u(\tau)= - \dd_q H(t,q_u(\tau),p_u(\tau)).
	\end{array}\right.
	\]
	As a consequence,
	\[\begin{split}
	\dd_u \aaa^t_s(\gamma_u) =\dd_u \int^t_s p_u(\tau)\cdot \dd_\tau q_u(\tau) &- H(\tau,q_u(\tau),p_u(\tau)) \,d \tau \\
	= \int^t_s \dd_u p_u(\tau)\cdot \dd_\tau q_u(\tau) &+ p_u(\tau)\cdot \dd_u \dd_\tau q_u(\tau)\\
	- \dd_q H(\tau,q_u(\tau),&p_u(\tau))\cdot \dd_u q_u(\tau)- \dd_p H(\tau,q_u(\tau),p_u(\tau))\cdot \dd_u p_u(\tau ) \,d \tau \\
	= \int^t_s p_u(\tau)\cdot \dd_u \dd_\tau q_u(\tau)& + \dd_\tau p_u(\tau) \cdot \dd_u q_u(\tau)  \,d \tau= \left[p_u \cdot\dd_u q_u\right]^t_s.
	\end{split}
	\]
\end{proof}

\subsection[General case]{Generating family in the general case}\label{gfgc}

As a consequence of Lemma \ref{lipliplip} and Proposition \ref{gron}, if we choose a  $\delta\leq \delta_1=\frac{ln(3/2)}{C}$, the map $g^t_s:(q,p)\mapsto (Q^t_s(q,p),p)$ is a $\cc^1$-diffeomorphism for all $0\leq t-s \leq \delta$, since we have ${\rm Lip}(g^t_s-{\rm id}) \leq {\rm Lip}(\phi^t_s - {\rm id})\leq 1/2$. If $0\leq t-s\leq \delta$, let 
$F^t_s:\rr^{2d} \to \rr$ be the function defined by\begin{equation}\label{flowfam}
F^t_s(Q,p)=\int^t_s \left(P^\tau_s(q,p)-p\right)\cdot\dd_\tau Q^\tau_s(q,p)-H(\tau,\phi^\tau_s(q,p))\;d\tau,
\end{equation}
where $q$ is the only point satisfying $Q^t_s(q,p)=Q$, \textit{i.e.} the first coordinate of $(g^t_s)^{-1}(Q,p)$. In other terms, if $\gamma(\tau)=(q(\tau),p(\tau))$ is the unique Hamiltonian trajectory such that $(q(t),p(s))=(Q,p)$,
\begin{equation}\label{flowfam2}
F^t_s(Q,p)=p\cdot(q(s)-Q)+\aaa^t_s(\gamma)=p\cdot(q(s)-Q)+\int^t_s p(\tau)\cdot \dd_\tau q(\tau) - H(\tau,\gamma(\tau)) \,  d\tau.
\end{equation}

\begin{prop} 	\label{F}
	The family of functions $(F^t_s)_{s\leq t\leq s + \delta}$ is $\cc^1$ with respect to $s$, $t$, $Q$, $p$ and its derivatives are given by\[
	\left\{\begin{array}{lcl}
	\dd_p F^t_s (Q,p)= q-Q,& &\dd_t F^t_s (Q,p)= -H(t,Q,P), \\
	\dd_Q F^t_s (Q,p)= P-p,& & \dd_s F^t_s (Q,p)= H(s,q,p),\end{array}\right.
	\]
	where $P$ and $q$ are uniquely defined by $(Q,P)=\phi^t_s(q,p)$.
	In particular,
	\begin{displaymath}
	(Q,P) = \phi^t_s (q,p) \iff  \left\{\begin{array}{rcl}
	\dd_p F^t_s (Q,p)&=&q-Q,\\
	\dd_Q F^t_s(Q,p)&=&P-p.
	\end{array}\right.
	\end{displaymath}
	Furthermore, if $Q=Q^t_s (q,p)$ and $\gamma$ denotes the Hamiltonian trajectory issued from $(q,p)$,	\[
	F^t_s(Q,p)=\aaa^t_s(\gamma) - p\cdot(Q-q).
	\]
\end{prop}

The generating family is constructed by adding boundary terms to the Hamiltonian action of a Hamiltonian trajectory depending on parameters. 

\begin{proof}[Proof of Proposition \ref{F}]
	Let us differentiate $F$ with respect to $s$, $t$, $Q$ and $p$. The rest of the proposition is a straightforward consequence of the form of the derivatives of $F$. In terms of Lemma \ref{varcalc}, let us denote by $u=(s,t,Q,p)$ and by $\gamma_u=(q_u,p_u)$ the unique Hamiltonian trajectory such that $p_u(s)=p$ and $q_u(t)=Q$. Let us gather the derivatives of $q_u$ at the endpoints in view of applying Lemma \ref{varcalc}: we differentiate $q_u(t)=Q$ with respect to $s$, $t$, $Q$ and $p$, while denoting by $\tau$ the time variable of the trajectory $\gamma_u$:
	\begin{equation}\label{derder}
	\dd_s q_u(t)= 0,\;
	\dd_t q_u(t)+\dd_\tau q_u(t)=0,  \;
	\dd_Q q_u(t)={\rm id},\;
	\dd_p q_u(t)=0.
	\end{equation}
	The equation \eqref{flowfam2} defining $F$ may now be written as:
	\[F^t_s(Q,p)=p\cdot(q_u(s)-Q)+\aaa^t_s (\gamma_u).\]
	Lemma \ref{varcalc} gives the dependence of $\aaa^t_s(\gamma_u)$ with respect to $u$. We differentiate this expression with respect to $s$, $t$, $Q$ and $p$, cautiously denoting by $\tau$ the time variable of the trajectory $\gamma_u=(q_u,p_u)$, and taking into account the term $p\cdot(q_u(s)-Q)$ and the boundaries of the integral defining the action:
	\[\begin{split}
	\dd_s F^t_s(Q,p)=\, &p\cdot (\dd_s q_u(s)+\dd_\tau q_u(s))-\left(p_u(s)\cdot \dd_\tau q_u(s)-H(s,q_u(s),p_u(s))\right)+[p_u\cdot \dd_s q_u]^t_s\\
	=\, & H(s,q_u(s),p_u(s))+(p-p_u(s))\cdot (\dd_s q_u(s)+\dd_\tau q_u(s))+ p_u(t)\cdot \dd_s q_u(t)\\
	=\, & H(s,q,p),\end{split}\]\[
	\begin{split}
	\dd_t F^t_s(Q,p)=\, & p\cdot \dd_t q_u(s)  + \left(p_u(t)\cdot \dd_\tau q_u(t)-H(t,q_u(t),p_u(t))\right)+[p_u\cdot \dd_t q_u]^t_s\\
	=\, &(p-p_u(s)) \cdot \dd_t q_u(s) + p_u(t)\cdot (\dd_\tau q_u(t)+\dd_t q_u(t)) -H(t,q_u(t),p_u(t))\\
	=\,&- H(t,Q,P),\\
	\dd_Q F^t_s(Q,p)=\, &p\cdot \dd_Q q_u(s)-p + [p_u\cdot \dd_Q q_u]^t_s\\
	= \, &(p-p_u(s))\cdot \dd_Q q_u(s) - p + p_u(t)\cdot \dd_Q q_u(t) = -p+P,\\
	\dd_p F^t_s(Q,p)=\, &p\cdot \dd_p q_u(s)+ q_u(s)-Q + [p_u\cdot \dd_p q_u]^t_s\\
	= \, &(p-p_u(s))\cdot \dd_p q_u(s) +q_u(s)-Q + p_u(t)\cdot \dd_p q_u(t)=q-Q
	\end{split}\]
	if we denote by $(P,q)=(p_u(t),q_u(s))$, using \eqref{derder} and $(p_u(s),q_u(t))=(p,Q)$.
\end{proof}

\begin{prop}\label{FHdep} If $H_\mu$ is a $\cc^2$ family of Hamiltonians such that $\|\dd^2_{q,p} H_\mu\|$ is bounded by $C$, let us denote by $F^t_{s,\mu}$ associated with $H_\mu$ as previously for $t-s\leq \delta$. Then \[
	\dd_\mu F^t_{s,\mu}(Q,p) = - \int^t_s \dd_\mu H_\mu(\tau,\gamma_\mu(\tau)) \,\;d\tau
	\]
	where $\gamma_\mu=(q_\mu,p_\mu)$ is the unique Hamiltonian trajectory for $H_\mu$ with $q_\mu(t)=Q$ and $p_\mu(s)=p$. \end{prop}

\begin{proof} Let us fix $Q$, $p$, $s$ and $t$, and take $\gamma_\mu$ as in the statement. 
	By definition \eqref{flowfam2},\[
	F^t_{s,\mu}(Q,p) =p\cdot(q_\mu(s)-Q))+\aaa^t_{s,H_\mu} (\gamma_\mu)	
	\]
	and thus differentiating w.r.t. $\mu$ gives the following, using Lemma \ref{varcalc}:
	\[		\dd_\mu F^t_{s,\mu}(Q,p) = p \cdot \dd_\mu q_\mu(s) + [p_\mu\cdot \dd_\mu q_\mu]^t_s - \int^t_s \dd_\mu H_\mu(\tau,\gamma_u(\tau)) \; d\tau.\]
	Now, since $q_\mu(t)=Q$ for all $\mu$, $\dd_\mu q_\mu(t)=0$, and since $p=p_\mu(s)$, the two first terms of the right hand side cancel, hence the conclusion.
\end{proof}
When $t-s$ is large, we choose a subdivision of the time interval with steps smaller than $\delta$ and add intermediate coordinates along this trajectory.	For each $s\leq t$ and $(t_i)$ such that $t_0=s\leq t_1 \leq \cdots \leq t_{N+1}=t$ and $t_{i+1}-t_i\leq \delta$ for each $i$, let $G^t_s:\rr^{2d(1+N)}\to \rr$ be the function defined by\[
G^t_s(p_0,Q_0,p_1,Q_1,\cdots, Q_{N-1}, p_N,Q_N)=\sum_{i=0}^N F^{t_{i+1}}_{t_i}(Q_i,p_i) +p_{i+1}\cdot(Q_{i+1}-Q_i)
\]
where indices are taken modulo $N+1$.

\begin{prop}\label{Ggen}The family of functions $(G^t_s)_{s\leq t}$ is $\cc^1$ with respect to $s$, $t$, $t_i$, $Q_i$ and $p_i$, and its derivatives are given by \[
	\left\{\begin{array}{rl}
	\dd_{p_i} G^t_s (p_0,\cdots,Q_N)=\dd_p F^{t_{i+1}}_{t_i}(Q_i,p_i) +Q_i -Q_{i-1}=q_i-Q_{i-1},\\
	\dd_{Q_i} G^t_s (p_0,\cdots,Q_N)= \dd_Q F^{t_{i+1}}_{t_i}(Q_i,p_i) +p_i-p_{i+1} =P_i -p_{i+1},
	\end{array}\right.
	\]
	where $P_i$ and $q_i$ are uniquely defined by $(Q_i,P_i)=\phi^{t_{i+1}}_{t_i}(q_i,p_i)$ and indices are taken modulo $N+1$.
	
	It is hence a \emph{generating family} for the flow $\phi$, meaning that if we denote $(Q,p)=(Q_N,p_0)$ and $\nu=(Q_0,p_1,\cdots, Q_{N-1}, p_N)$, \begin{displaymath}
	(Q,P) = \phi^t_s (q,p) \iff  \exists \nu \in \rr^{2dN},
	\left\{\begin{array}{rcl}
	\dd_p G^t_s (p,\nu,Q)&=&q-Q,\\	
	\dd_Q G^t_s(p,\nu,Q)&=&P-p,\\
	\dd_\nu G^t_s(p,\nu,Q)&=&0,
	\end{array}\right.
	\end{displaymath}
	and in this case $(Q_i,p_{i+1})=\phi^{t_{i+1}}_s(q,p)$ for all $0\leq i\leq N-1$. 

	Furthermore, if $Q=Q^t_s (q,p)$ and $\gamma$ denotes the Hamiltonian trajectory issued from $(q,p)$,\[G^t_s(p,\nu,Q)=\aaa^t_s(\gamma)-p\cdot(Q-q)\]
	if $\dd_\nu G^t_s(p,\nu,Q)=0$.
\end{prop} 
\begin{proof} The derivatives of $G$, which are directly obtained from the ones of $F$, give that, if $p$ and $Q$ are fixed, 
	\begin{displaymath}
	\left\{\begin{array}{rcl}
	\dd_p G^t_s (p,\nu,Q)&=&q-Q,\\
	\dd_Q G^t_s(p,\nu,Q)&=&P-p,\\
	\dd_\nu G^t_s(p,\nu,Q)&=&0,
	\end{array}\right. \iff \left\{\begin{array}{l}
	q=q_0,\\
	P_N=P,\\
	(Q_i,p_{i+1})=\phi^{t_{i+1}}_{t_i}(Q_{i-1},p_i)\;\, \forall\, 0\leq i\leq N-1.
	\end{array}\right. 
	\end{displaymath}		
	If this is satisfied, $\nu$ describes a non broken Hamiltonian geodesic, $(Q_i,p_{i+1})=\phi^{t_{i+1}}_s(q,p)$ and $(Q,P)=\phi^t_s(q,p)$.
	If $(Q,P)=\phi^t_s(q,p)$, then $\nu$ is given by $\phi^{t_i}_s(q,p)$ and the right hand system holds.
	
	The critical value of $\nu \mapsto G^t_s(p,\nu,Q)$ is obtained by summing up the result obtained for $F$ in Proposition \ref{F}.
\end{proof}

The last statement compares the generating families of flows related to Hamiltonians with Lipschitz difference.
\begin{prop}\label{lipfam}
	Let $H$ and $\tilde{H}$ be two $\cc^2$ Hamiltonians on $\rr \times T^\star \rr^d$ such that $\|\dd^2_{q,p}H\|$ and $\|\dd^2_{q,p}\tilde{H}\|$ are uniformly bounded by a constant $C$ and $\|\dd_{q,p} H-\dd_{q,p} \tilde{H}\|$ is uniformly bounded by a constant $K$. We can find a $\delta>0$ suiting both $\tilde{H}$ and $H$ and build $\tilde{G}^t_{s}$ and $G^t_{s}$ with the same subdivision $(t_i)$, and then $\tilde{G}^t_{s}-G^t_{s}$ is Lipschitz with constant $4 \frac{K}{C} (e^{C(t-s)}-1)$ and also with constant $2 \frac{K}{C}$.
\end{prop}
\begin{proof}
	Let $\delta\leq \delta_1=\frac{\ln(3/2)}{C}$ so that both $\phi^t_s-{\rm id}$ and $\tilde{\phi}^t_s-{\rm id}$ are $\frac{1}{2}$-Lipschitz if $0\leq t-s \leq \delta$, see Proposition \ref{gron}, and 
	in that case $g^t_s:(q,p)\mapsto (Q^t_s(q,p),p)$ satisfies also ${\rm Lip}(g^t_s-{\rm id})\leq 1/2$.
	
	Proposition \ref{gron} states that $\|\tilde{\phi}^{t_{i+1}}_{t_i}-\phi^{t_{i+1}}_{t_i}\|_\infty \leq \frac{K}{C} (e^{C(t_{i+1}-t_i)}-1)$ under the assumptions made on $H$ and $\tilde{H}$. We are hence going to check that for all $i$, $\|\dd_{Q_i} \tilde{G}^t_s - \dd_{Q_i} G^t_s \|$ and $\|\dd_{p_i} \tilde{G}^t_s - \dd_{p_i} G^t_s \|$ are both bounded by $4 \|\tilde{\phi}^{t_{i+1}}_{t_i}-\phi^{t_{i+1}}_{t_i}\|_\infty$ in order to get the wanted Lipschitz constants. 	
	Proposition \ref{Ggen} states that $\|\dd_{Q_i} \tilde{G}^t_s - \dd_{Q_i} G^t_s \|=\|\tilde{P}_i-P_i\|$ and $\|\dd_{p_i} \tilde{G}^t_s - \dd_{p_i} G^t_s \|=\|\tilde{q}_i-q_i\|$, where $P_i$ and $q_i$ (resp. $\tilde{P}_i$ and $\tilde{q}_i$) are uniquely defined by $(Q_i,P_i)=\phi^{t_{i+1}}_{t_i}(q_i,p_i)$ (resp. $(Q_i,\tilde{P}_i)=\tilde{\phi}^{t_{i+1}}_{t_i}(\tilde{q}_i,p_i)$). Since $(q_i,p_i)=(g^{t_{i+1}}_{t_i})^{-1}(Q_i,p_i)$ and $(\tilde{q}_i,p_i)=(\tilde{g}^{t_{i+1}}_{t_i})^{-1}(Q_i,p_i)$, Lemma \ref{lipliplip} gives \[
	\|\tilde{q_i}-q_i\| \leq \|(\tilde{g}^{t_{i+1}}_{t_i})^{-1}-(g^{t_{i+1}}_{t_i})^{-1}\|_\infty \leq \frac{ \|\tilde{g}^{t_{i+1}}_{t_i}-g^{t_{i+1}}_{t_i}\|_\infty}{1-{\rm Lip}(g^t_s-{\rm id})}\leq 2 \|\tilde{\phi}^{t_{i+1}}_{t_i}-{\phi}^{t_{i+1}}_{t_i}\|_\infty
	\] 
	since ${\rm Lip}(g^{t_{i+1}}_{t_i}-{\rm id}) \leq 1/2$. Now,
	\[\begin{split}
	\|\tilde{P}_i-P_i\|&\leq \|\tilde{\phi}^{t_{i+1}}_{t_i}(\tilde{q}_i,p_i)-\phi^{t_{i+1}}_{t_i}(q_i,p_i)\|\\
	& \leq \|\tilde{\phi}^{t_{i+1}}_{t_i}(\tilde{q}_i,p_i)-\phi^{t_{i+1}}_{t_i}(\tilde{q}_i,p_i)\|+ {\rm Lip}(\phi^{t_{i+1}}_{t_i})\|\tilde{q}_i-q_i\|\\
	& \leq \|\tilde{\phi}^{t_{i+1}}_{t_i}-{\phi}^{t_{i+1}}_{t_i}\|_\infty + {\rm Lip}(\phi^{t_{i+1}}_{t_i}) 2  \|\tilde{\phi}^{t_{i+1}}_{t_i}-{\phi}^{t_{i+1}}_{t_i}\|_\infty \\&\leq 4\|\tilde{\phi}^{t_{i+1}}_{t_i}-{\phi}^{t_{i+1}}_{t_i}\|_\infty
	\end{split}\]
	since $\phi^{t_{i+1}}_{t_i}$ is $\frac{3}{2}$-Lipschitz.
	
	Since $t_{i+1}-t_i$ is smaller than $t-s$ and than $\delta$ for all $i$, we have proved that $\|d\tilde{G}^t_s-dG^t_s\|$ is bounded by $4\frac{K}{C}(e^{C\delta}-1)\leq 2 \frac{K}{C}$ and by $4\frac{K}{C}(e^{C(t-s)}-1)$.
\end{proof}

\subsection[Convex case]{Generating family in the convex case}\label{gfcc}

In this section we assume that the Hamiltonian $H$ satisfies Hypothesis \ref{esti} with constant $C$, and that there exists 
$m>0$ such that for each $(t,q,p)$, $\dd^2_p H(t,q,p) \geq m {\rm id}$ in the sense of quadratic forms.

\begin{prop}\label{twist} The following holds in the sense of quadratic forms:\[
	\dd_p Q^t_s \geq m(t-s){\rm id}-2\left(e^{C(t-s)}-1-C(t-s)\right){\rm id}.\]
	In particular there exists $\delta_2>0$ depending only on $C$ and $m$ such that if $|t-s|\leq \delta_2$,
	\[
	\dd_p Q^t_s \geq \frac{m}{2}(t-s){\rm id}\]
	which implies that the function $p \mapsto Q^t_s(q,p)$ is $\frac{m(t-s)}{2}$-monotone, meaning that
	\[
	(Q^t_s(q,\tilde{p})-Q^t_s(q,p))\cdot(\tilde{p}-p) \geq \frac{m}{2}(t-s)\|\tilde{p}-p\|^2.\]	 
	In particular, if $|t-s|\leq \delta_2$, $(q,p)\mapsto(q,Q^t_s(q,p))$ is a $\cc^1$-diffeomorphism.
\end{prop}
\begin{rem}\label{quadcomp}
	For $A$ a not necessarily symmetric matrix, we say that $A\geq c {\rm id}$ in the sense of quadratic forms if $Ax\cdot x\geq c\|x\|^2$ for all $x$. If $\|A\|\leq a$, then in particular $-a {\rm id}\leq A \leq a {\rm id}$.
\end{rem}
\begin{proof}
	Let us recall the variational equation
	\[ \dd_p \dot{Q}^t_s = \dd^2_p H \dd_p P^t_s + \dd^2_{q,p} H \dd_p Q^t_s\]
	that we write under the form
	\[ \dd_p \dot{Q}^t_s - \dd^2_p H = \dd^2_p H (\dd_p P^t_s- {\rm id}) + \dd^2_{q,p} H \dd_p Q^t_s. \]
	Lemma \ref{gron} gives that $\|\dd_q Q^t_s - {\rm id}\|$, $\|\dd_p Q^t_s\|$ and $\|\dd_p P^t_s - {\rm id}\|$ are smaller than $e^{C(t-s)}-1$. Adding the estimate on $\dd^2 H$, we get
	\[\|\dd_p \dot{Q}^t_s - \dd^2_p H \|\leq 2C(e^{C(t-s)}-1),\]
	which implies that \[\dd_p \dot{Q}^t_s \geq \dd^2_p H - 2C(e^{C(t-s)}-1) {\rm id} \geq \left(m-2C(e^{C(t-s)}-1)\right) {\rm id}\] in the sense of quadratic forms, see Remark \ref{quadcomp}. Integrating the result between $s$ and $t$ we obtain
	\[\dd_p Q^t_s \geq m(t-s){\rm id}-2\left(e^{C(t-s)}-1-C(t-s)\right){\rm id}.\]	
	Since the second term of the right hand side is second order, there exists a constant $\delta_2>0$ depending only on $C$ and $m$ such that if $|t-s|\leq \delta_2$,
	\[
	\dd_p Q^t_s \geq \frac{m}{2}(t-s){\rm id},\]
	which means that for all $z$,
	\[
	\dd_p Q^t_s(q,p) z \cdot z \geq \frac{m}{2}(t-s)\|z\|^2.\]
	Applying this to $z=\tilde{p}-p$ we get  \[\begin{split}
	(Q^t_s(q,\tilde{p})-Q^t_s(q,p))\cdot(\tilde{p}-p) &= \int^1_0 \dd_p Q^t_s(q,p + \tau(\tilde{p}-p))(\tilde{p}-p) d\tau \cdot (\tilde{p}-p) \\
	& = \int^1_0 \dd_p Q^t_s(q,p + \tau(\tilde{p}-p))(\tilde{p}-p) \cdot (\tilde{p}-p) d\tau \\
	& \geq \int^1_0 \frac{m}{2}(t-s)\|\tilde{p}-p\|^2 d\tau \geq \frac{m}{2}(t-s)\|\tilde{p}-p\|^2. \end{split}\]	
	We have proved that the function $p \mapsto Q^t_s(q,p)$ is $\frac{m(t-s)}{2}$-monotone.
	It is then a classical result that $p\mapsto Q^t_s(q,p)$ is a global $\cc^1$-diffeomorphism (see for example Proposition $51$ of \cite{bernard}), and therefore $(q,p)\mapsto (q, Q^t_s(q,p))$ is also a global $\cc^1$-diffeomorphism.
\end{proof}

\begin{prop}\label{Gccve} There exists $\delta_3>0$ depending only on $C$ and $m$ such that if $G^t_s$ is constructed with a maximal step smaller than $\delta_3$, the function $(p_0,p_1,\cdots,p_N) \mapsto G^t_s(p_0,Q_0,p_1,Q_1,\cdots, Q_{N-1}, p_N,Q_N)$ is uniformly strictly concave.
\end{prop}
\begin{proof} Let us denote the function $(p_0,p_1,\cdots,p_N) \mapsto G^t_s(p_0,Q_0,p_1,\cdots, Q_{N-1}, p_N,Q_N)$ by $g$.
	Proposition \ref{Ggen} gives that $\dd_{p_i} G^t_s (p_0,\cdots,Q_N)=q_i-Q_{i-1}$,
	where $q_i$ is the only point such that $Q_i=Q^{t_{i+1}}_{t_i}(q_i,p_i)$. On one hand, we get that if $i\neq j$, $\dd^2_{p_ip_j} G^t_s$ is zero. On the other hand, $\dd^2_{p_i} G^t_s = \dd_{p_i} q_i$. 
	
	Differentiating $Q^{t_{i+1}}_{t_i}(q_i,p_i)=Q_i$ w.r.t. $p_i$ gives
	\[\dd_q Q^{t_{i+1}}_{t_i}(q_i,p_i) \dd_{p_i} q_i+\dd_p Q^{t_{i+1}}_{t_i}(q_i,p_i)=0,\]
	so we have
	\[\dd^2_{p_i} G^t_s=-(\dd_q Q^{t_{i+1}}_{t_i})^{-1}\dd_p Q^{t_{i+1}}_{t_i}.\]	
	Lemma \ref{gron} gives that $\|\dd_p Q^{t_{i+1}}_{t_i}\|\leq e^{C(t_{i+1}-t_i)}-1$ and $\|\dd_q Q^{t_{i+1}}_{t_i} - {\rm id} \|\leq e^{C(t_{i+1}-t_i)}-1$, and hence $\dd_q Q^{t_{i+1}}_{t_i}$ is invertible as long as $e^{C(t_{i+1}-t_i)}<2$ and satisfies \begin{equation}\label{B3}
	\left\|(\dd_q Q^{t_{i+1}}_{t_i})^{-1}-{\rm id}\right\|\leq \frac{e^{C(t_{i+1}-t_i)}-1}{2-e^{C(t_{i+1}-t_i)}}.
	\end{equation}
	Using \eqref{B3} and the estimate of Proposition \ref{twist} we get
	\[\begin{split}
	\dd^2_{p_i} G^t_s=-((\dd_q Q^{t_{i+1}}_{t_i})^{-1}-&{\rm id})\dd_p Q^{t_{i+1}}_{t_i}- \dd_p Q^{t_{i+1}}_{t_i}\\
	\leq \frac{e^{C(t_{i+1}-t_i)}-1}{2-e^{C(t_{i+1}-t_i)}}&(e^{C(t_{i+1}-t_i)}-1){\rm id}  \\
	& - m(t_{i+1}-t_i) {\rm id}+2\left(e^{C(t_{i+1}-t_i)}-1-C(t_{i+1}-t_i)\right){\rm id}. 
	\end{split} \]
	Since the only first order term is $- m(t_{i+1}-t_i) {\rm id}$, there exists a $\delta_3>0$ depending only on $C$ and $m$ such that if $t_{i+1}-t_i\leq \delta_3$,
	\[\dd^2_{p_i} G^{t_{i+1}}_{t_i}\leq - \frac{m}{2}(t_{i+1}-t_i) {\rm id}.\]
	If $\delta\leq \delta_3$, then $d^2 g$, which is a blockwise diagonal matrix, is smaller than $-\frac{m \delta}{2} {\rm id}$ and $g$ is hence uniformly strictly concave.
\end{proof}
When the Hamiltonian $H$ is strictly convex w.r.t. $p$, the Lagrangian function on the tangent bundle is associated as follows:
\[L(t,q,v)= \sup_{p \in \left(\rr^d\right)^\star} p\cdot v - H(t,q,p).\]

Assume that $\delta< \min(\delta_1,\delta_2,\delta_3)$, and let $h_i$ be the inverse of $(q,p)\mapsto\left(q,Q^{t_{i+1}}_{t_i}(q,p)\right)$ (see Proposition \ref{twist}). We define 
\[A^t_s(q,Q_0,\cdots,Q_{N-1},Q)=\sum_{i=0}^{N} \int_{t_i}^{t_{i+1}} L\left(\tau,Q_{t_i}^\tau(h_i(Q_{i-1},Q_i)),\dd_\tau Q_{t_i}^\tau(h_i(Q_{i-1},Q_i))\right) d\tau\]
with the notations $q=Q_{-1}$ and $Q=Q_N$.

\begin{prop}\label{famcvx} The so-called \emph{Lagrangian generating family} $A$ is $\cc^1$ and satisifies :
	\begin{enumerate} 
		\item 
		\[A^t_s(q,Q_0,\cdots,Q_{N-1},Q)=  \max_{(p_0,\cdots,p_N)} G^t_s(p_0,Q_0,\cdots, Q_{N-1}, p_N,Q)
		+p_0\cdot(Q-q).\]
		
		\item 
		\[\left\lbrace\begin{array}{lc}
		\dd_{Q_i} A^t_s (q,Q_0,\cdots,Q_{N-1},Q)= P_i-p_{i+1} & \forall i=0\cdots N-1,\\
		\dd_q A^t_s(q,Q_0,\cdots,Q_{N-1},Q)=-p_0,&\\
		\dd_Q A^t_s(q,Q_0,\cdots,Q_{N-1},Q)=P_N,&
		\end{array}\right.\]
		where $P_i$ and $p_i$ are uniquely defined by $(Q_i,P_i)=\phi^{t_{i+1}}_{t_i}(Q_{i-1},p_i)$. 
	\end{enumerate}
\end{prop}	 

This function is indeed a \emph{generating family} for the flow, in the sense that if $v=(Q_0,\cdots,Q_{N-1})$, the graph of the flow $\phi^t_s$ is the set
\[\left\lbrace\left((q,-\dd_q A^t_s(q,v,Q)),(Q,\dd_Q A^t_s (p,v,Q)) \right) \left|\, \dd_v A^t_s (p,v,Q)=0\right. \right\rbrace.\]

\begin{proof}\begin{enumerate}
		
		\item 	The function $(p_0,p_1,\cdots,p_N)\mapsto G^t_s(p_0,Q_0,p_1,Q_1,\cdots, Q_{N-1}, p_N,Q)+p_0\cdot(Q-q)$  is uniformly strictly concave by Proposition \ref{Gccve}, and its maximum is hence attained by a unique point. 
		
		For $i$ from $1$ to $N$, this is a consequence of the derivative of $G^t_s$ given in Proposition 
		\ref{Ggen}: $\dd_{p_i} G^t_s(p_0,Q_0,p_1,Q_1,\cdots, Q_{N-1}, p_N,Q)=q_i-Q_{i-1}=0$ if and only if $Q^{t_{i+1}}_{t_i}(Q_{i-1},p_i)=Q_i$. For $i=0$, the derivative with respect to $p_0$ is $q_0-q$ where $q_0$ is the only point such that $Q^{t_1}_s(q_0,p_0)=Q_0$, and consequently $\dd_{p_0}\left( G^t_s + p_0\cdot(Q-q)\right) =0$ if and only if $Q^{t_1}_s(q,p_0)=Q_0$.		
		
		The maximum is hence uniquely attained by the $\cc^1$ function
		\[\mathbf{p}:(q,Q_0\cdots,Q) \mapsto \left(h_0^2(q,Q_0),h_1^2(Q_0,Q_1),\cdots,h_N^2(Q_{N-1},Q)\right),\]
		where $h_i^2$ denotes the second coordinate of $h_i$. In other terms, its coordinates satisfy $Q^{t_{i+1}}_{t_i}(Q_{i-1},\mathbf{p}_i)=Q_i$		
		for all $i$ from $0$ to $N$, with the notations $q=Q_{-1}$ and $Q=Q_N$.

		By definition of the Lagrangian, if $(q(t),p(t))$ is a Hamiltonian trajectory associated with $H$, then
		\[L(t,q(t),\dot{q}(t))=  p(t)\cdot \dot{q}(t) - H(t,q(t),p(t)).\]			
		In particular the function $F$ defined in \eqref{flowfam2} can be written in Lagrangian terms:
		\[	F^t_s(Q,p)=p\cdot(q-Q)+\int^t_s L(\tau,Q^\tau_s(q,p),\dd_\tau Q^\tau_s(q,p))\;d\tau.\]
		where $q$ is the only point such that $Q^t_s(q,p)=Q$, and the function $G$ is hence the following:
		\[\begin{split}
		G^t_s&(p_0,Q_0,p_1,Q_1,\cdots, Q_{N-1}, p_N,Q_N)= \sum_{i=0}^N  F^{t_{i+1}}_{t_i}(Q_i,p_i) +p_{i+1}\cdot(Q_{i+1}-Q_i)\\
		&= \sum_{i=0}^N \int^{t_{i+1}}_{t_i} L(\tau,Q^\tau_{t_i}(q_i,p_i),\dd_\tau Q^\tau_{t_i}(q_i,p_i))\;d\tau
		+ p_i\cdot(q_i-Q_i) + p_{i+1} \cdot (Q_{i+1}-Q_i)\\
		&=\sum_{i=0}^N \int^{t_{i+1}}_{t_i} L(\tau,Q^\tau_{t_i}(q_i,p_i),\dd_\tau Q^\tau_{t_i}(q_i,p_i))\;d\tau
		+ p_i\cdot(q_i-Q_{i-1}),
		\end{split}	
		\]
		where $q_i$ is the only point such that $Q^{t_{i+1}}_{t_i}(q_i,p_i)=Q_i$.
		
		Now, if $(\mathbf{p}_0,\cdots,\mathbf{p}_N)$ is the critical point, we have on one hand that $q_i =Q_{i-1}$ and on the other hand that $Q^{t_{i+1}}_{t_i}(q_i,\mathbf{p}_i)=Q_i$ if and only if $(q_i,\mathbf{p}_i)=h_i(q_i,Q_i)$, hence the result.

		\item Since $A^t_s(q,\cdots,Q)=G^t_s(Q_0,\cdots,Q,\mathbf{p}(q,\cdots,Q))	+ \mathbf{p}_0(q,\cdots,Q)\cdot(Q-q)$ while reorganising the variables, we have for all $i$ from $-1$ to $N$
		\[\begin{split}
		\dd_{Q_i}A^t_s (q,\cdot\cdot,Q)  = \dd_{Q_i} &\left(G^t_s(Q_0,\cdot\cdot,Q,\mathbf{p}(q,\cdot\cdot,Q))
		+ \mathbf{p}_0(q,\cdot\cdot,Q)\cdot(Q-q)\right) \\
		&+ \underbrace{\dd_\mathbf{p} \left(G^t_s(Q_0,\cdot,Q,\mathbf{p}(q,\cdot\cdot,Q))
			+  \mathbf{p}_0(q,\cdot\cdot,Q)\cdot(Q-q)\right)}_{=0}   \dd_{Q_i} \mathbf{p} 
		\end{split}
		\]		
		since $\mathbf{p}(q,\cdot\cdot,Q)$ is the critical point. The result is then a straightforward consequence of Proposition \ref{Ggen} and of the second point.
	\end{enumerate}\end{proof}	
	Let us state what happens in the case of a uniformly strictly concave Hamiltonian.
	\begin{rem}\label{concave}
		If $H$ is uniformly strictly concave (which means that $-H$ is uniformly strictly convex), Proposition \ref{twist} analogous statement is that $-Q^t_s$ is $m(t-s)/2$ monotone, which implies the twist property: $(q,p)\mapsto(q,Q^t_s(q,p))$ is a $\cc^1$-diffeomorphism for $|t-s| \leq \delta_2$. Proposition \ref{Gccve} analogous statement is that $-G^t_s$ is strictly concave with respect to its $\boldmath{p}$ variable for $|t-s|\leq \delta_3$. The Lagrangian is now defined by
		\[L(t,q,v)= \inf_{p \in \left(\rr^d\right)^\star} p\cdot v - H(t,q,p),\]
		and the analogous statement of Proposition \ref{famcvx} is that \[A^t_s(q,Q_0,\cdots,Q_{N-1},Q)=  \min_{(p_0,p_1,\cdots,p_N)} G^t_s(p_0,Q_0,\cdots, Q_{N-1}, p_N,Q)
		+p_0\cdot(Q-q),\]
		where $A$ is defined as in the convex case. Finally, the next Proposition holds in both convex and concave cases.
	\end{rem}

	\begin{prop}\label{lbnl2}
		Let $H$ and $\tilde{H}$ be two $\cc^2$ Hamiltonians on $\rr \times T^\star \rr^d$ such that 
		\begin{itemize}
			\item 	$\dd^2_{q,p}H$ and $\dd^2_{q,p}\tilde{H}$ are uniformly bounded by a constant $C$,
			\item $\dd^2_p H\geq m {\rm id}$, $\dd^2_p \tilde{H}\geq m {\rm id}$ (or $\leq -m {\rm id}$ in the concave case),
			\item 	$\dd_{q,p} H-\dd_{q,p} \tilde{H}$ is uniformly bounded by a constant $K$. 
		\end{itemize}
		We fix a subdivision $s\leq t_0 \leq \cdots \leq t_{N+1}=t$ such that $0<t_{i+1}-t_i<\delta$, with $\delta$ smaller than $\delta_1$, $\delta_2$ and $\delta_3$, and build the Lagrangian generating families $A^t_s$ and $\tilde{A}^t_s$ as previously, respectively for $H$ and $\tilde{H}$.	
		Then the difference $\tilde{A}^t_{s}-A^t_{s}$ is Lipschitz.
	\end{prop}	
	
	\begin{proof} We denote by $\tilde{\cdot}$ the objects defined for $\tilde{H}$ instead of $H$. 
		Given the form of the derivatives of $A^t_s$ obtained in Proposition \ref{famcvx}, it is enough to prove that 
		$\tilde{p}_i-p_i$ and $\tilde{P}_i-P_i$ are bounded uniformly with respect to $(q,\cdots,Q) $ for all $i$,  where $P_i$ and $p_i$ (resp. $\tilde{P}_i$ and $\tilde{p}_i$) are uniquely defined by $(Q_i,P_i)=\phi^{t_{i+1}}_{t_i}(Q_{i-1},p_i)$ (resp. $(Q_i,\tilde{P}_i)=\tilde{\phi}^{t_{i+1}}_{t_i}(Q_{i-1},\tilde{p}_i)$). 
		
		Proposition \ref{twist} states that $p\mapsto Q^{t_{i+1}}_{t_i}(q,p)$ is $\frac{m(t_{i+1}-t_i)}{2}$-monotone, meaning that for all $p$ and $\tilde{p}$
		\[(Q^{t_{i+1}}_{t_i}(q,\tilde{p})-Q^{t_{i+1}}_{t_i}(q,p))\cdot(\tilde{p}-p) \geq \frac{m}{2}(t_{i+1}-t_i)\|\tilde{p}-p\|^2.\]
		Applying the Cauchy-Schwarz inequality and dividing by $\|\tilde{p}-p\|$ we get
		\[ \|\tilde{p}-p\| \leq \frac{2}{m(t_{i+1}-t_i)} \left\|Q^{t_{i+1}}_{t_i}(q,\tilde{p})-Q^{t_{i+1}}_{t_i}(q,p)\right\|.\]
		Take $q=Q_{i-1}$, $\tilde{p}=\tilde{p}_i$ and $p=p_i$. Since $Q^{t_{i+1}}_{t_i}(Q_{i-1},p_i)=\tilde{Q}^{t_{i+1}}_{t_i}(Q_{i-1},\tilde{p}_i)$, we have
		\[ \|\tilde{p}_i-p_i\|\leq  \frac{2}{m(t_{i+1}-t_i)}    \left\|Q^{t_{i+1}}_{t_i}(Q_{i-1},\tilde{p}_i)-\tilde{Q}^{t_{i+1}}_{t_i}(Q_{i-1},\tilde{p}_i)\right\| 
		\leq \frac{2}{m\mu}  \left\|\phi^{t_{i+1}}_{t_i} - \tilde{\phi}^{t_{i+1}}_{t_i}\right\|_\infty
		\]
		where $\mu$ denotes the minimum of $t_{i+1}-t_i$.
		The first estimate of Proposition \ref{gron} gives:
		\[ \|\tilde{p}_i-p_i\|\leq  \frac{2}{m\mu} \frac{K}{C}(e^{C\delta}-1).\]
		
		Finally, since $P_i=P^{t_{i+1}}_{t_i}(Q_{i-1},p_i)$ and $\tilde{P}_i=\tilde{P}^{t_{i+1}}_{t_i}(Q_{i-1},\tilde{p}_i)$,	
		\[\begin{split}
		\|\tilde{P}_i-P_i\|&\leq \left\|\phi^{t_{i+1}}_{t_i} - \tilde{\phi}^{t_{i+1}}_{t_i}\right\|_\infty 
		+ \left\|P^{t_{i+1}}_{t_i}(Q_{i-1},p_i) - P^{t_{i+1}}_{t_i}(Q_{i-1},\tilde{p}_i) \right\|\\
		& \leq \left\|\phi^{t_{i+1}}_{t_i} - \tilde{\phi}^{t_{i+1}}_{t_i}\right\|_\infty + Lip(\phi^{t_{i+1}}_{t_i}) \|\tilde{p}_i-p_i\|
		\end{split}
		\]
		is uniformly bounded since $\phi^{t_{i+1}}_{t_i}$ is $\frac{3}{2}$-Lipschitz (see Proposition \ref{gron}).
	\end{proof}	
	
\section[Minmax]{Minmax: a critical value selector}\label{minmax}
We denote by $\mathcal{Q}_m$ the set of functions on $\rr^m$ that can be written as the sum of a nondegenerate quadratic form and of a Lipschitz function.
The aim of this appendix is to build a function $\sigma:\bigcup_{m\in \nn}\mathcal{Q}_m\to \rr$, named \emph{minmax}, satisfying:
\begin{enumerate}
	\item \label{critval2} if $f$ is $\cc^1$, then $\sigma(f)$ is a critical value of $f$,
	\item \label{add2} if $c$ is a real constant, then $\sigma(c+f)=c+\sigma(f)$,
	\item \label{transl2} if $\phi$ is a Lipschitz $\cc^\infty$-diffeomorphism on $\rr^m$ such that $f\circ \phi$ is in $\mathcal{Q}_m$, then \[\sigma(f \circ \phi )=\sigma(f),\]
	\item \label{monmin2} if $f_0-f_1$ is Lipschitz and $f_0 \leq f_1$ on $\rr^d$, then $\sigma(f_0)\leq \sigma(f_1)$,
	\item \label{loc2} 
	if $(f_\mu)_{\mu\in[s,t]}$ is a $\cc^1$ family of $\mathcal{Q}_m$ with $(\zz-f_\mu)_\mu$ equi-Lipschitz for some nondegenerate quadratic form $\zz$, then for all $\mu \neq \tilde{\mu}\in [s,t]$,
	\[ \min_{\mu \in [s,t]} \min_{x \in Crit(f_\mu)} \dd_\mu f_\mu(x) \leq  \frac{\sigma(f_{\tilde{\mu}})-\sigma(f_\mu)}{\tilde{\mu}-\mu}\leq \max_{\mu \in [s,t]} \max_{x \in Crit(f_\mu)}  \dd_\mu f_\mu(x). \] 
	\item \label{opp2} $\sigma(-f)=-\sigma(f)$,
	\item \label{cvx2} if $f(x,y)$ is a $\cc^2$ function of $\mathcal{Q}_m$ such that $\dd^2_y f \geq c\rm{id}$ for a $c>0$, and if $g(x)=\min_y f(x,y)$ is in some $\mathcal{Q}_{\tilde{m}}$, then $\sigma(g)=\sigma(f)$.
\end{enumerate}
For smooth functions, \eqref{critval}, \eqref{transl} and \eqref{add} are proved in Proposition \ref{prop},  \eqref{monmin} is implied by Proposition \ref{mon}, and \eqref{opp2} and \eqref{cvx2} are proved respectively in Propositions \ref{oppfield} and \ref{mmaxcvx}. They are extended to non smooth functions in Propositions \ref{propc0} and \ref{propc1}, and 
\eqref{loc2} is proved in Proposition \ref{locloc}.

\begin{csqs}\label{csqs} 
	These properties imply the following consequences: \begin{enumerate}
		\item \label{conti2} If $f$ and $g$ are two functions of $\mathcal{Q}_m$ with difference bounded and Lipschitz on $\rr^m$, then $
		|\sigma(f)-\sigma(g)| \leq \|f-g\|_\infty.$ This is a consequence of properties \eqref{add2} and \eqref{monmin2}.
		\item \label{mmaxstab2} If $g(x,\eta)=f(x)+\zz(\eta)$ where $\zz$ is a nondegenerate quadratic form and $f$ is in $\mathcal{Q}_m$, then $\sigma(g)=\sigma(f)$. This is a consequence of properties \eqref{opp2} and \eqref{cvx2} for smooth functions, which may be extended by continuity thanks to the previous point.
		\item \label{mmaxquad2}
		If $f_\mu=\zz_\mu + \ell_\mu$ is a $\cc^1$ family of $\qq_m$ with $\ell_\mu$ equi-Lipschitz, such that the set of critical points $f_\mu$ does not depend on $\mu$ and such that $\mu\mapsto f_\mu$ is constant on this set, then $\mu \mapsto \sigma(f_\mu)$ is constant. This is a consequence of properties \eqref{transl2} and \eqref{loc2}.
		\item \label{coerc} If $f$ is bounded below, then $\sigma(f)=\min(f)$. This is a consequence of properties \eqref{critval2} and \eqref{monmin2}.
	\end{enumerate}
\end{csqs}
Consequences \ref{csqs}-\eqref{mmaxquad2} and \ref{csqs}-\eqref{coerc} are proved in the main corpus, see respectively Consequences \ref{mmaxquad} and \ref{coercmin}.

The construction of such a critical value selector proves Propositions \ref{mmax} and \ref{mmaxbis}. 

We will use two deformation lemmas proved in Appendix \ref{deformation}, and we refer to \cite{weiminmax} for a survey of minmax related subtleties, including an example due to F. Laudenbach where the minmax is not uniquely defined.

\begin{rem}\label{discvit} In this paper we describe the geometric solution associated with the considered Cauchy problem with a particular generating family proposed by Chaperon. In a more general setting, Viterbo's uniqueness theorem on generating functions state that if $S$ and $\tilde{S}$ are two generating functions quadratic at infinity describing a same Lagrangian submanifold which is Hamiltonianly isotopic to the zero section, they may be obtained one from another via a combination of the three following transformations:
	\begin{itemize}
		\item Addition of a constant:  $\tilde{S} = S + c$ for some $c\in\rr$,
		\item Diffeomorphism operation: $\tilde{S}=S\circ \phi$ for some fiber $\cc^\infty$-diffeomorphism $\phi$,
		\item Stabilization: $\tilde{S}(x,\xi,\nu)=S(x,\xi)+\zz(\nu)$ for a nondegenerate quadratic form $\zz$.
	\end{itemize}
	The proof of D. Theret in \cite{theret} puts forward the fact that the diffeomorphism $\phi$ may be chosen affine outside a compact set - in particular such a diffeomorphism is Lipschitz and if $f$ is in $\mathcal{Q}_m$, so does $f\circ \phi$. Hence, the invariance of the minmax by additivity (property \eqref{add2}), by diffeomorphism action (property \eqref{transl2}) and by stabilization (property \ref{csqs}-\eqref{mmaxstab2}) gives that the minmax behave well when applied to generating functions. Up to adding a constant, it is the same for generating functions describing the same Lagrangian submanifold.
\end{rem}

\subsection{Definition of the minmax for smooth functions}$ $\\

Let us denote by $\mathcal{Q}^\infty_m$ the set of $\cc^\infty$ functions of $\mathcal{Q}_m$. The critical points and values of $\cc^1$ functions of $\mathcal{Q}_m$ are bounded:

\begin{prop}\label{vc} If $\zz$ is a nondegenerate quadratic form and $\ell$ is a $\cc^1$ Lipschitz function with constant $L$, then the set of critical points of the function $f=\zz+\ell$ is closed and contained in the ball $\bar{B}(0,L/m)$ where $m=\inf_{\|x\|=1} \|d\zz(x)\|$. The set of critical values of $f$ is hence closed and bounded.
\end{prop}

\begin{nota} For $f$ a function and $c$ a real number, let $f^c = \left\lbrace x \in \rr^m | f(x)\leq c \right\rbrace$ be the sublevel set of $f$ associated with the value $c$. Note that $f^c \subset f^{c'}$ if $c \leq c'$.
\end{nota}

\begin{defi} Let $f$ be a function of $\mathcal{Q}^\infty_m$ and $a$ be a real constant. Since the critical values of $f$ are bounded, we can find $c\geq|a|$ greater than any critical values of $f$ in modulus.
	For $a \leq c$, let $i^c_a$ be the canonical injection \begin{displaymath}
	(f^a,f^{-c}) \hookrightarrow (f^c,f^{-c}).
	\end{displaymath}		
	It induces a morphism $i^{c\star}_a$ in relative cohomology:\begin{displaymath}
	H^\bullet(f^c,f^{-c}) \overset{i^{c\star}_a}{\rightarrow} H^\bullet(f^a,f^{-c}).
	\end{displaymath}
	We assume that the cohomology is calculated with coefficients in a field, which allows to choose a simplified definition.
	
	Let the \textit{minmax} of $f$ be defined by\begin{displaymath}
	\sigma(f)=\inf\left\{a \in \rr | i^{c\star}_a \neq 0 \right\}=\sup\left\{a \in \rr | i^{c\star}_a = 0 \right\}.
	\end{displaymath}
	This definition does not depend on the choice of $c$ when $c$ is large enough.
\end{defi}

\begin{proof}
	The fact that $\sigma(f)$ does not depend on the choice of $c$ when it is large enough is a consequence of the following lemma:
	
	\begin{lem}\label{infty} If $c_1\geq |a|$ and $c_2 \geq a$ are two real constants greater than any critical values of $f$ in modulus, $i^{c_1\star}_a$ and $i^{c_2\star}_a$ are conjugate in cohomology. 
		Therefore they are simultaneously zero or non-zero.
	\end{lem}
	
	\begin{proof} Suppose $c_2>c_1$. If $a=-c_1$, let us check that $i^{c_1\star}_a=i^{c_2\star}_a=0$:\begin{displaymath}
		H^\bullet(f^{c_1 },f^{-c_1}) \overset{i^{c_1\star}_a}{\rightarrow} \underbrace{H^\bullet(f^{-c_1},f^{-c_1})}_{=\{0\}}
		\end{displaymath}
		and therefore $i^{c_1\star}_{-c_1}=0$. 	We can prove that $i^{c_2\star}_{-c_1}=0$ in the same way:
		\begin{displaymath}
		H^\bullet(f^{c_2 },f^{-c_2}) \overset{i^{c_2\star}_a}{\rightarrow} \underbrace{H^\bullet(f^{-c_1},f^{-c_2})}_{=\{0\}}
		\end{displaymath}
		where the nullity of $H^\bullet(f^{-c_1},f^{-c_2})$ is guaranteed by the retraction constructed in Lemma \ref{deform}.
		
		Now, if $a>-c_1$, there is an $\ep>0$ such that $-c_1+\ep \leq a$, and $f$ has no critical value in $[-c_2-\ep,-c_1+\ep]$ or in $[c_1-\ep,c_2+\ep]$. Since $-c_1 + \ep \leq a \leq c_1$, Deformation lemma \ref{deform} gives two homotopy equivalences $\Phi_+$ and $\Phi_-$ such that:\begin{displaymath}
		\left\{\begin{array}{c}
		\Phi_+(f^{c_2})=f^{c_1}\\
		\Phi_+(f^a)=f^a
		\end{array}\right. \;\textrm{ and } \;
		\left\{\begin{array}{c}
		\Phi_-(f^{-c_1})=f^{-c_2}\\
		\Phi_-(f^a)=f^a.
		\end{array}\right.
		\end{displaymath} The homotopy equivalences give isomorphisms in cohomology, and the following diagram commutes:\begin{displaymath}
		\begin{array}{c c c}
		H^\bullet (f^{c_1},f^{-c_1}) &\overset{i^{c_1\star}_{a}}{\to} & H^\bullet(f^a,f^{-c_1})\\
		\wr\downarrow (\Phi_+^\star)^{-1} & & \wr\downarrow (\Phi_+^\star)^{-1} \\
		H^\bullet (f^{c_2},f^{-c_1}) & \comm & H^\bullet(f^a,f^{-c_1})\\
		\wr\downarrow \Phi_-^\star & & \wr\downarrow \Phi_-^\star \\
		H^\bullet (f^{c_2},f^{-c_2}) &\underset{i^{c_2\star}_{a}}{\to} & H^\bullet(f^a,f^{-c_2})
		\end{array}
		\end{displaymath}
		which proves that $i^{c_1\star}_a$ and $i^{c_2\star}_a$ are conjugate in cohomology.
	\end{proof}
	
	Let us now fix $c$ large enough and prove that $\inf\left\{a \in \rr | i^{c\star}_a \neq 0 \right\}=\sup\left\{a \in \rr | i^{c\star}_a = 0 \right\}$. To do so, we are going to prove that any element of the set $\left\{a \in \rr | i^{c\star}_a \neq 0 \right\}$ is bigger than any element of its complement set $\left\{a \in \rr | i^{c\star}_a = 0 \right\}$.
	Let $a$ be such that $i^{c\star}_a \neq 0$ and $b$ be such that $i^{c\star}_b =0$. Assume that $b>a$. The following diagram commutes:\begin{displaymath}
	\begin{array}{c c c}
	(f^a,f^{-c}) &\overset{i}{\hookrightarrow} &(f^b,f^{-c})\\
	&\underset{i^c_a}{\searrow}\;\; \comm   &   \downarrow i^c_b\\
	&				& (f^c,f^{-c})\end{array}
	\end{displaymath}			
	where $i$ denotes the canonical injection from $(f^a,f^{-c})$ to $(f^b,f^{-c})$. It induces a commutative diagram in cohomology:\begin{displaymath}
	\begin{array}{c c c}
	H^\bullet(f^a,f^{-c}) &\overset{i^\star}{\leftarrow} &H^\bullet(f^b,f^{-c})\\
	&\underset{i^{c\star}_a}{\nwarrow} \;\; \comm  &   \uparrow i^{c\star}_b\\
	&				& H^\bullet(f^c,f^{-c})\end{array}
	\end{displaymath}
	Since $i^{c\star}_b$ is zero, $i^{c\star}_a$ is necessarily zero which is excluded. We have proved that $a\geq b$ (and then $a>b$ since $i^{c\star}_a \neq i^{c\star}_b$), and consequently:\begin{displaymath}
	\inf\left\{a \in \rr | i^{c\star}_a \neq 0 \right\}=\sup\left\{a \in \rr | i^{c\star}_a = 0 \right\}.
	\end{displaymath}
\end{proof}

\begin{theo}\label{vac} The minmax $\sigma(f)$ is a critical value of $f$.\end{theo}
\begin{proof} Suppose that $\sigma(f)$ is not a critical value of $f$. Then, since the set of critical values of $f$ is closed (see Proposition \ref{vc}), there is a $\ep>0$ such that $f$ has no critical value in $[\sigma(f)-\ep,\sigma(f)+\ep]$. Since $\sigma(f)$ is finite, by definition, there exist $a$ and $b$ such that $\sigma(f) -\ep < a \leq \sigma(f) \leq b < \sigma(f)+\ep$, $i^\star_a=0$ and $i^\star_b \neq 0$. Taking $c$ strictly bigger than $|a|$, $|b|$ and any critical value of $f$, Proposition \ref{infty} states that $i^{c\star}_a=0$ and $i^{c\star}_b \neq 0$.
	
	One can find an $\ep'>0$ such that $[a-\ep',b+\ep']\subset[\sigma(f)-\ep,\sigma(f)+\ep]$ and $b+\ep' \leq c$, so that $[a- \ep',b+\ep']$ does not contain any critical point of $f$, and Deformation lemma \ref{deform} builds a continuous function $\Phi$ such that $\Phi(f^b,f^{-c})=(f^a,f^{-c})$ and also $\Phi(f^c,f^{-c})=(f^c,f^{-c})$ since $b+\ep' \leq c$. Since $\Phi$ is a homotopy equivalence, it defines an isomorphism in cohomology. The following diagram should then commute:\begin{displaymath}
	\begin{array}{c c c}
	H^\bullet (f^c,f^{-c}) &\overset{i^{c\star}_{a}=0}{\to} & H^\bullet(f^a,f^{-c})\\
	\wr\downarrow \Phi^\star & \comm & \wr\downarrow \Phi^\star \\
	H^\bullet (f^c,f^{-c}) &\underset{i^{c\star}_{b}\neq 0}{\to} & H^\bullet(f^b,f^{-c})
	\end{array}
	\end{displaymath}
	which is impossible. Hence, $\sigma(f)$ is necessarily a critical value of $f$.
\end{proof}

\subsection{Minmax properties for smooth functions}

\begin{prop}\label{prop} Let $f$ be in $\mathcal{Q}^\infty_m$. Then the minmax satisfies:\begin{enumerate}
		\item $\sigma(f)$ is a critical value of $f$,
		\item \label{addan} if $c$ is a real number, $\sigma(c+f)=c+\sigma(f)$,
		\item \label{trans} if $\phi$ is a Lipschitz $\cc^\infty$-diffeomorphism on $\rr^m$ such that $f\circ \phi$ is in $\mathcal{Q}_m$, then \[\sigma(f \circ \phi )=\sigma(f).\]
	\end{enumerate}
\end{prop}
\begin{proof}\begin{enumerate}
		\item has already been proved (see Theorem \ref{vac}).
		\item If $b> 0$ is a real number, $g=b+f$ is in $\mathcal{Q}^\infty_m$. For all $c \in \rr$, $f^c=g^{c+b}$. Choose $c$ big enough so that $c-2b$ is strictly greater than $|a|$ and than the critical values of $f$. Take $a$ in $\rr$ and let us show that $i^{c,f\star}_{a}\neq 0 \Longleftrightarrow i^{c-b,g\star}_{a+b}\neq 0$. 
		There is an $\ep>0$ such that $f$ has no critical value of $f$ in $[c+\ep,c-2b-\ep]$. Now take the homotopy equivalence constructed in Lemma \ref{deform} and satisfying:
		\begin{displaymath}
		\left\{\begin{array}{c}
		\Phi(f^{c})=f^{c-2b}\\
		\Phi(f^u)=f^u \;\; \forall u \leq c-2b.
		\end{array}\right. 
		\end{displaymath}
		
		This gives the following commutative diagram, since $a$ and $-c$ are smaller than $c-2b$:\begin{displaymath}
		\begin{array}{c c c }
		H^\bullet(f^c,f^{-c})&\overset{i^{c,f\star}_{a}}{\to} & H^\bullet (f^a,f^{-c}) \\
		\wr\uparrow \Phi^\star &  & \wr\uparrow \Phi^\star \\
		H^\bullet(f^{c-2b},f^{-c}) &\comm & H^\bullet(f^{a},f^{-c})\\
		\parallel & & \parallel\\
		H^\bullet(g^{c-b},g^{-c+b}) &\underset{i^{(c-b),g\star}_{a+b}}{\to}  &  H^\bullet (g^{a+b},g^{-c+b}) 
		\end{array}
		\end{displaymath}
		which proves that $i^{c,f\star}_{a}=0 \Longleftrightarrow i^{(c-b),g\star}_{a+b}=0$. But since the critical values of $g$ are the critical values of $f$ added to the constant $b$, $c-b$ is greater than any critical value of $g$ in modulus since $c-2b$ is greater in modulus than the critical values of $f$. Lemma \ref{infty} states that the nullity of $i^{c,f\star}_{a}$ (resp. $i^{c,g\star}_{a}$) does not depend on $c$ large enough, hence: 	\[\sigma(f)=\inf\left\{a \in \rr | i^{c,f\star}_a \neq 0\right\}=\inf\left\{a \in \rr | i^{(c-b),g\star}_{a+b} \neq 0\right\}=\sigma(g)-b.\]
		
		\item Let $\phi$ be a Lipschitz $\cc^\infty$-diffeomorphism of $\rr^m$ such that $g=f \circ \phi$ is in $\mathcal{Q}_m^\infty$. Note that $f$ and $g$ have the same critical values. Take $a$ in $\rr$ and $c\geq |a|$ greater than any critical value of $f$ (hence $g$).
		
		For all $u \in \rr$, $f^u=\phi(g^u)$. Since $\phi$ is a $\cc^\infty$-diffeomorphism mapping the pair $(g^{u'},g^u)$ to $(f^{u'},f^u)$ for all real numbers $u<u'$, $\phi$ gives an isomorphism in cohomology. The following diagram commutes:\begin{displaymath}
		\begin{array}{c c c}
		H^\bullet (f^c,f^{-c}) &\overset{i^{c,f\star}_{a}}{\to} & H^\bullet(f^a,f^{-c})\\
		\wr\downarrow \phi^\star & \comm & \wr\downarrow \phi^\star\\
		H^\bullet (g^c,g^{-c}) &\underset{i^{c,g\star}_{a}}{\to}& H^\bullet(g^a,g^{-c})
		\end{array}
		\end{displaymath}
		which shows that $i^{c,f\star}_{a}\neq 0 \Longleftrightarrow i^{c,g\star}_{a}\neq 0$, hence $\sigma(f)=\sigma(g)$.
	\end{enumerate}
\end{proof}

Now let us focus on the monotonicity of the minmax. 

\begin{defi} If $f_0$ and $f_1$ are two functions of $\mathcal{Q}^\infty_m$ with Lipschitz difference, let us consider the homotopy $f_t = (1-t)f_0 + tf_1$ between $f_0$ and $f_1$ and denote by $\CC_{f_0,f_1}$ the set of critical points $\CC_{f_0,f_1} = \left\{ x \in \rr^m | \exists t \in [0,1], df_t(x)=0\right\}$. 
\end{defi}

\begin{prop}\label{CCcompact} Under these assumptions, the set $\CC_{f_0,f_1}$ is compact.
\end{prop}

\begin{proof}  Let us denote by $f_0=\zz+\ell_0$ and $f_1=\zz+\ell_1$. If $L$ is a Lipschitz constant suiting both $\ell_0$ and $\ell_1$, note that $\ell_0 +t(\ell_1-\ell_0)$ is also $L$-Lipschitz. The critical points of $f_t$ are hence in the ball $\bar{B}(0,L/m)$ by Proposition \ref{vc}, and $\CC_{f_0,f_1}$ is a bounded set.
	
	Let $(x_n)$ be a converging sequence of $\CC_{f_0,f_1}$ and denote by $x$ its limit. By definition of $\CC_{f_0,f_1}$, there is a sequence $(t_n)\in [0,1]$ such that $df_{t_n}(x_n)=0$ for all $n$. Since $(t_n)$ is bounded, it is possible to find a subsequence of $t_n$ converging to some $t \in [0,1]$. Since $(t,x)\mapsto f_t(x)$ is $\cc^1$, $df_t(x)$ is zero, and $\CC_{f_0,f_1}$ is closed.
\end{proof}

\begin{prop}\label{mon} Let $f_0$ and $f_1$ be two functions of $\mathcal{Q}^\infty_m$ with Lipschitz difference. If $U$ is a set containing $\CC_{f_0,f_1}$ and $f_0 \geq f_1$ on $U$, then $\sigma(f_0) \geq \sigma(f_1)$. In particular if $f_0 \geq f_1$ on $\CC_{f_0,f_1}$ (or if $f_0 \geq f_1$ on $\rr^m$), then $\sigma(f_0) \geq \sigma(f_1)$.
\end{prop}

\begin{csq}\label{contan} If $f_0$ and $f_1$ are two functions of $\mathcal{Q}^\infty_m$ with Lipschitz difference: \begin{displaymath}
	\underset{U}{\inf}(f_0-f_1)\leq \underset{\CC_{f_0,f_1}}{\inf}(f_0-f_1)\leq	\sigma(f_0)-\sigma(f_1) \leq \underset{\CC_{f_0,f_1}}{\sup}(f_0-f_1)\leq \underset{U}{\sup}(f_0-f_1).
	\end{displaymath}
	for each set $U$ containing the set $\CC_{f_0,f_1}$. In particular if $f_0-f_1$ is Lipschitz and bounded on $\rr^m$, then ${|\sigma(f_0)-\sigma(f_1)|\leq \|f_0-f_1\|_\infty}$.
\end{csq}

\begin{proof} Since $f_1 + \underset{\CC_{f_0,f_1}}{\inf}(f_0-f_1)\leq  f_0\leq f_1 + \underset{\CC_{f_0,f_1}}{\sup}(f_0-f_1)$ on $\CC_{f_0,f_1}$ and the three functions are in $\mathcal{Q}^\infty_m$ with Lipschitz difference,  Proposition (\ref{mon}) gives\begin{displaymath}
	\sigma(f_1+\underset{\CC_{f_0,f_1}}{\inf}(f_0-f_1)) \leq 	\sigma(f_0) \leq  \sigma(f_1+\underset{\CC_{f_0,f_1}}{\sup}(f_0-f_1)).
	\end{displaymath}
	The additivity (\ref{prop}-\ref{addan}) then concludes:\begin{displaymath}
	\underset{\CC_{f_0,f_1}}{\inf}(f_0-f_1)	\leq \sigma(f_0)-\sigma(f_1) \leq \underset{\CC_{f_0,f_1}}{\sup}(f_0-f_1).\end{displaymath}
\end{proof}

\begin{proof} Let us first prove Proposition \ref{mon} in the case of an open and bounded set $U$.
	Take $a$ in $\rr$ and $C=\underset{t \in [0,1]}{\max} \underset{U}{\sup} |f_t| $, and choose a $c$ bigger than $C$ and $|a|$. Note that $c$ is bigger in modulus than the critical values of $f_0$ and $f_1$ (which are contained in $U$). Lemma \ref{deff} gives a $\cc^1$-diffeomorphism $\Psi:(f_0^c,f_0^{-c})\to(f_1^c,f_1^{-c})$, sending the pair $(f_0^a,f_0^{-c})$ into the pair $(f_1^a,f_1^{-c})$ (since $\Psi(f_0^a) \subset f_1^a$ and $\Psi(f_0^{-c})=f_1^{-c}$). This results in the following commutative diagram:\begin{displaymath}
	\begin{array}{c c c}
	H^\bullet(f_1^c,f_1^{-c}) & \overset{i^{c,f_1\star}_{a}}{\to} & H^\bullet(f_1^a,f_1^{-c})\\
	\wr\downarrow \Psi^\star & \comm & \downarrow \Psi^\star\\
	H^\bullet(f_0^c,f_0^{-c}) & \underset{i^{c,f_0\star}_{a}}{\to} & H^\bullet(f_0^a,f_0^{-c})\end{array}
	\end{displaymath}
	Hence, if $i^{c,f_1\star}_a$ is zero, since the left arrow is one-to-one, $i^{c,f_0\star}_a$ is necessarily zero. This proves that $\{a\in \rr | i^{c,f_0\star}_{a} \neq 0 \} \subset \{a\in \rr | i^{c,f_1\star}_{a} \neq 0 \}$ and then $\sigma(f_1) \leq \sigma(f_0)$.
	
	Now, if $U$ is not open anymore, but bounded, it is contained for all $\delta>0$ in the open and bounded set $U_\delta = \{x\in \rr^d| d(x,U)<\delta\}$. Furthermore since $f_0\geq f_1$ on $U$ and since $U_\delta$ is bounded, we have by continuity of $f_0$ and $f_1$ that $f_0 \geq f_1 + w(\delta)$ on $U_\delta$ with $w(\delta) \to 0$ when $\delta \to 0$. The previous work states that $\sigma(f_0)\geq \sigma(f_1+w(\delta))=\sigma(f_1)+w(\delta)$ by additivity of the minmax, and letting $\delta$ tend to $0$ finishes the proof.
	
	Finally, we get rid of the boundness assumption by observing that since $\CC_{f_0,f_1}$ is compact (Proposition \ref{CCcompact}), we may always replace $U$ by the intersection of $U$ with a ball large enough to contain $\CC_ {f_0,f_1}$, which ends the proof. 
\end{proof}

\begin{prop}\label{oppfield}
	If the cohomology is calculated with coefficients in a field, $\sigma(-f)=-\sigma(f)$  for each function $f$ of $\mathcal{Q}^\infty_m$.
\end{prop}
\begin{proof}
	If $f$ is in $\mathcal{Q}^\infty_m$ with an associate nondegenerate form $\zz$ of index $\lambda$, take $c$ bigger in modulus than the critical values of $f$. The homology calculation for the quadratic form gives that $H_k (f^c,f^{-c})=0$ if $k\neq \lambda$ and $H_\lambda(f^c,f^{-c})$ is one dimensional. In particular, if the homology is calculated with coefficients in a field, the homology morphism $i^{c}_{a\star}:H_\bullet(f^a,f^{-c})\to H_\bullet(f^c,f^{-c})$  induced by $i^c_a$ is non zero if and only if it is one-to-one. Since $i^{c}_{a\star}$ is the transposition of $i^{c\star}_{a}$, they are simultaneously non zero.
	
	Alexander duality gives the following commutative diagram, with exact columns:
	\begin{displaymath}
	\begin{array}{c c c}
	H_\bullet (f^a,f^{-c}) &\simeq & H^\bullet(\rr^m\setminus f^{-c},\rr^m\setminus f^{a})\\
	i^c_{a\star} \downarrow  & \comm & \downarrow \\
	H_\bullet (f^c,f^{-c}) &\simeq& H^\bullet(\rr^m\setminus f^{-c},\rr^m\setminus f^{c})\\
	\;\;\,\;\;\,\,\downarrow & \comm & \downarrow \\
	H_\bullet (f^c,f^{a}) &\simeq & H^\bullet(\rr^m\setminus f^{a},\rr^m\setminus f^{c})
	\end{array}
	\end{displaymath}
	If $a$ is not a critical value of $f$, for $\ep>0$ small enough $\rr^m\setminus f^a = \{-f<-a\}$ retracts on $-f^{-a-\ep}$ via the homotopy equivalence constructed in Lemma \ref{deform}, just as $-f^{-a}$. The same can be done for $c$ and $-c$, and composing the cohomology induced isomorphisms we get an isomorphism $\Phi^\star$, completing the previous diagram as follows:
	\begin{displaymath}
	\begin{array}{c c c c c }
	H_\bullet (f^a,f^{-c}) &\simeq & H^\bullet(\rr^m\setminus f^{-c},\rr^m\setminus f^{a})& &\\
	i^c_{a\star} \downarrow  & \comm & \downarrow& & \\
	H_\bullet (f^c,f^{-c}) &\simeq& H^\bullet(\rr^m\setminus f^{-c},\rr^m\setminus f^{c})& \overset{\Phi^\star}{\simeq}& H^\bullet((-f)^{c},(-f)^{-c})\\
	\;\;\,\;\;\,\,\downarrow & \comm & \downarrow & \comm & \downarrow(i^{-a}_{c,-f})^\star\\
	H_\bullet (f^c,f^{a}) &\simeq & H^\bullet(\rr^m\setminus f^{a},\rr^m\setminus f^{c})
	& \underset{\Phi^\star}{\simeq}& H^\bullet((-f)^{-a},(-f)^{-c})
	\end{array}
	\end{displaymath}
	If $a$ is larger than $\sigma(f)$, $i^{c\star}_a$ is non zero, hence $i^c_{a\star}$ is non zero and it is then one-to-one. Since the first column is exact, this implies that $(i^{-a}_{c,-f})^\star$ is zero, hence $-a \leq \sigma(-f)$. This being true for each $a$ larger than $\sigma(g)$, it comes that $-\sigma(f)\leq \sigma(-f)$.
	
	If $a$ is smaller than $\sigma(f)$, $i^{c\star}_a$, hence $i^c_{a\star}$, are zero and it follows that $(i^{-a}_{c,-f})^\star$ is non zero, hence $-a \geq \sigma(-f)$. As before this implies that $-\sigma(f)\geq \sigma(-f)$, and the result holds.
\end{proof}

\begin{rem}\label{cvsuniq} 
	The proof of Proposition \ref{oppfield} is the only place where we need to work with coefficients in a field.
\end{rem}

\begin{prop}\label{mmaxcvx}
	If $f:(x,y)\in \rr^d \times \rr^k\to \rr$ is a function of $\mathcal{Q}^\infty_{d+k}$ such that $\dd^2_y f \geq c\rm{id}$ for some $c>0$, and if $g(x)=\min_y f(x,y)$ is in $\mathcal{Q}_d$, then $\sigma(g)=\sigma(f)$.
\end{prop}
\begin{proof} If $\dd^2_y f \geq c\rm{id}$, $y \mapsto f(x,y)$ attains for each $x$ a strict minimum at a point $y(x)$ and $x \mapsto y(x)$ is $\cc^1$ by implicit differentiation of $\dd_y f(x,y(x))=0$. 	Note that $g(x)=f(x,y(x))$ and $f$ have the same critical values and choose $c$ larger in modulus than these critical values. 
	
	We denote by $\tilde{g}^a$ the set $\left\{(x,y(x))|g(x)\leq a\right\}$. It is the restriction of the graph of $x \mapsto y(x)$ on $g^a$. Hence $\Psi : x \mapsto (x,y(x))$, which is a $\cc^1$-diffeomorphism from $\rr^d$ to the graph of $x\mapsto y(x)$, maps for all $a$ $g^a$ on $\tilde{g}^a$, and it induces an isomorphism in relative cohomology.
	
	For all $a$ in $\rr$, the sublevel set $f^a$ retracts to $\tilde{g}^a$ via  $\Phi_t(x,y)=(x,(1-t)y + ty(x))$ which is a deformation retraction. One can indeed check, using the convexity of $y \mapsto f(x,y)$ and the fact that $y(x)$ is the minimum of this function, that:\[\left\{\begin{array}{l}
	\Phi_0 = {\rm id},\\
	\Phi_1(f^a)\subset \tilde{g}^a,\\
	\Phi_t(f^a)\subset f^a \;\,\forall t \in [0,1],\\
	\Phi_t= {\rm id} \textrm{ on } \tilde{g}^a.
	\end{array}\right.\]
	
	Since this retraction does not depend on $a$, the following diagram commutes:\begin{displaymath}
	\begin{array}{c c c}
	H^\bullet (f^c,f^{-c}) &\overset{i^{c,f\star}_{a}}{\to} & H^\bullet(f^a,f^{-c})\\
	\wr\uparrow \Phi_1^\star & & \wr\uparrow \Phi_1^\star\\
	H^\bullet (\tilde{g}^c,\tilde{g}^{-c}) &\comm & H^\bullet(\tilde{g}^a,\tilde{g}^{-c})\\
	\wr\uparrow \Psi^{-1\star} & & \wr\uparrow \Psi^{-1\star}\\
	H^\bullet (g^c,g^{-c}) &\underset{i^{c,g\star}_{a}}{\to} & H^\bullet(g^a,g^{-c})
	\end{array}
	\end{displaymath}
	Hence $i^{c,g\star}_{a}$ and $i^{c,f\star}_{a}$ are simultaneously nonzero and therefore $\sigma(g)=\sigma(f)$.
\end{proof}			

\subsection{Extension to non-smooth functions}
From now on the aim is to extend by continuity the definition and properties of the minmax to non-smooth functions. 

\begin{defi} If $f$ is in $\mathcal{Q}_m$, there exists by definition a nondegenerate quadratic form $\zz$ and a Lipschitz function $\ell$ such that $f=\zz+\ell$. Since $\ell$ is Lipschitz, there exists an equi-Lipschitz sequence $(\ell_n)$ of $\cc^\infty$ functions such that $\ell_n$ converge uniformly towards $\ell$. Then the minmax of $f=\zz+\ell$ is defined by\begin{displaymath}
	\sigma(f)=\lim_{n\to \infty} \sigma(\zz+\ell_n).
	\end{displaymath}
	This does not depend on the choice of $(\ell_n)$.
\end{defi}

\begin{proof} Let us show that the limit exists, and that it does not depend on the choice of the sequence $(\ell_n)$.
	\begin{itemize}
		\item Let $\ep>0$ be fixed. Since $\ell_n$ converges uniformly, it is a Cauchy sequence and there is a $N>0$ such that:\begin{displaymath}
		\|\ell_n-\ell_m\|_\infty\leq \ep \;\; \forall n,m \geq N.\end{displaymath}
		Then, since $\zz+\ell_n$ and $\zz+\ell_m$ are in $\mathcal{Q}^\infty_m$ with Lipschitz and bounded difference, Consequence \ref{contan} gives: \begin{displaymath}
		|\sigma(\zz+\ell_n)-\sigma(\zz+\ell_m)|\leq \|\ell_n-\ell_m\|_\infty\leq \ep \;\; \forall n,m \geq N
		\end{displaymath}
		and $(\sigma(\zz+\ell_n))$ is a Cauchy sequence in $\rr$, hence has a limit denoted $\sigma(f)$.
		\item Let $(\ell_n)$ and $(\tilde{\ell}_n)$ be two equi-Lipschitz sequences of $\cc^\infty$ functions, and assume that $\ell_n$ and $\tilde{\ell}_n$ admit the same uniform limit $\ell$. Let us show that $\sigma(\zz+\ell_n)$ and $\sigma(\zz+\tilde{\ell}_n)$ tend to the same limit. 
		
		Let $\ep>0$. Since $\ell_n$ and $\tilde{\ell}_n$ have the same limit, there is a $N>0$ such that:\begin{displaymath}
		\|\ell_n-\tilde{\ell}_n\|_\infty\leq \ep \;\; \forall n \geq N.
		\end{displaymath}
		Then, since $\zz+\ell_n$ and $\zz+\tilde{\ell}_n$ are in $\mathcal{Q}^\infty_m$ with Lipschitz and bounded difference, Consequence \ref{contan} gives:
		\begin{displaymath}
		|\sigma(Q+\ell_n)-\sigma(Q+\tilde{\ell}_n)|\leq \ep \;\; \forall n \geq N.
		\end{displaymath}
		Letting $n$ tend to $\infty$ shows that the limit does not depend on the choice of the sequence $(\ell_n)$.
	\end{itemize}
\end{proof}
Let us gather the properties satisfied for continuous functions of $\mathcal{Q}_m$:
\begin{prop}\label{propc0} If $f$ is in $\mathcal{Q}_m$, the properties of the smooth minmax still hold:\begin{enumerate}
		\item if $c$ is a real constant, then $\sigma(c+f)=c+\sigma(f)$,
		\item if $f_0 \leq f_1$ on $\rr^m$ and if $f_1-f_0$ is Lipschitz, then $\sigma(f_0)\leq \sigma(f_1)$,
		\item if $\phi$ is a Lipschitz $\cc^\infty$-diffeomorphism on $\rr^m$ such that $f\circ \phi$ is in $\mathcal{Q}_m$, then \[\sigma(f \circ \phi )=\sigma(f),\]
		\item $\sigma(-f)=-\sigma(f)$.	
	\end{enumerate}
\end{prop}
\begin{proof}
	\begin{enumerate}
		\item It is enough to notice that if $\zz+\ell_n$ converges to $f$ as in the definition, then $\zz+\ell_n+c$ converges to $f+c$. Then, $\sigma(\zz+c +\ell_n)=c+\sigma(\zz+\ell_n)$ by the additivity property (\ref{prop}-\ref{addan}), and the statement holds when $n$ tends to $\infty$.
		
		\item If $f^0\leq f^1$ are in $\mathcal{Q}_m$ and if their difference is Lipschitz, then there exist two sequences of equi-Lipschitz $\cc^\infty$ functions $(\ell^0_n)$ and $(\ell^1_n)$ such that $\zz+\ell^0_n$ (resp. $\zz+\ell^1_n$) converges uniformly to $f^0$ (resp. $f^1$) with $\ell^0_n\leq \ell^1_n$ for $n$ big enough. Then, Proposition \ref{mon} states that $\sigma(\zz+\ell^0_n) \leq \sigma(\zz+\ell^1_n)$ for $n$ big enough, and the statement holds when $n$ tends to $\infty$.
		
		\item Since $\|(\zz+\ell_n)\circ \phi - f \circ \phi \|_\infty \leq \|\zz+\ell_n -f\|_\infty$, if $(\zz+\ell_n)$ converges uniformly to $f=\zz+\ell$, then $(\zz+\ell_n)\circ\phi$ converges uniformly to $f\circ \phi$. 
		Moreover, since $\phi$ is Lipschitz, $\ell_n\circ \phi$ and $\ell \circ \phi$ are (equi-)Lipschitz. Now since $f\circ \phi = \zz\circ \phi + \ell\circ \phi$ is in $\mathcal{Q}_m$ and $\ell \circ \phi$ is Lipschitz, $\zz \circ \phi$ is in $\mathcal{Q}^\infty_m$ (as $\zz$ is $\cc^\infty$) and the sequence $((\zz+\ell_n) \circ \phi)$ is still in $\mathcal{Q}^\infty_m$.
		
		Thus, $((\zz+\ell_n)\circ \phi)$ is a sequence converging uniformly to $f \circ \phi$, as required in the definition. Since Property (\ref{prop}-\ref{trans}) states that $\sigma((\zz+\ell_n)\circ\phi)=\sigma(\zz+\ell_n)$ for all $n$, the statement holds when $n$ tends to $\infty$.
		
		\item This is a direct consequence of Proposition \ref{oppfield}.
	\end{enumerate}
\end{proof}

\begin{prop}\label{propc1}\begin{enumerate} The properties involving critical elements hold for $\cc^1$ functions of $\mathcal{Q}_m$:
		\item If $f\in \mathcal{Q}_m$ is $\cc^1$, then $\sigma(f)$ is a critical value of $f$.
		\item \label{monan} If $f_0,f_1\in \mathcal{Q}_m$ are $\cc^1$ with Lipschitz difference, and $\CC_{f_0,f_1}$ is the set of critical points of the homotopy $f_t=(1-t)f_0+tf_1$, then \[\underset{\CC_{f_0,f_1}}{\inf}(f_0-f_1)\leq	\sigma(f_0)-\sigma(f_1) \leq \underset{\CC_{f_0,f_1}}{\sup}(f_0-f_1).\]
	\end{enumerate}
\end{prop}	

\begin{proof}	
	\begin{enumerate}
		\item If $f=\zz+\ell$ is $\cc^1$, then $\ell$ is $\cc^1$ and there exists an equi-Lipschitz sequence $(\ell_n)$  of $\cc^\infty$ functions such that $\ell_n$ uniformly converges towards $\ell$ and $d\ell_n$ converge uniformly towards $d\ell$, hence $\zz+\ell_n$ (resp. $d\zz+d \ell_n)$) uniformly converges towards $f$ (resp. $df$).
		
		For all $n$, $\sigma(\zz+\ell_n)$ is a critical value of $\zz+\ell_n$, hence there exists $x_n$ in $\rr^m$ such that
		$d\zz(x_n)+\ell_n(x_n)=0$ and $\sigma(\zz+\ell_n)=(\zz+\ell_n)(x_n)$.
		
		Since the sequence $(\ell_n)$ is equi-Lipschitz, the sequence $(x_n)$ is contained in the closed ball $\bar{B}(L/m)$ where $L$ denotes a Lipschitz constant suiting all $\ell_n$ and $m=\inf_{\|x\|=1} \|d\zz(x)\|$, see Proposition \ref{vc}.
		
		Hence $x_n$ admits a subsequence converging to some $x$ in $\rr^m$. On the one hand, since $d(\zz+\ell_n)$ converges uniformly towards $df$ and $d(\zz+\ell_n)(x_n)=0$, $x$ is a critical point of $f$. On the other hand, since  $\zz+\ell_n$ converges uniformly towards $f$ and $(\zz+\ell_n)(x_n)=\sigma(\zz+\ell_n)$, $f(x)=\sigma(f)$. Thus $\sigma(f)$ is a critical value of $f$.
		
		\item Take $f^0$ and $f^1$ in $\mathcal{Q}_m$, $\cc^1$ and with Lipschitz difference. There exists an equi-Lipschitz sequence $(\ell_n^0)$  of 
		$\cc^\infty$ functions such that $\zz+\ell_n^0$ (resp. $d\zz+d\ell_n^0$) converges uniformly to $f^0$ (resp. $df^0$). Note that if $t$ is in $[0,1]$, the sequence $(\ell_n^t)=(\ell_n^0+t(f^1-f^0))$ is equi-Lipschitz uniformly with respect to $t$, and $f_n^t=\zz+\ell_n^t$ converges uniformly to $f^t=(1-t)f^0+tf^1$, and the derivative sequence $(df_n^t)$ converges uniformly to $df^t$.
		
		For all $n$, Consequence \ref{contan} states that:\[
		\inf_{\CC_{f^0_n,f^1_n}} (f^1_n-f^0_n) \leq \sigma(f^1_n)-\sigma(f^0_n) \leq \sup_{\CC_{f^0_n,f^1_n}} (f^1_n-f^0_n) .
		\]
		Let us focus on the second inequality.
		Since $\CC_{f^0_n,f^1_n}$ is compact (Proposition \ref{CCcompact}), the supremum is attained at some $x_n$ in $\CC_{f^0_n,f^1_n}$. By definition of $\CC_{f^0_n,f^1_n}$, there exists a sequence $(t_n)$ of $[0,1]$ such that $x_n$ is a critical point of $f^{t_n}_n$.
		
		Now, since the sequence $(\ell^t_n)_n$ is equi-Lipschitz uniformly with respect to $t$, there exists a ball $B(0,R)$, where $R$ depends only on the Lipschitz constants and on $\zz$, containing $\CC_{f^0_n,f^1_n}$ for all $n$. The sequence $(t_n, x_n)$ is hence bounded and we may assume it converges to some $(t,x)$. Since $df^{t_n}_n$ converges uniformly towards $df^t$, the fact that $x_n$ is a critical point of $f^{t_n}_n$ implies that $x$ is a critical point of $f^t$, hence $x$ is in $\CC_{f^0,f^1}$. 
		
		But then letting $n$ tend to $\infty$ in \[
		\sigma(f^1_n)-\sigma(f^0_n)\leq \sup_{\CC_{f^0_n,f^1_n}} (f^1_n-f^0_n)=f^1_n(x_n)-f^0_n(x_n)
		\]
		gives that\[
		\sigma(f^1_n)-\sigma(f^0_n) \leq f^1(x)-f^0(x) \leq \sup_{\CC_{f^0,f^1}}(f^1-f^0),
		\]
		using first the uniform convergence of $f^1_n-f^0_n$ towards $f^1-f^0$ and then the fact that $x$ is in $\CC_{f^0,f^1}$.
	\end{enumerate}
\end{proof}

The next proposition is the improved version of Proposition \ref{propc1}-\eqref{monan} that we require in the definition of a critical value selector, see Definition \ref{mmax}.

\begin{prop} \label{locloc} Let $(f_t)_{t\in[0,1]}$ be a $\cc^1$ homotopy of $\mathcal{Q}_m$ such that there exists a nondegenerate quadratic function $\zz$ and an equi-Lipschitz family of $\cc^1$ functions $(\ell_t)_{t\in[0,1]}$ with $f_t=\zz + \ell_t$. Then for all $s\neq t$ in $[0,1]$
	\[ \min_{t \in [0,1]} \min_{x \in Crit(f_t)} \dd_t f_t(x) \leq  \frac{\sigma(f_t)-\sigma(f_s)}{t-s}\leq \max_{t \in [0,1]} \max_{x \in Crit(f_t)}  \dd_t f_t(x).  \]
\end{prop}

Let $(f_t)_{t\in[0,1]}$ be as in the proposition. Note that if $m=\inf_{\|x\|=1}\|d\zz(x)\|$, the critical points of $f_t$ are contained for each $t$ in the compact set $C=\bar{B}(0,L/m)$. The set
$\left\lbrace (t,x) , t \in [0,1], \dd_x f_t(x)=0\right\rbrace$ is also compact: it is contained in the bounded set $[0,1]\times C$ and is closed by continuity of $\dd_x f$ w.r.t. $t$ and $x$.	Both $\min_{t \in [0,1]} \min_{x \in Crit(f_t)} \dd_t f_t(x)$ and $\max_{t \in [0,1]} \max_{x \in Crit(f_t)}  \dd_t f_t(x)$ are hence attained, and we denote them respectively by $a$ and $b$.

\begin{lem}\label{claim}
	For all $\ep >0$, there exists $\alpha>0$ such that for all $t$ in $[0,1]$, $\|\dd_xf_t(x)\|\leq \alpha$ implies $ a-\ep \leq \dd_tf_t(x)\leq b+ \ep$.
\end{lem} 
\begin{proof}
	Assume that there exists an $\ep >0$ and a sequence $(t_n,x_n)$ such that $\|\dd_x f_{t_n}(x_n)\|\leq 1/n$ and $\dd_t f_{t_n}(x_n) \notin  (a+\ep,b+\ep)$.
	Since $f_{t_n}=\zz+\ell_{t_n}$, $\|\dd_x f_{t_n}(x_n)\|\geq m\|x_n\|-L$ and the sequence $x_n$ is necessarily bounded. Since $t_n$ is in $[0,1]$, there exists a subsequence of $(t_n,x_n)$ converging to some $(t,x)$. The continuity of $df$ gives then a contradiction at the point $(t,x)$.  
\end{proof}
\begin{proof}[Proof of Proposition \ref{locloc}]	
	Let us define \[
	w(\delta)=\underset{
		x \in C,|t-s|\leq \delta}{\sup} \left\{\dd_t f_s(x)-\dd_t f_t(x), \dd_x f_s(x)-\dd_x f_t(x)\right\}.\]	
	The continuity of $d f$ and the compacity of $C$ grants that $w(\delta) \to 0$ when $\delta \to 0$.
	
	Let us fix $\ep >0$ and prove that $(a-2\ep)(t-s)\leq \sigma(f_t)-\sigma(f_s)\leq (b+2\ep)(t-s)$ for all $s\leq t$ in $[0,1]$. Take $\alpha$ as in Lemma \ref{claim} and $\delta>0$ such that both $w(\delta)<\ep$ and $w(\delta)<\alpha$. We first show the result for $t-s \leq \delta$, and it is immediately extended to large $t-s$ by iteration.
	
	For all $x$ in $\rr^d$, we have\[
	(t-s)\inf_{\tau\in [s,t]} \dd_t f_\tau(x) \leq f_t(x)-f_s(x)\leq (t-s)\sup_{\tau\in [s,t]} \dd_t f_\tau(x) .
	\]
	Now if $\CC_{f_s,f_t}$ denotes the set of critical points of the functions $g_u=(1-u)f_s + u f_t$ for $u$ in $[0,1]$, on the one hand, one has that $\CC_{f_s,f_t} \subset C=\bar{B}(0,L/m)$, while on the other hand Proposition \ref{propc1}-\eqref{monan} states that:\[
	\inf_{\CC_{f_s,f_t}} (f_t-f_s ) \leq \sigma(f_t)-\sigma(f_s) \leq \sup_{\CC_{f_s,f_t}} (f_t-f_s ),
	\]
	which implies
	\[	(t-s) \inf_{\begin{array}{c}
		\tau \in [s,t]\\
		x \in \CC_{f_s,f_t}
		\end{array}}\dd_t f_\tau(x) \leq 	\sigma(f_t)-\sigma(f_s) 	\leq (t-s) \sup_{\begin{array}{c}
		\tau \in [s,t]\\
		x \in \CC_{f_s,f_t}
		\end{array}}\dd_t f_\tau(x).	\]
	Since $\CC_{f_s,f_t}$ and $[s,t]$ are compact, the right hand side supremum is attained for some $\tau$ and $x$, where $x$ is the critical point of a function $g_u=(1-u)f_s + u f_t$, and consequently satisfies $\dd_x f_s(x)=u(\dd_x f_s(x)-\dd_x f_t(x))$. Since $x$ is in $C$ and $u$ is in $[0,1]$, we get $\|\dd_x f_s(x)\|\leq w(|t-s|) \leq \alpha$ by definition of $w$ and $\delta$. Lemma \ref{claim} then implies that $\dd_t f_s(x) \leq b+\ep$.
	
	Now let us estimate $\dd_t f_\tau(x)$ : since $x$ is in $C$ and $w(\delta)\leq \ep$,
	\[\begin{split}
	\dd_t f_\tau(x) \leq \dd_t f_s (x) + w(|\tau-s|) \leq b + 2\ep.
	\end{split}
	\]
	Putting it altogether we get that for all $\ep>0$,\[\sigma(f_t)-\sigma(f_s) \leq f_t(x)-f_s(x)\leq (t-s)\dd_t f_\tau(x) \leq (t-s)(b+2\ep) \]
	for $t-s\leq \delta$, and hence for all $t$ and $s$. The same work for the left hand side infimum gives that for all $\ep>0$, 
	\[(t-s)(a-2\ep)\leq\sigma(f_t)-\sigma(f_s)\leq (t-s)(b+2\ep),\]
	and letting $\ep$ tend to $0$ gives the wanted estimate.\end{proof}

\section{Deformation lemmas}\label{deformation}
\subsection{Global deformation of sublevel sets}
We still work with functions of $\mathcal{Q}^\infty_m$, \textit{i.e.} with functions that can be written as the sum of a nondegenerate quadratic function and of a $\cc^\infty$ Lipschitz function. 

\begin{lem}[Strong deformation retraction] \label{deform} Let $f$ be a function of $\mathcal{Q}_m^\infty$. Take $\ep>0$ and $a<b$ in $\rr$. If $[a-\ep,b+\ep]$ does not contain any critical value of $f$, then there is a strong deformation retraction mapping $f^b$ to $f^a$, \emph{i.e.} a continuous function $\Phi : [0,1]\times \rr^m \to \rr^m$ such that
	\[\left\lbrace\begin{array}{l}
	\Phi_0 = {\rm id}_{\rr^m},\\
	\Phi_1(f^b) \subset f^a,\\
	{\Phi_t}\big|_{f^a}={\rm id}_{f^a} \;\; \forall t \in [0,1]\\
	\Phi_t(f^c) \subset f^c \;\; \forall t \in [0,1], \forall c \in \rr
	\end{array}\right.
	\]
	satisfying the additional requirement $\Phi_t(f^c) = f^c$ for all  $t \in [0,1]$ and $c > b+\ep$.
\end{lem}

\begin{proof}\emph{First step.} We build a continuous function $\Psi : [0,1]\times \rr^m \to \rr^m$ such that
	\begin{equation}\label{defret}\left\lbrace\begin{array}{l}
	\Psi_0 = {\rm id}_{\rr^m},\\
	\Psi_1(f^b) \subset f^a,\\
	\Psi_t(f^c) \subset f^c \;\; \forall t \in [0,1],\; \forall c \in \rr\\
	\Psi_t(f^c) = f^c, \;\; \forall t \in [0,1],\; \forall c > b+\ep,\end{array}\right.\end{equation}
	without requiring that $\Psi_t$ is the identity on $f^a$ for all $t$.

	Let $X$ be the locally Lipschitz vector field defined for $x$ in $\rr^m$ by
	\begin{displaymath}
	X(x)=\left\lbrace \begin{array}{l}
	\nabla f(x) \textrm{ si }  \|\nabla f(x)\| \leq 1\\
	\frac{\nabla f(x)}{\|\nabla f(x)\|}\textrm{ si } \|\nabla f(x)\| > 1
	\end{array}\right.
	\end{displaymath}
	Let us take a $\cc^\infty$ function $\phi:\rr \to [0,1]$ satisfying $\phi=1$ on $(-\infty,b]$ and $\phi=0$ on $[b+\ep,\infty)$, and consider the following vector field:
	\begin{displaymath}
	Y(x)=\phi(f(x))X(x),
	\end{displaymath}
	defined such that $Y=X$ on $f^b$ and $Y(x)=0$ if $f(x)\geq b+\ep$.
	
	Let us denote by $\Psi_t(x)$ the flow associated with $-Y$ as follows: \begin{displaymath}
	\left\lbrace \begin{array}{l}
	\dd_t \Psi_t(x)=-Y(\Psi_t(x))\\
	\Psi_0(x)=x.
	\end{array}\right.
	\end{displaymath}
	As $\|Y\|$ is locally Lipschitz and bounded by the constant $1$, $\Psi$ is defined on $\rr^+ \times \rr^m$ and $\Psi_t$ is a homeomorphism of $\rr^m$ for all $t$. Let us check that $t \mapsto f(\Psi_t(x))$ is non-increasing: \begin{displaymath}
	\dd_t \left(f(\Psi_t(x))\right) = - \underbrace{\phi(f(\Psi_t(x)))}_{\geq 0}\underbrace{X(\Psi_t(x))\cdot \nabla f(\Psi_t(x)) }_{\geq \min\left\lbrace \|\nabla f(\Psi_t(x))\|^2,\|\nabla f(\Psi_t(x))\|\right\rbrace\geq 0} \leq 0.
	\end{displaymath}
	In particular, $\Psi_t(f^c) \subset f^c$ for all $t\geq 0$, and $c \in \rr$.
	
	Let us prove that $\Psi_t(f^c)=f^c$ for all $c > b+ \ep$. It is enough to prove that $f^c \subset \Psi_t(f^c)$ since the other inclusion is true for all $c$. Since $Y=0$ on $\rr^m \setminus f^{b+\ep}$, $\Psi_t{\big|_{\rr^m \setminus f^c}}={\rm id}_{\rr^m \setminus f^c}$ for all $c>b+\ep$. Then, if $x \in f^c$, there is a $y \in \rr^m$ such that $\Psi_t(y)=x$ (since $\Psi_t$ is onto), and $y$ cannot be in $\rr^m \setminus f^c$ since $x \in f^c$. Hence, $x$ belongs to $\Psi_t(f^c)$. 
	
	The aim is now to find a $T>0$ such that $\Psi_T(f^b) \subset f^a$.
	
	Let us prove that there is a real constant $M_0>0$ such that:\begin{displaymath}\|df(x)\|\geq M_0 \,\; \forall x \in f^b \setminus f^{a-\ep}.\end{displaymath} Suppose that $(x_n)$ is a sequence of $f^b \setminus f^{a-\ep}$ such that $df(x_n) \to 0$. Since $f=\zz+\ell$ with $\zz$ nondegenerate quadratic and $\ell$ Lipschitz, $(x_n)$ is hence bounded an admits a converging subsequence ; let $x$ be the limit. Since $f$ and $df$ are continuous, $df(x)=0$ and $f(x)$ belongs to $[a-\ep,b]$. As this is excluded, the existence of $M_0$ is proved.				
	
	Let $x$ be in $f^b$. If $t\geq 0$, $\Psi_t(x)$ is in $f^b$ too. Hence, we have $\|\nabla f(\Psi_t(x)\|\geq M_0$ as long as $f(\Psi_t(x))> a-\ep$, and the estimation of $\frac{d\,}{dt} f(\Psi_t(x))$ can be improved: \begin{displaymath}\begin{split}
	\dd_t \left(f(\Psi_t(x))\right) &\leq - \underbrace{\phi(f(\Psi_t(x)))}_{=1 \textrm{ since } \Psi_t(x) \in f^b}\min\left\{ \|\nabla f(\Psi_t(x))\|^2,\|\nabla f(\Psi_t(x))\|\right\} \\
	&\leq -\min\left\{ M_0^2,M_0 \right\} < 0.
	\end{split}		
	\end{displaymath}
	Let $K=\min (M_0,M_0^2) > 0$. As long as $f(\Psi_t(x))> a-\ep$, the previous calculation gives:\begin{displaymath}
	f(\Psi_t(x)) \leq \underbrace{f(\Psi_0(x))}_{=f(x)\leq b} - Kt \leq b-Kt.
	\end{displaymath}
	Let $T=\frac{b-a}{K}$. Assume that for all $t \in [0,T]$, $f(\Psi_t(x))> a$. The previous calculation shows that $f(\Psi_T(x)) \leq b-KT=a$, which is absurd. Hence there exists $t \in [0,T]$ such that $f(\Psi_t(x))\leq a$ and then since $t\mapsto f(\Psi_t(x))$ is non increasing, $\Psi_T(x) \subset f^a$.
	
	Up to a time rescaling sending $T$ to $1$ ($\tilde{\Psi}_t(x)=\Psi_{t/T}(x)$), we have just constructed a deformation retraction satisfying \eqref{defret}.\\		
	\emph{Second step.} Let us now build the strong deformation retraction.
	For all $x$ in $\rr^m$, let $\tau(x)$ be defined by
	\begin{displaymath}
	\tau(x) = \inf \left\{ t \in [0,1] | \Psi_t(x) \in f^a\right\}. 
	\end{displaymath}
	It is a continuous function on $\rr^m$. If $\Psi_t(x)$ stays out of $f^a$ for all $t$ in $[0,1]$ (this is the case for all $x$ in $\rr^m \setminus f^{b+\ep}$), then $\tau(x)$ is by convention equal to $1$. Since $t\mapsto f(\Psi_t(x))$ is non-increasing,
	if $\Psi_t(x)$ is not in $f^a$, $t\leq \tau(x)$.

	Let us define the mapping $\Phi$: \begin{displaymath}
	\begin{array}{r c c l}
	\Phi :& [0,1]\times \rr^m & \to & \rr^m \\
	&  (t,x)          & \mapsto &\Phi_t(x)=
	\Psi_{\min(t,\tau(x))}(x)
	\end{array}
	\end{displaymath}
	so that in particular $\Phi_0=\Psi_0$ and $\Phi_1(x)=\Psi_{\tau(x)}(x)$.
	The continuity of $\tau$ and $\Psi$ implies the continuity of $\Phi$. Let us check that $\Phi$ is as required in the Lemma:
	\begin{itemize}\renewcommand{\labelitemi}{$\bullet$}
		\item $\Phi_0=\Psi_0={\rm id}_{\rr^m}$,
		\item for all $x$ in $f^b$, $\Psi_1(x)$ is in $f^a$, and as a consequence $\Phi_1(x)=\Psi_{\tau(x)}(x)$ is in $f^a$,
		\item since $\tau=0$ on $f^a$, $\Phi_t=\Psi_0={\rm id}$ on $f^a$,
		\item for all $t$ in $[0,1]$, $\Phi_t(f^c)\subset \cup_{u\in[0,1]} \Psi_u(f^c) \subset f^c$.
		\item let us fix $t$ in $[0,1]$ and show that if $c>b+\ep$, $f^c \subset \Phi_t(f^c)$. Since $f^c \subset \Psi_t(f^c)$ for such a $c$, for all $x$ in $f^c$ there exists $y$ in $f^c$ such that $\Psi_t(y)=x$. If $x$ is not in $f^a$, $\tau(y)\geq t$, and hence $x=\Psi_t(y)=\Phi_t(y)$ is in $\Phi_t(f^c)$. If $x$ is in $f^a$, since $\Phi_t$ is the identity on $f^a$, $x=\Phi_t(x)$ is in $\Phi_t(f^a)\subset\Phi_t(f^c)$.
	\end{itemize}
\end{proof}

\subsection{Sending sublevel sets to sublevel sets}

\begin{lem}[Deformation of big sublevel sets of $\mathcal{Q}_m^\infty$ functions with Lipschitz difference]\label{deff}
	Let $\ell_0$ and $\ell_1$ be two $\cc^\infty$ Lipschitz functions, $\zz$ be a nondegenerate quadratic form on $\rr^m$, and define $f_t = \zz+\ell_0 + t(\ell_1-\ell_0)$ the homotopy between $f_0=\zz+\ell_0$ and $f_1=\zz+\ell_1$. Let $U$ be an open and bounded set of $\rr^m$ containing $\CC= \left\{ x \in \rr^m | \exists t \in [0,1], df_t(x)=0\right\}$. There exists a $\cc^\infty$-diffeomorphism $\Psi$ of $\rr^m$ such that:\begin{displaymath}
	\Psi(f_0^c)=f_1^c  \;\;\,
	\begin{array}{l}
	\forall c > \underset{t\in [0,1]}{\max}\underset{U} {\sup}f_t,\\
	\forall c < \underset{t\in [0,1]}{\min}\underset{U}{\inf} f_t.
	\end{array}
	\end{displaymath}
	Moreover, if $f_0 \geq f_1$ on $U$, $\Psi$ can be constructed so that $\Psi(f_0^a) \subset f_1^a$ for all $a\in \rr$.
\end{lem}

\begin{proof} Since $\CC$ is compact (see Proposition \ref{CCcompact}), there exists an open set $\Omega$ containing $\CC$ such that $\overline{\Omega} \cap U^c$ is empty ($\Omega$ is an open set which is "strictly included" in the open set $U$). Let $X_t$ be the vector field defined on $\rr^m\setminus \Omega$ by\begin{displaymath}
	X_t(x) = - \dd_t(f_t(x))\frac{\nabla f_t(x)}{\|\nabla f_t(x)\|^2} \,\;\textrm{for}\,\; t \in [0,1].
	\end{displaymath}
	
	\begin{lem}\label{rmpf} If $\gamma(t)$ is a trajectory for the vector field $X_t$, that is if $\gamma(t)$ stays in $\rr^m\setminus \Omega$ and $\dot{\gamma}(t)=X_t(\gamma(t))$, then $f_t(\gamma(t))$ does not depend on $t$.\end{lem} \begin{proof}
		This is proved by the following calculation:\\
		
		\hskip 4cm$	\dd_t (f_t(\gamma(t))) = \underbrace{\dot{\gamma}(t)\cdot \nabla f_t(\gamma(t))}_{=-\dd_t f_t(\gamma(t))} + \dd_t f_t(\gamma(t))=0.$\end{proof}

	Since $\bar{\Omega}$ and $\rr^m\setminus U$ are closed and disjoint, it is possible to find  $g:\rr^m \to [0,1]$ smooth such that $\left\{\begin{array}{l}g=0 \textrm{ on } \bar{\Omega} \\ g=1 \textrm{ on } \rr^m \setminus U \end{array}\right.$. 
	Let us define $Y_t(x) = g(x) X_t(x)$. The vector field $Y$ is well-defined, $\cc^\infty$ on $\rr^m$. It satisfies:\begin{displaymath}
	\left\{\begin{array}{l}
	Y_t=X_t \textrm{ on } \rr^m \setminus U \\
	Y_t=0 \textrm{ on } \Omega.
	\end{array}\right.
	\end{displaymath}
	\begin{lem} The vector field $Y$ is bounded.\end{lem} \begin{proof}
		If $m=\inf_{\|x\|=1} \|d\zz(x)\|$, we get that $\|\nabla f_t(x)\|\geq  m\|x\|-L$ for all $x$ in $\rr^d$. As a consequence, if $\|x\|\geq 2L/m$,
		\[\|Y_t(x)\|\leq \frac{|\dd_t f_t(x)|}{\|  \nabla f_t (x)\|}\leq \frac{L}{m(2L/m)-L}\leq 1.\]
		Now, define \[M=\sup_{\begin{array}{c}
			t\in[0,1]\\
			\|x\|\leq 2L/m
			\end{array}} \|Y_t(x)\|.\] 
		Then $Y$ is bounded by $\max(1,M)$ on $\rr^m$.
	\end{proof}

	The flow $\psi$ of $Y$ is hence defined on $\rr \times \rr^m$ ; it is the $\cc^\infty$ solution of the Cauchy problem:\begin{displaymath}
	\left\{\begin{array}{l} \dd_t \psi(t,x)= Y_t(\psi(t,x))\\
	\psi(0,x)=x.
	\end{array}\right.
	\end{displaymath}
	Let $\Psi$ be the $\cc^\infty$-diffeomorphism mapping $x$ to $\psi(1,x)$.\\
	
	Let us denote by $C_+$ (resp. $C_-$) the quantity $\underset{t\in [0,1]}{\max}\underset{U} {\sup}f_t$ (resp. $ \underset{t\in [0,1]}{\min}\underset{U}{\inf} f_t$), and prove that for $c \in \rr \setminus[C_-,C_+]$, $\Psi\left(\{f_0=c\}\right)= \{f_1=c\}$. 
	Take $x$ and $y$ such that $\Psi(x)=y$, and denote by $\gamma(t)$ the trajectory $t\mapsto \psi(t,x)$. Since $Y$ and $X$ coincide on $\rr^m \setminus U \subset \rr^m \setminus \Omega$, Lemma \ref{rmpf} states that as long as $\gamma(t)$ is in $\rr^m \setminus U$, $f_t(\gamma(t))$ is constant. By definition of 
	$C_+$ and $C_-$, $\{f_t=c\}$ is included in $\rr^m \setminus U$ for all $t$ in $[0,1]$, and as a consequence \[ \exists t \in [0,1],\; f_t(\gamma(t))=c \implies \forall t \in [0,1],\; f_t(\gamma(t))=c. \]
	This means that $f_0(x)=c$ if and only if $f_1(y)=c$, and since $\Psi$ is one-to-one we hence proved that $\Psi\left(\{f_0=c\}\right)= \{f_1=c\}$.	
	
	As a consequence, we obtain applying the previous work to a suitable union of levelsets that $\Psi\left(\{f_0\leq c\}\right) = \{f_1\leq c\}$ for $c < C_-$, and that $\Psi\left(\{f_0> c\}\right) = \{f_1> c\}$ for $c > C_+$. Since $\Psi$ is one-to-one, this implies $\Psi\left(\{f_0\leq c\}\right) = \{f_1\leq c\}$ for $c > C_+$.\\
	
	Finally, assume that $f_0 \geq f_1$ on $U$. Let us again estimate the evolution of $f_t(\gamma(t))$ for a trajectory $\dot{\gamma}(t)=Y_t(\gamma(t))$: \begin{displaymath}\begin{split}
	\dd_t f_t(\gamma(t)) &=\dot{\gamma}(t)\cdot \nabla f_t(\gamma(t)) + \dd_t f_t(\gamma(t))\\
	&= g(\gamma(t))(f_0-f_1)(\gamma(t)) +(f_1-f_0)(\gamma(t))\\
	& = (1-g(\gamma(t)) (f_1-f_0) (\gamma(t)) = \left\{\begin{array}{l}
	=0 \textrm{ if } \gamma(t) \in \rr^m\setminus U\\
	\leq 0 \textrm{ if } \gamma(t) \in U.
	\end{array}\right.
	\end{split}\end{displaymath}		
	since $g=1$ on $\rr^m \setminus U$, $1-g\geq 0$ and $f_1 \leq f_0$ on $U$.
	Now, for $a\in\rr$, let $x \in f_0^a$. Since $t \mapsto f_t(\psi(t,x))$ is non-increasing, $f_1(\Psi(x)) \leq f_0(x)\leq a$ and we have proved that $\Psi(f_0^a) \subset f_1^a$.
\end{proof}

	\section[Uniqueness of viscosity solution]{Uniqueness of viscosity solution: doubling variables}\label{C}
	Let us first recall a possible definition of viscosity solutions in a continuous setting:
	\begin{defi} A continuous function $u$ is a \emph{subsolution} of \eqref{HJ} on the set $(0,T)\times \rr^d$ if for each $\cc^\infty$ function $\phi:(0,T)\times \rr^d\to \rr$ such that $u-\phi$ admits a (strict) local maximum at a point $(t,q)\in(0,T)\times \rr^d$, \[\dd_t \phi(t,q)+H(t,q,\dd_q\phi(t,q))\leq 0.\]  		
		A continuous function $u$ is a \emph{supersolution} of \eqref{HJ} on the set $(0,T)\times \rr^d\to \rr$ if for each $\cc^\infty$ function $\phi:(0,T)\times \rr^d$ such that $u-\phi$ admits a (strict) local minimum at a point $(t,q)\in(0,T)\times \rr^d$, \[\dd_t \phi(t,q)+H(t,q,\dd_q\phi(t,q))\geq 0.\]  		
		A viscosity solution is both a sub- and supersolution of \eqref{HJ}.
	\end{defi} 	
	
	The following proposition justifies the name of \emph{viscosity operator} for an operator satisfying Hypotheses \ref{L-O}.
	\begin{prop}\label{key} Let $H$ be a $\cc^2$ Hamiltonian with uniformly bounded second spatial derivative and $V^t_s : \cc^{0,1} (\rr^d,\rr) \to \cc^{0,1} (\rr^d,\rr)$ be a viscosity operator defined for each $0\leq s \leq t$. Then for each Lipschitz function $u_0~:~\rr^d ~\to ~\rr$,\begin{displaymath}
		u(t,q) = V^t_0 u_0(q)
		\end{displaymath}
		solves the Hamilton-Jacobi equation in the viscosity sense on $(0,\infty)\times \rr^d$.\end{prop}
	
	This characterization and its proof may be found in \cite{bernard} (Proposition 20). A similar axiomatic description of the viscosity solutions was initially proposed in \cite{aglm} (Theorem 2) for multiscale analysis, see also \cite{fleming} (Theorem 5.1).

	The uniqueness of the viscosity operator for $H$ satisfying Hypothesis \ref{esti} is a consequence of  a stronger uniqueness result for unbounded solutions stated by H. Ishii in \cite{ishii} (Theorem 2.1 with Remark 2.2), see also \cite{lions87}. It is also a consequence of the following finite speed of propagation argument proposed by G. Barles in \cite{barles} (Theorem 5.3). We write the proof here for the sake of completeness, adopting his arguments and notations, and using only the second estimate of Hypothesis \ref{esti}.
	\begin{prop}[Finite speed of propagation]\label{fsp} If $H$ satisfies $\|\dd_{q,p} H\|\leq C(1+\|p\|)$ for some $C>0$, and
		$u$ and $v$ are respectively sub- and supersolutions of \eqref{HJ} on $[0,T]\times \rr^d$ which are $L$-Lipschitz uniformly in time with respect to the space variable, then:
		\[u(0,\cdot)\leq v(0,\cdot) \textrm{ on } B(0,R) \Longrightarrow  
		u \leq v \text{ on } [0,T]\times B(0,R-C(1+2L)T)
		\]		
		as long as $R$ is strictly larger than $C(1+2L)T$.
	\end{prop}
	
	\begin{csq}\label{viscspeed}
		If $u$ and $v$ are two viscosity solutions of \eqref{HJ} which are $L$-Lipschitz with respect to $q$ on $[0,T]\times \rr^d$, then for each $t$ in $[0,T]$:
		\[|u(t,q)-v(t,q)|\leq \|u(0,\cdot)-v(0,\cdot)\|_{\bar{B}\left(q,C(1+2L)t\right)}\] 
	\end{csq}
	
	\begin{proof} 
		We apply Proposition \ref{fsp} with $R=C(1+2L)t+\delta$ to the subsolution $u$ and the supersolution $v+\|u(0,\cdot)-v(0,\cdot)\|_{\bar{B}(q,R)}$, use the symmetry and let $\delta$ tend to $0$.
	\end{proof}
	
	\begin{csq}\label{viscuniq} If $u$ and $v$ are both viscosity solutions on $[0,T]\times \rr^d$ that satisfy $u(0,\cdot)=v(0,\cdot)$ on $\rr^d$ and are Lipschitz uniformly in time, they coincide on $[0,T]\times \rr^d$. In particular, there exists at most one viscosity operator.
	\end{csq}
	
	\begin{lem}\label{extvisc}
		If $u$ is a continuous function of $(0,T]\times \rr^d$ and also a subsolution of \eqref{HJ} on $(0,T)\times \rr^d$, then it is a subsolution on $(0,T]\times \rr^d$, meaning that if $u-\phi$ attains a strict maximum on $(0,T]\times \rr^d$ at some $(T,q_0)$, the derivatives of $\phi$ satisfy the required inequality. 
	\end{lem}
	
	\begin{proof}
		Take $\phi$ $\cc^\infty$ on $(0,T]\times \rr^d$ such that $u-\phi$ attains its strict maximum at some $(T,q_0)$. Let us consider the functions $(t,q)\mapsto u(t,q)-\phi(t,q)-\frac{\eta}{T-t}$ for small $\eta>0$. Since $u-\phi$ attains a strict maximum at $(T,q_0)$, there exists a sequence $(t_\eta,q_\eta)$ in $(0,T)\times \rr^d$ of local maximal points of $u-\phi-\frac{\eta}{T-t}$ such that $(t_\eta,q_\eta)$ tends to $(T,q_0)$ when $\eta\to0$.
		
		Since $u$ is a subsolution on $(0,T)\times \rr^d$, this implies that:\[
		\dd_t \left(\phi(t,q)+\frac{\eta}{T-t}\right)+H\left(t_\eta,q_\eta,\dd_q\left(\phi(t,q)+\frac{\eta}{T-t}\right)\right)\leq 0\]
		hence\[
		\dd_t \phi(t_\eta,q_\eta)+\frac{\eta}{(T-t_\eta)^2}+H\left(t_\eta,q_\eta,\dd_q\phi(t_\eta,q_\eta)\right)\leq 0.\]
		The positive term $\frac{\eta}{(T-t_\eta)^2}$ may be dropped, and then the continuity of $\phi$ gives that:
		\[
		\dd_t \phi(T,q_0)+H\left(T,q_0,\dd_q\phi(T,q_0)\right)\leq 0.\]
	\end{proof}
	
	\begin{lem}\label{lemvisc}
		If the assumptions of Proposition \ref{fsp} are satisfied, the function $w=u-v$ is a subsolution on $(0,T]\times \rr^d$ of\[\dd_t w - C(1+2L) \|\dd_q w\|=0.\]
	\end{lem}
	
	\begin{proof}
		Let us assume that $\phi$ is a $\cc^\infty$ function such that $w-\phi$ attains a strict local maximum at a point $(t_0,q_0)$ in $(0,T)\times\rr^d$.	The aim is to show that 
		\[\dd_t \phi(t_0,q_0) \leq C(1+2L) \|\dd_q \phi(t_0,q_0)\|.\]
		Here is where the variables are doubled: let us define the function 
		\[\Psi_{\ep,\alpha}:(t,q,s,p) \mapsto u(t,q)-v(s,p)-\frac{\|q-p\|^2}{\ep^2} - \frac{|t-s|^2}{\alpha^2} - \phi(t,q).\] In particular $\Psi_{\ep,\alpha}(t_0,q_0,t_0,q_0)=w(t_0,q_0)-\phi(t_0,q_0)$ is the local maximum of $w-\phi$ for all $\ep>0$ and $\alpha>0$.
		
		Take $r>0$ such that the maximum of $w-\phi$ on $\bar{B}((t_0,q_0),r)$ is attained only at $(t_0,q_0)$. Then $\Psi_{\ep,\alpha}$ attains a maximum on the compact set $\bar{B}((t_0,q_0),r)\times \bar{B}((t_0,q_0),r)$, and we denote by $(\bar{t},\bar{q},\bar{s},\bar{p})$ a point reaching this maximum, without forgetting that these quantities depend on $\ep$ and $\alpha$. 
		\begin{lem}\label{doubling} The point $(\bar{t},\bar{q},\bar{s},\bar{p})$ satisfies:
			\begin{enumerate} 
				\item  $(\bar{t},\bar{q}),(\bar{s},\bar{p}) \to (t_0,q_0)$ when $\ep,\alpha \to 0$,
				\item  $\frac{\|\bar{q}-\bar{p}\|}{\ep^2}\leq L$.
			\end{enumerate}
		\end{lem}
		\begin{proof}
			\begin{enumerate}
				\item Since $(\bar{t},\bar{q},\bar{s},\bar{p})$ belongs to the compact set  $\bar{B}((t_0,q_0),r)\times \bar{B}((t_0,q_0),r)$, accumulation points $(t,q,s,p)$ exist when $\ep$ and $\alpha$ tend to zero. 
				These accumulation points must satisfy $(t,q)=(s,p)$: else, the value of $\Psi_{\ep,\alpha}(\bar{t},\bar{q},\bar{s},\bar{p})$ explodes towards $-\infty$ while it is supposed to remain larger than $\Psi_{\ep,\alpha}(t_0,q_0,t_0,q_0)$ which is the maximum of $w-\phi$ and does not therefore depend on $\ep$ and $\alpha$. 
				
				Now, let us denote by $(t,q)\in \bar{B}((t_0,q_0),r)$ an accumulation point of both $(\bar{t},\bar{q})$ and $(\bar{s},\bar{p})$. Since $\Psi_{\ep,\alpha}(\bar{t},\bar{q},\bar{s},\bar{p})\geq \Psi_{\ep,\alpha}(t_0,q_0,t_0,q_0)=w(t_0,q_0)-\phi(t_0,q_0)$, we also have using the sign of $-\frac{\|\bar{q}-\bar{p}\|^2}{\ep^2} - \frac{|\bar{t}-\bar{s}|^2}{\alpha^2}$ that
				\[u(\bar{t},\bar{q})-v(\bar{s},\bar{p}) - \phi(\bar{t},\bar{q}) \geq w(t_0,q_0)-\phi(t_0,q_0).\]
				Hence if $\ep$ and $\alpha$ tend to zero,
				\[w(t,q)-\phi(t,q)\geq w(t_0,q_0)-\phi(t_0,q_0),\]
				and the fact that $(t_0,q_0)$ is the only point of $\bar{B}((t_0,q_0),r)$ where the maximum is attained concludes the proof.
				
				\item Since $(\bar{t},\bar{q},\bar{s},\bar{q})$ is in the set $\bar{B}((t_0,q_0),r)\times \bar{B}((t_0,q_0),r)$, \[\Psi_{\ep,\alpha}(\bar{t},\bar{q},\bar{s},\bar{q}) \leq \Psi_{\ep,\alpha}(\bar{t},\bar{q},\bar{s},\bar{p})\]
				hence 
				\[u(\bar{t},\bar{q})-v(\bar{s},\bar{q}) - \frac{|\bar{t}-\bar{s}|^2}{\alpha^2} - \phi(\bar{t},\bar{q}) \leq u(\bar{t},\bar{q})-v(\bar{s},\bar{p})-\frac{\|\bar{q}-\bar{p}\|^2}{\ep^2} - \frac{|\bar{t}-\bar{s}|^2}{\alpha^2} - \phi(\bar{t},\bar{q}) \]
				and since $v$ is $L$-Lipschitz,
				\[\frac{\|\bar{q}-\bar{p}\|^2}{\ep^2} \leq v(\bar{s},\bar{q}) -v(\bar{s},\bar{p}) \leq L \|\bar{q}-\bar{p}\|. \]
			\end{enumerate}
		\end{proof}
		
		Now, since $(\bar{t},\bar{q},\bar{s},\bar{p})$ converge to $(t_0,q_0,t_0,q_0)$, it is in $B((t_0,q_0),r)\times B((t_0,q_0),r)$ for $\ep$ and $\alpha$ small enough, and the fact that it maximizes $\Psi_{\ep,\alpha}$ tells us that:
		
		\begin{itemize}
			\item $(\bar{t},\bar{q})$ is a maximum point of \[(t,q) \mapsto u(t,q) - \underbrace{\left(\phi(t,q) + v(\bar{s},\bar{p})+\frac{\|q-\bar{p}\|^2}{\ep^2} + \frac{|t-\bar{s}|^2}{\alpha^2}\right)}_{=\phi_1(t,q)},\]		
			and since $u$ is a subsolution, the derivatives of $\phi_1$ satisfy
			\[\dd_t \phi_1(\bar{t},\bar{q})+H(\bar{t},\bar{q},\dd_q\phi_1(\bar{t},\bar{q}))\leq 0,    \]
			hence
			\[\dd_t \phi(\bar{t},\bar{q})+ 2\cdot \frac{\bar{t}-\bar{s}}{\ep^2}+H\left(\bar{t},\bar{q},\dd_q\phi(\bar{t},\bar{q})+ 2 \cdot\frac{\bar{q}-\bar{p}}{\ep^2} \right)\leq 0.\]
			Note also that since $u$ is $L$-Lipschitz with respect to $q$, the $q$-derivative of $\phi_1$ at a point of maximum of $u-\phi$ is necessarily bounded by $L$, hence:\begin{equation}\label{majdiff}
			\|\dd_q\phi(\bar{t},\bar{q})+ 2 \cdot\frac{\bar{q}-\bar{p}}{\ep^2}\|\leq L.
			\end{equation}
			
			\item $(\bar{s},\bar{p})$ is a minimum point of \[(s,p) \mapsto v(s,p) - \underbrace{\left(u(\bar{t},\bar{q})-\phi(\bar{t},\bar{q}) -\frac{\|\bar{q}-p\|^2}{\ep^2} - \frac{|\bar{t}-s|^2}{\alpha^2}\right)}_{=\phi_2(s,p)},\]		
			and since $v$ is a supersolution, the derivatives of $\phi_2$ satisfy
			\[\dd_s \phi_1(\bar{s},\bar{p})+H(\bar{s},\bar{p},\dd_p\phi_1(\bar{s},\bar{p}))\leq 0,    \]
			hence
			\[ 2\cdot \frac{\bar{t}-\bar{s}}{\ep^2}+H\left(\bar{s},\bar{p},2\cdot \frac{\bar{q}-\bar{p}}{\ep^2} \right)\geq 0.\]
		\end{itemize}
		Combining the two previous points gives that\[\begin{split}
		\dd_t \phi(\bar{t},\bar{q}) \leq &\, H\left(\bar{s},\bar{p},2\cdot \frac{\bar{q}-\bar{p}}{\ep^2} \right) - H\left(\bar{t},\bar{q},\dd_q\phi(\bar{t},\bar{q})+ 2 \cdot\frac{\bar{q}-\bar{p}}{\ep^2} \right) \\ 
		\leq & \, H\left(\bar{s},\bar{p}, 2\cdot \frac{\bar{q}-\bar{p}}{\ep^2} \right)-H\left(\bar{t},\bar{q},2\cdot \frac{\bar{q}-\bar{p}}{\ep^2} \right)\\& +\underbrace{H\left(\bar{t},\bar{q},2\cdot \frac{\bar{q}-\bar{p}}{\ep^2} \right) - H\left(\bar{t},\bar{q},\dd_q\phi(\bar{t},\bar{q})+ 2 \cdot\frac{\bar{q}-\bar{p}}{\ep^2} \right)}.\\
		& \;\;\;\;\;\;\,\,\,\,\,\,\,\,\,\,\,\,\,\,\,\,\,\,\,\,\,\,\,\,\,\,\,\,\,\,\,\,\,\,\,\,\,\,\, \leq C(1+2L)\|\dd_q\phi(\bar{t},\bar{q})\|
		\end{split}\]
		Let us explain the last point: the estimate \eqref{majdiff} and the second result of Lemma \ref{doubling} state that both $\dd_q\phi(\bar{t},\bar{q})+ 2 \cdot\frac{\bar{q}-\bar{p}}{\ep^2}$ and $ 2 \cdot\frac{\bar{q}-\bar{p}}{\ep^2}$ are bounded by $2L$. The assumption made on $\|\dd_{p,q}H\|$ implies that $\dd_p H$ is bounded by $C(1+2L)$ on the set $[0,T]\times \rr^d \times \bar{B}(0,2L)$, and hence
		\[\left|H\left(\bar{t},\bar{q},2\cdot \frac{\bar{q}-\bar{p}}{\ep^2} \right) - H\left(\bar{t},\bar{q},\dd_q\phi(\bar{t},\bar{q})+ 2 \cdot\frac{\bar{q}-\bar{p}}{\ep^2} \right)  \right|\leq C(1+2L)\|\dd_q\phi(\bar{t},\bar{q})\|.\]

		Lemma \ref{doubling} implies that the quantity $H\left(\bar{s},\bar{p}, 2\cdot \frac{\bar{q}-\bar{p}}{\ep^2} \right)-H\left(\bar{t},\bar{q},2\cdot \frac{\bar{q}-\bar{p}}{\ep^2} \right)$ tends to $0$ when $\ep$ and $\alpha$ tend to $0$. To finish, since $(\bar{t},\bar{q})$ tends to $(t_0,q_0)$:
		
		\[\dd_t \phi(t_0,q_0) \leq C(1+2L) \|\dd_q \phi(t_0,q_0)\|.\]
		
		We then extend the subsolution property to $\{T\}\times \rr^d$ with Lemma \ref{extvisc}.
	\end{proof}
	
	\begin{proof}[Proof of Proposition \ref{fsp}.]
		Take $R>C(1+2L)T$ and let us denote by $M$ the maximum of $w$ on the set $[0,T]\times \bar{B}(0,R)$.		
		We are going to prove that for all $\delta>0$ such that $R>\delta + C(1+2L)T$, $w(t,q)\leq \delta t$ on the set $[0,T]\times B(0,R-C(1+2L)T-\delta)$, using a comparison with an ad hoc smooth solution of $\dd_t w - C(1+2L)\|\dd_q w\| =0$.
		
		For such a $\delta>0$, it is possible to find a smooth and increasing function $\chi_\delta : \rr \to \rr$ such that $\chi_\delta(r) =0$ if $r\leq R-\delta$ and $\chi_\delta(r)=M$ if $r \geq R$. Then \[\phi_\delta:(t,q)\mapsto \chi_\delta(\|q\|+C(1+2L)t)\] is a smooth solution of $\dd_t w - C(1+2L)\|\dd_q w\| =0$ on $[0,T]\times \bar{B}(0,R)$. Let us then show that the function $(t,q)\mapsto w(t,q)-\phi_\delta(t,q) - \delta t$ on $[0,T]\times \bar{B}(0,R)$ is non positive. 
		
		The maximum of this function cannot be attained at a point $(t,q)$ of $(0,T]\times B(0,R)$, or else the fact that $w$ is a subsolution on $(0,T]\times B(0,R)$ (Lemma \ref{lemvisc}) gives that:
		\[\dd_t \phi_\delta(t,q)+\delta-C(1+2L)\|\dd_q \phi_\delta(t,q)\|\leq 0.\]
		Since $\phi_\delta$ solves the equation in the classical way and $\delta$ is positive, this is impossible.
		
		So, either the maximum is attained at a point $(0,q)$, or at a point $(t,q)$ with $\|q\|=R$.
		
		In the first case, the maximum is of the form $w(0,q)-\phi_\delta(\|q\|)$ and is hence non positive since $u\leq v$ on $\{0\}\times \rr^d$ and $\phi_\delta$ is non negative.
		
		In the second case, $\phi_\delta(t,q)=M$ and the maximum is of the form $w(t,q)-M-\delta t$. Since $w$ is smaller than $M$ on $[0,T]\times \bar{B}(0,R)$, the maximum is non positive.
		
		Hence, for each $(t,q)$ in $[0,T]\times \bar{B}(0,R)$,
		\[w(t,q)\leq \phi_\delta(t,q)+\delta t.\]
		Since $\phi_\delta(t,q)$ is zero on $[0,T]\times B(0,R-C(1+2L)T-\delta)$, on this set we have:
		\[w(t,q)\leq \delta t.\]
		Letting $\delta$ tend to zero gives that $w=u-v \leq 0$ on $[0,T]\times B(0,R-C(1+2L)T)$.
		
	\end{proof}

\section{Graph selector}\label{graphselector}
In this appendix we present the graph selector notion in the usual symplectic framework and its application to the variational resolution of the evolutive Hamilton-Jacobi equation. The graph selector can also be used to address other dynamical questions, see \cite{paternain}, \cite{MCarnaud} and \cite{santos}.

\subsection{Graph selector}

Let us settle in a usual symplectic framework: we assume that $M$ is a closed Riemannian $d$-manifold and look at its cotangent bundle $\pi:T^\star M \to M$. If $q=(q_1,\cdots,q_d)$ are the coordinates of a chart on $M$, the dual coordinates $p=(p_1,\cdots,p_d) \in T^\star_q M$ are defined by $p_i(e_j)=\delta_{ij}$, where $e_j$ is the $j^{th}$ vector of the canonical basis and $\delta_{i,j}$ is the Kronecker symbol. The manifold $T^\star M$ is endowed with the Liouville $1$-form $\lambda$, which writes $\lambda=pdq$ in this dual chart. The symplectic structure on $T^\star M$ is given by the symplectic form $\omega = d\lambda= dp \wedge dq$ in the dual chart.

A submanifold $\mathcal{L}$ of $T^\star M$ is called \emph{Lagrangian} if it is $d$-dimensional and if $i_\LL^\star w =0$, where $i_\LL : \LL \to T^\star M$ is the inclusion. It is \emph{exact} if $i_\LL^\star \lambda$ is exact, \emph{i.e.} if there exists a smooth function $S: \LL \to \rr$ such that $dS = i_\LL^\star \lambda$. Such a function is called a primitive of $\LL$, and is uniquely determined up to the addition of a constant. If $\LL$ is an exact Lagrangian submanifold, we call \emph{wavefront} for $\LL$ a set of the form $\WW=\{(\pi(x),S(x)), x \in \LL\}$ for $S$ a primitive of $\LL$, see Figure \ref{geomsolwf}. 
\begin{figure}[h!]
	\begin{center}
		\def\svgwidth{.7\columnwidth} 
		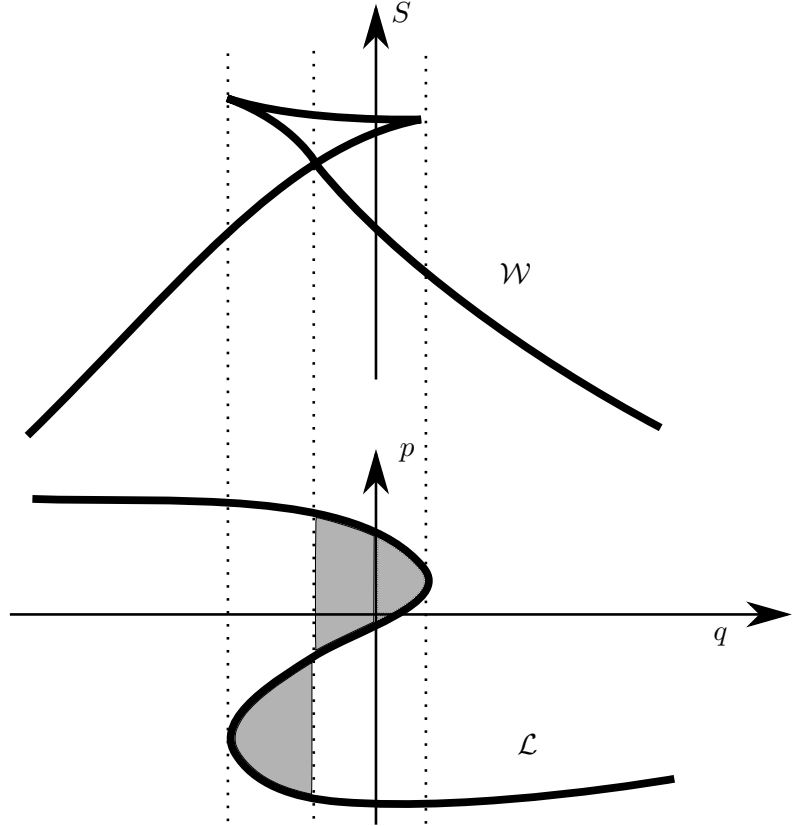 
	\end{center}
	\caption{A Lagrangian submanifold and an associate wavefront. The two greyed domains delimited by the position of the intersection in the wavefront have the same area.}
	\label{geomsolwf}
\end{figure}

If $\mathcal{L}$ is an exact Lagrangian submanifold and $\mathcal{W}$ is a wavefront for $\mathcal{L}$, we call \emph{graph selector} a Lipschitz\footnote{a continuous function with graph is included in $\WW$ is automatically Lipschitz if $\mathcal{L}$ is uniformly bounded in the fiber variable.} function $u$ whose graph is included in $\WW$. 
Since a possible primitive $S$ of the Lagrangian submanifold is given by an underlying action, the existence of a graph selector can be deduced under reasonable hypotheses from the existence of \emph{action selectors}. These action selectors are obtained by using either generating family techniques (see \cite{chap2}), via Floer homology (see \cite{floer} and \cite{oh}) or lately by microlocal sheaf techniques (see \cite{guillermou}).  In \cite{MO2}, the link between  the invariants constructed with generating families and via the Floer homology is studied, which leads to the conclusion that they give the same graph selector under a suitable normalization (see also \cite{MVZ}).

A graph selector provides simultaneously a continuous section of the wavefront and a discontinuous section of the Lagrangian submanifold:
\begin{prop}[Graph selector] \label{graphsel} Let $\mathcal{L}$ be an exact Lagrangian submanifold of $T^\star M$ such that $\pi_{|\mathcal{L}}$ is proper, $\mathcal{W}$ be a wavefront for $\mathcal{L}$, and $u$ be a {graph selector}. Then $(q,du(q)) \in \mathcal{L}$ for almost every $q$.
\end{prop}
The author was unable to locate the proof of this statement in the literature, yet it is close to Proposition $2.4$ in \cite{paternain} and to Proposition II in \cite{ottolenghi}, which both deal with the graph selector in terms of generating family. We present a proof improved by J.-C. Sikorav.

\begin{proof} 	Let $S: \mathcal{L} \to \rr$ be a primitive of $\mathcal{L}$ and $u$ be a graph selector of the associated wavefront. If $x$ is in $\mathcal{L}$, we will denote by $p_x \in T^\star_{\pi(x)} \LL$ the second coordinate of $x=(\pi(x),p_x)$.
	
	We are going to prove that if $q\in M$ is a regular value of $\pi_{|\mathcal{L}}$ and a point of differentiability of $u$, $(q,du(q))$ is in $\LL$. Then combining Rademacher's theorem (on $u$) and Sard's theorem (on $\pi_{|\mathcal{L}}$)  imply that the statement holds for almost every $q$. 
	
	Let us fix such a point $q$. We denote by $\LL_q$ the fiber $\pi_{|\mathcal{L}}^{-1}(\{q\})$, which is finite set since $q$ is a regular value of the proper map $\pi_{|\mathcal{L}}$.	
	We are going to prove that for all $v$ in $\SSS^{d-1}$, there exists $x=(q,p) \in \LL_q$ such that $du(q).v=p.v$. 
	
	Let $v\in \SSS^{d-1}$. We work in a local chart in the neighbourhood of $q\in M$: take a sequence $q_n$ such that $\lim_{n\to \infty} \frac{q_n-q}{\|q_n-q\|} =v$. For all $n$, there exists $x_n$ in $\mathcal{L}_{q_n}$ such that $u(q_n)=S(x_n)$. Since $\pi_{|\mathcal{L}}$ is proper, we may assume without loss of generality that $x_n$ admits a limit $x$ in $\LL$. We again work in the local chart to write $x_n=x + x_n-x$, where $x_n-x$ is a sequence of $T_x \LL$ converging to zero. We have on one hand
	\[u(q_n)-u(q)= du(q)(q_n-q) + o(\|q_n-q\|) = \|q_n-q\|du(q)v + o(\|q_n-q\|)\]
	and on the other hand
	\[u(q_n)-u(q)= S(x_n)-S(x)=dS(x)(x_n-x) + o(\|x_n-x\|) = p_x d\pi(x)(x_n-x) + o(\|x_n-x\|).     \]
	Now, since $\pi(x_n)=q_n$ for each $n$, we have since $d\pi_{|\mathcal{L}}(x)$ is invertible
	\[d\pi(x)(x_n-x)= q_n- q + o(\|q_n-q\|) =  \|q_n-q\|v + o(\|q_n-q\|). \]
	Putting these three equations together we get
	\[ \|q_n-q\|du(q)v = \|q_n-q\|p_x v + o(\|q_n-q\|), \]
	and dividing by $\|q_n-q\|$ and letting $n$ tend to $+\infty$ gives that $du(q).v=p_x.v$. 
	
	Now we define $E_x = \{v \in \SSS^{d-1} | \, du(q)v= p_xv \}$. The previous result implies that $\{E_x\}_{x\in \LL_q}$ is a finite cover of $\SSS^{d-1}$, hence $\{\textrm{Vect}(E_x)\}_{x\in \LL_q}$ is a finite cover of	$\rr^d$ made of vector subspaces: one of them is hence the whole space $\rr^d$, and the corresponding $x\in \LL_q$ hence satisfies $du(q)=p_x$. 
\end{proof}

\subsection{Application to the evolutive Hamilton-Jacobi equation}

We follow \cite{viterboX} to explicit the link between the variational operator and the graph selector introduced in the previous paragraph for a $\cc^2$ initial condition $u_0$.
We define the autonomous suspension of $H$ by $K(t,s,q,p)=s+ H(t,q,p)$ on the cotangent $T^\star(\rr \times\rr^{d})$, identified with $T^\star \rr \times T^\star \rr^d$, and denote by $\Phi$ its Hamiltonian flow. 
The Hamiltonian system for $K$ writes
\[\left\{\begin{array}{ll}
\dot{t}=1,&\dot{q}= \dd_p H(t,q,p),\\
\dot{s}=-\dd_t H(t,q,p),&\dot{p}=-\dd_q H(t,q,p),
\end{array}\right.\]
hence $t$ can be taken as the time variable. 

The submanifold ${\Gamma}_0=\{(0,-H(0,q_0,du_0(q_0)),q_0,du_0(q_0)),q_0 \in \rr^d\}$ is contained in the level set $K^{-1}(\{0\})$, and since $K$ is autonomous, it is constant along its trajectories, and as a consequence ${\Phi}^t({\Gamma}_0)=\left\{(t,-H(t,\phi^t_0 (q_0,du_0(q_0))),\phi^t_0(q_0,du_0(q_0))),q_0 \in \rr^d\right\}$. We call \emph{suspended geometric solution} of the Cauchy problem the Lagrangian submanifold ${\mathcal{L}}=\cup_{t\in\rr} {\Phi}^t ({\Gamma}_0)$, and the following set is a wavefront for $\mathcal{L}$:
\[\mathcal{W}=\left\{\left(t,q,u_0(q_0)+ \aaa^t_0(\phi^\tau_0(q_0,du_0(q_0)))\right) \left|\begin{array}{c}
t\in \rr, q\in \rr^d,q_0 \in \rr^d,\\
Q^t_0(q_0,du_0(q_0))=q.
\end{array} \right\}\right.\]
\begin{proof}[Proof of Proposition \ref{solpp}]The axioms required to be a variational operator implies that the function $u:(t,q)\mapsto R^t_0 u_0(q)$ is a graph selector for $\mathcal{L}$: it is Lipschitz, and the variational property asks that its graph is contained in $\mathcal{W}$. Also,
Proposition \ref{graphsel} states that for almost every $(t,q)$, $(t,\dd_t u(t,q), q, \dd_q u(t,q))$ belongs to ${\mathcal{L}}\subset  K^{-1}(\{0\})$, which proves Proposition \ref{solpp}. \end{proof}

\end{appendices}
	\bibliographystyle{alpha}
	\bibliography{biblio} 

\end{document}